\documentclass[reqno,dvipdfmx]{amsart}

\usepackage{amsmath,amsthm,amssymb,graphicx,cite,bm,mathrsfs,latexsym,comment,url}

\usepackage{color}
\usepackage[normalem]{ulem}

\theoremstyle{plain}
\newtheorem{thm}{Theorem}[section]
\newtheorem{lem}[thm]{Lemma}

\newtheorem{dfn}[thm]{Definition}
\newtheorem{prop}[thm]{Proposition}
\newtheorem{rmk}[thm]{Remark}

\renewcommand{\labelenumi}{(\theenumi)}

\def\D{\mathrm{D}}
\def\H{\mathrm{H}}

\def\T{\mathrm{T}}

\def\d{\mathrm{d}}

\def\h{\mathrm{h}}

\def\Cset{\mathbb{C}}
\def\Nset{\mathbb{N}}
\def\Rset{\mathbb{R}}

\def\Zset{\mathbb{Z}}

\def\epsilon{\varepsilon}

\def\e{\mathrm{e}}
\def\i{\mathrm{i}}
\DeclareMathOperator{\trace}{tr}
\DeclareMathOperator{\sech}{sech}

\DeclareMathOperator{\gKer}{gKer}
\DeclareMathOperator{\Ker}{Ker}

\DeclareMathOperator{\lspan}{span}

\renewcommand{\Re}{\mathrm{Re}\,}
\renewcommand{\Im}{\mathrm{Im}\,}

\newcommand\defeq{\mathrel{\rlap{\raisebox{0.3ex}{$\cdot$}}\raisebox{-0.3ex}{$\cdot$}}=}

\newcommand*\pFqskip{8mu}
\catcode`,\active
\newcommand*\pFq{\begingroup
        \catcode`\,\active
        \def ,{\mskip\pFqskip\relax}%
        \dopFq
}
\catcode`\,12
\def\dopFq#1#2#3#4#5{%
        {}_{#1}F_{#2}\biggl(\genfrac..{0pt}{}{#3}{#4};#5\biggr)%
        \endgroup
}

\makeatletter
 \@addtoreset{equation}{section}
\makeatother
\def\theequation{\arabic{section}.\arabic{equation}}

\begin{document}


\title[Bifurcations and Spectral Stability of Solitary Waves]{Bifurcations and Spectral Stability of Solitary Waves in Coupled Nonlinear Schr\"odinger Equations}
\thanks{This work was partially supported by JSPS KAKENHI Grant Number JP17H02859.}

\author{Kazuyuki Yagasaki}

\author{Shotaro Yamazoe}

\address{Department of Applied Mathematics and Physics, Graduate School of Informatics,
Kyoto University, Yoshida-Honmachi, Sakyo-ku, Kyoto 606-8501, JAPAN}
\email{yagasaki@amp.i.kyoto-u.ac.jp (K.\ Yagasaki)}
\email{yamazoe@amp.i.kyoto-u.ac.jp (S.\ Yamazoe)}

\date{\today}

\begin{abstract}
We study bifurcations and spectral stability of solitary waves in coupled nonlinear Schr\"odinger (CNLS) equations on the line.
We assume that the coupled equations possess a solution of which one component is identically zero, and call it a \emph{fundamental solitary wave}.
By using a result of one of the authors and his collaborator, pitchfork bifurcations of the fundamental solitary wave are detected.
We utilize the Hamiltonian-Krein index theory and Evans function technique to determine the spectral or orbital stability of the bifurcated solitary waves as well as that of the fundamental one under some nondegenerate conditions which are easy to verify, compared with those of the previous results.
We apply our theory to CNLS equations with a cubic nonlinearity and give numerical evidences for the theoretical results.
\end{abstract}
\maketitle


\section{Introduction}

We consider coupled nonlinear Schr\"odinger (CNLS) equations of the form
\begin{align}
	\begin{aligned}
		&\mathrm{i}\partial_t u+\partial_x^2u+\partial_1F(|u|^2,|v|^2;\mu)u=0,\\
		&\mathrm{i}\partial_t v+\partial_x^2v+\partial_2F(|u|^2,|v|^2;\mu)v=0,
	\end{aligned}
	\label{eqn:CNLS}
\end{align}
where $(u,v)=(u(t,x),v(t,x))$ are complex-valued unknown functions of $(t,x)\in\Rset\times\Rset$, $\mu$ is a real parameter, $F:\Rset\times\Rset\times\Rset\to\Rset$ is a given function, and $\partial_jF$ represents the derivative of $F$ with respect to the $j$-th argument for $j=1,2$.
Our special attention is paid to the cubic nonlinearity,
\begin{align}
	F(\zeta_1,\zeta_2;\beta_1)=\frac{1}{2}\zeta_1^2+\beta_1\zeta_1\zeta_2+\frac{\beta_2}{2}\zeta_2^2,
	\label{eqn:cubic}
\end{align}
where $\beta_1,\beta_2$ are real constants and $\beta_1$ is taken as the control parameter $\mu$ below.
In applications, \eqref{eqn:CNLS} arises when considering propagation of two interacted waves in nonlinear media.
For example, in nonlinear optics, it describes nonlinear pulse propagation in birefringent fibers.
See \cite{Ag95,NgWa16,Ro76,Ya10} and references therein.
\eqref{eqn:CNLS} has a Hamiltonian structure with the Hamiltonian
\begin{align*}
	\mathcal H(u,v;\mu)=\frac{1}{2}\int_\Rset\Bigl(|\partial_x u(x)|^2+|\partial_x v(x)|^2-F(|u|^2,|v|^2;\mu)\Bigr)\d x,
\end{align*}
and the Cauchy problem of \eqref{eqn:CNLS} is well-posed in the Sobolev space $H^1(\Rset,\Cset^2)$ (see, e.g., Section 4 of \cite{Ca96}).

Here we are interested in solitary waves of the form
\begin{align}
	\begin{aligned}
		&u(t,x)=\e^{\i(\omega t+cx-c^2t+\theta)}U(x-2ct-x_0),\\
		&v(t,x)=\e^{\i(st+cx-c^2t+\phi)}V(x-2ct-x_0),
	\end{aligned}
	\label{eqn:soliton}
\end{align}
such that the real-valued functions $(U,V)=(U(x),V(x))$ satisfy $U(x),V(x)\to 0$ as $x\to\pm\infty$,
 where $c$, $x_0$, $\theta$, $\phi$, and $\omega,s>0$ are real constants.
Henceforth, without loss of generality, we take $c,x_0,\theta,\phi=0$ by scaling the independent variables if necessary, since \eqref{eqn:CNLS} is invariant under the Galilean transformations
\begin{align*}
	(u(t,x),v(t,x))\mapsto\e^{\i(cx-c^2t)}(u(t,x-2ct),v(t,x-2ct)),\quad c\in\Rset,
\end{align*}
the spatial translations
\begin{align}
	(u(t,x),v(t,x))\mapsto(u(t,x-x_0),v(t,x-x_0)),\quad x_0\in\Rset,
	\label{eqn:transl}
\end{align}
and the gauge transformations
\begin{align}
	(u(t,x),v(t,x))\mapsto(\e^{\i\theta}u(t,x),\e^{\i\phi}v(t,x)),\quad\theta,\phi\in\Rset.
	\label{eqn:gauge}
\end{align}
So $(U,V)=(U(x),V(x))$ solves
\begin{align}
	\begin{aligned}
		&-U''+\omega U-\partial_1F(U^2,V^2;\mu)U=0,\\
		&-V''+sV-\partial_2F(U^2,V^2;\mu)V=0,
	\end{aligned}
	\label{eqn:UV}
\end{align}
where the prime represents the differentiation with respect to $x$.
Throughout this paper we assume that \eqref{eqn:UV} possess a homoclinic solution of the form $(U,V)=(U_0(x),0)$.
This assumption holds with $U_0(x)=\sqrt{2\omega}\sech(\sqrt\omega x)$ for the particular nonlinearity \eqref{eqn:cubic}.
We refer to the solitary wave of the form
\begin{align}
	(u(t,x),v(t,x))=(\e^{\i\omega t}U_0(x),0),
	\label{eqn:fun_soliton}
\end{align}
as a \textit{fundamental solitary wave}.
Obviously, $(U,V)=(-U_0(x),0)$ is also a homoclinic solution to \eqref{eqn:UV} but it gives the same solitary wave \eqref{eqn:soliton} as \eqref{eqn:fun_soliton} if $\theta=\pi$ is taken.
So we only treat \eqref{eqn:fun_soliton}.

Bl\'azquez-Sanz and Yagasaki \cite{BlYa12} studied bifurcations of homoclinic orbits in a large class of ordinary differential equations using Gruendler's version of Melnikov's method \cite{Gu92}, and applied their theory to show that infinitely many pitchfork bifurcations of $(U,V)=(U_0(x),0)$ occur in \eqref{eqn:UV} with \eqref{eqn:cubic} when $\beta_1$ is taken as a control parameter.
Moreover, in a general situation, they proved that if bifurcations of homoclinic orbits occur, then the associated variational equations, i.e., linearized equations, around the homoclinic orbits are integrable in the meaning of differential Galois theory \cite{PuSi03}.
Homoclinic orbits born at the pitchfork bifurcations in \eqref{eqn:UV} correspond to solitary waves of the form \eqref{eqn:soliton} with nonzero $U$- and $V$-components, which are often called vector solitons, bifurcated from the fundamental solitary waves \eqref{eqn:fun_soliton} in \eqref{eqn:CNLS}.

Once establishing their existence, the stability problem of these bifurcated solitary waves naturally arises.
Ohta \cite{Oh96} proved the orbital stability of the fundamental solitary waves in \eqref{eqn:CNLS} with \eqref{eqn:cubic} and $\beta_2=1$.
Subsequently, his result was extended to more general nonlinearity cases by several authors \cite{Bh15,CiZu00,Ng11,NgWa11,NgWa16}.
All of these results are based on variational characterizations of solitary waves.
On the other hand, the spectral stability of vector solitons in \eqref{eqn:CNLS} with \eqref{eqn:cubic} or the saturable nonlinearity was discussed by Pelinovsky and Yang \cite{PeYa05}.
They employed the \emph{Hamiltonian-Krein index} theory, which was established by Pelinovsky \cite{Pe05} and Kapitula et al.\ \cite{KKS04,KKS05} independently, and carried out some formal calculations to claim that any purely imaginary embedded eigenvalue with a positive (resp.\ negative) Krein signature disappears (resp.\ changes to a pair of eigenvalues with nonzero real parts) after bifurcations.
This suggests that the bifurcated solitary waves of \eqref{eqn:CNLS} with \eqref{eqn:cubic} are spectrally stable and unstable, respectively, when they have positive and sign-changing $V$-components.
A rigorous analysis was given by Cuccagna et al.\ \cite{CPV05} using a technique based on the analytic continuation of the resolvent.
Moreover, some nondegenerate conditions which are generically satisfied but difficult to verify for concrete examples were assumed there. 
Li and Promislow \cite{LiPr00} also studied the spectral stability of solitary waves in CNLS equations with a different special nonlinearity, using the \emph{Evans function} technique.
This technique was first proposed in Evans' earlier works about the stability of pulse solutions in nerve axon equations \cite{Ev72a,Ev72b,Ev72c,Ev75}, and subsequently discussed in a topological framework by Alexander et al.\ \cite{AGJ90}.

In this paper, we study pitchfork bifurcations of the fundamental solitary waves detected by the Melnikov analysis of \cite{BlYa12} in \eqref{eqn:CNLS} and determine the spectral and/or orbital stability of the fundamental and bifurcated solitary waves under some nondegenerate conditions which are easy to verify, compared with those of \cite{CPV05,PeYa05}.
Here the terminology ``pitchfork bifurcation'' is used with caution:
 a pair of homoclinic solutions to \eqref{eqn:UV},
 which correspond to the same family of solitary waves of the form \eqref{eqn:soliton}
 in \eqref{eqn:CNLS}, are born at the bifurcation point.
Our main tool to determine their spectral stability
 is the Evans function technique as in \cite{LiPr00} but more careful treatments are required.
The Hamiltonian-Krein index theory is also used as in \cite{PeYa05}.
We apply our theory to the particular nonlinearity case of \eqref{eqn:cubic}:
  pitchfork bifurcations of the fundamental solitary waves are detected
 and the spectral or orbital stability of the fundamental and bifurcated solitary waves
 are determined for almost all parameter values.
In particular, the nondegenerate conditions for bifurcations and spectral stability are given explicitly.
The theoretical results are also demonstrated with numerical computations
 obtained by the computer tool \texttt{AUTO} \cite{DO12}.

The outline of this paper is as follows:
In Section~2, we state our assumptions and main results
 on pitchfork bifurcations of the fundamental solitary waves in \eqref{eqn:CNLS}.
In Section~3, we use the Melnikov analysis of \cite{BlYa12} for \eqref{eqn:UV}
 and prove the main result on pitchfork bifurcations.
For the reader's convenience, in our setting,
 the Melnikov analysis is briefly reviewed. 
In Section~4, we present the Hamiltonian-Krein index theory and Evans function technique
 in appropriate forms.
Applying these approaches,
  we prove the main results on the stability of the fundamental and bifurcated solitary waves
 in Sections~5 and 6, respectively.
Finally, we illustrate our theory for the particular nonlinearity \eqref{eqn:cubic}
 along with some numerical computations in Section~7.


\section{Main Results}

In this section, we give our assumptions and main results
 on bifurcations of the fundamental solitary wave \eqref{eqn:fun_soliton}
 and the stability of the fundamental and bifurcated solitary waves
 in the CNLS equations \eqref{eqn:CNLS}.
%
\subsection{Bifurcations of the Fundamental Solitary Wave}
We first make the following assumptions.
\begin{itemize}
\setlength{\leftskip}{-1em}
\item[\textbf{(A1)}]
	$F:\Rset\times\Rset\times\Rset\to\Rset$ is a $C^3$ function satisfying
\begin{align*}
	F(0,0;\mu)=\partial_1F(0,0;\mu)=\partial_2F(0,0;\mu)=0,\quad
	\partial_\mu F(\zeta,0;\mu)=0,
\end{align*}
for any $\zeta,\mu\in\Rset$.
Henceforth we suppress the dependence of $\partial_1^jF(\zeta,0;\mu)$ on $\mu$ and write $\partial_1^jF(\zeta,0)$ for $j=0,1,2,3$.
\end{itemize}
\begin{itemize}
\setlength{\leftskip}{-1em}
\item[\textbf{(A2)}]
	When $V=0$, the first equation of \eqref{eqn:UV},
	\begin{align*}
		-U''+\omega U-\partial_1F(U^2,0)U=0,
	\end{align*}
	has a homoclinic solution $U_0(x)\not\equiv0$ such that $\lim_{x\to\pm\infty}U_0(x)=0$ and $U_0'(0)=0$.
\item[\textbf{(A3)}]
	For some $\mu_0\in\Rset$, the \emph{normal variational equation} of \eqref{eqn:UV}
	around the homoclinic orbit $(U,V)=(U_0(x),0)$,
	\begin{align}
		-\delta V''+s\delta V-\partial_2F(U_0(x)^2,0;\mu_0)\delta V=0,
		\label{eqn:VE2}
	\end{align}
	 has a nonzero bounded solution $\delta V=V_1(x)$ such that
	$\lim_{x\to\pm\infty}V_1(x)=0$.
	Henceforth we take $\mu_0=0$ by shifting the value of $\mu$ if necessary without loss of generality.
\end{itemize}

\begin{rmk}\label{rmk:U0V1}\
	\begin{enumerate}
	\setlength{\leftskip}{-2em}
	\item[(i)]
		{\rm\textbf{(A2)}} holds if and only if
		\begin{align*}
			\zeta_0\defeq\inf\{\zeta>0\mid F(\zeta^2,0)=\omega\zeta^2\}>0
		\end{align*}
		exists and $\partial_1F(\zeta_0^2,0)>\omega$.
		In addition, $U_0(x)$ decays exponentially as $x\to\pm\infty$ and is unique up to translations.
		We can also take $U_0(x)$ such that it is a positive even function and satisfies $U_0(0)=\zeta_0$ and $U_0'(x)<0$ for $x>0$.
		See Section~6 of \cite{BeLi83}.
		Henceforth, we always assume these properties for $U_0(x)$.
	\item[(ii)]
		In {\rm\textbf{(A3)}}, $V_1(x)$ is either even or odd.
		Actually, when {\rm\textbf{(A3)}} holds, since the space of all solutions to \eqref{eqn:VE2} decaying as $x\to\pm\infty$ is of dimension one and $V_1(-x)$ is such a solution to \eqref{eqn:VE2},
		we have $V_1(x)=cV_1(-x)$ for some $c\in\Rset$.
		Hence, if $V_1(0)\ne0$, then $c=1$.
		Otherwise we have $c=-1$ since $V_1'(x)=-cV_1'(-x)$ and $V_1'(0)\ne0$.
	\end{enumerate}
\end{rmk}

Define the two integrals
\begin{align}
	\bar{a}_2=&-\int_{-\infty}^\infty\partial_2\partial_\mu F(U_0(x)^2,0;0)V_1(x)^2\,\d x,
	\label{eqn:a2ex}\\
	\bar{b}_2=&-2\int_{-\infty}^\infty\partial_1\partial_2F(U_0(x)^2,0;0)U_0(x)U_2(x)V_1(x)^2\,\d x\nonumber\\
	&-\int_{-\infty}^\infty\partial_2^2F(U_0(x)^2,0;0)V_1(x)^4\,\d x,
	\label{eqn:b2ex}
\end{align}
where
\begin{align}
	U_2(x)
	=&-\phi_{11}(x)\int_x^\infty\phi_{12}(y)\partial_1\partial_2F(U_0(y)^2,0;0)U_0(y)V_1(y)^2\,\d y\nonumber\\
	&-\phi_{12}(x)\int_0^x\phi_{11}(y)\partial_1\partial_2F(U_0(y)^2,0;0)U_0(y)V_1(y)^2\,\d y.
	\label{eqn:U2}
\end{align}
Applying the Melnikov analysis of \cite{BlYa12} reviewed in 
 Section~3.1, we have the following bifurcation result.

\begin{thm}\label{thm:cnls_bif}
	Suppose that {\rm\textbf{(A1)}-\textbf{(A3)}} and $\bar a_2,\bar b_2\ne0$ hold.
	Then a pitchfork bifurcation of the fundamental solitary wave \eqref{eqn:fun_soliton} occurs at $\mu=0$.
	In addition, it is supercritical or subcritical, depending on whether $\bar{a}_2\bar{b}_2<0$ or $>0$.
	Moreover, the bifurcated solitary waves are written as
	\begin{align}
		(u(t,x),v(t,x))=(\e^{\i\omega t}U_\epsilon(x),\e^{\i st}V_\epsilon(x))
		\label{eqn:bif_soliton}
	\end{align}
	with
	\begin{align}
		U_\epsilon(x)=U_0(x)+\epsilon^2U_2(x)+O(\epsilon^4),\quad
		V_\epsilon(x)=\pm(\epsilon V_1(x)+O(\epsilon^3)),
		\label{eqn:bif_UV}
	\end{align}
	where $\epsilon>0$ is a small parameter such that
	\begin{align}
		\mu=\bar\mu\,\epsilon^2,\qquad\bar\mu\defeq-{\bar b_2}/{\bar a_2}.
		\label{eqn:mu_eps}
	\end{align}
\end{thm}
A proof of Theorem~\ref{thm:cnls_bif} is given in Section~3.

\subsection{Stability of the Fundamental and Bifurcated Solitary Waves}

We turn to the stability of the fundamental and bifurcated solitary waves.
We begin with the definition of orbital stability.
\begin{dfn}
	The solitary wave \eqref{eqn:soliton2} is \emph{orbitally stable} if for any $\epsilon>0$ there exists a constant $\delta>0$ such that for all $(u_0,v_0)\in H^1(\Rset,\Cset^2)$ with
	\begin{align*}
		\inf_{x_0,\theta,\phi\in\Rset}\bigl(\|u_0-\e^{\i\theta}U(\cdot-x_0)\|_{H^1}+\|v_0-\e^{\i\phi}V(\cdot-x_0)\|_{H^1}\bigr)<\delta,
	\end{align*}
	the solution $(u(t),v(t))\in C(\Rset,H^1(\Rset,\Cset^2))$ to \eqref{eqn:CNLS} with $(u(0),v(0))=(u_0,v_0)$ satisfies
	\begin{align*}
		\sup_{t\in\Rset}\inf_{x_0,\theta,\phi\in\Rset}\bigl(\|u(t)-\e^{\i\theta}U(\cdot-x_0)\|_{H^1}+\|v(t)-\e^{\i\phi}V(\cdot-x_0)\|_{H^1}\bigr)<\epsilon.
	\end{align*}
	Otherwise, it is \emph{orbitally unstable}.
\end{dfn}

The \emph{linearized operator} $\mathscr L:H^2(\Rset,\Cset^4)$ $\subset L^2(\Rset,\Cset^4)\to L^2(\Rset,\Cset^4)$ around the solitary wave
\begin{align}
	(u(t,x),v(t,x))=(\e^{\i\omega t}U(x),\e^{\i st}V(x))
	\label{eqn:soliton2}
\end{align}
is given by
\begin{align}
	\mathscr L
	=-\i\Sigma_3\left[-\partial_x^2I_4+\begin{pmatrix}\omega I_2&O_2\\O_2&sI_2\end{pmatrix}-\begin{pmatrix}B_{11}(x)&B_{12}(x)\\B_{21}(x)&B_{22}(x)\end{pmatrix}\right]
	\label{eqn:lin_op}
\end{align}
where $I_n$ and $O_n$, respectively, represents the $n\times n$ identity and zero matrices,
 and
\begin{align*}
	&\sigma_1=\begin{pmatrix}0&1\\1&0\end{pmatrix},\quad
	\sigma_3=\begin{pmatrix}1&0\\0&-1\end{pmatrix},\quad
	\Sigma_3=\begin{pmatrix}\sigma_3&O_2\\O_2&\sigma_3\end{pmatrix},\\
	&B_{11}(x)=\partial_1^2F(U(x)^2,V(x)^2)U(x)^2(I_2+\sigma_1)+\partial_1F(U(x)^2,V(x)^2)I_2,\\
	&B_{12}(x)=B_{21}(x)=\partial_1\partial_2F(U(x)^2,V(x)^2)U(x)V(x)(I_2+\sigma_1),\\
	&B_{22}(x)=\partial_2^2F(U(x)^2,V(x)^2)V(x)^2(I_2+\sigma_1)+\partial_2F(U(x)^2,V(x)^2)I_2.
\end{align*}
See Section~4.1 for the derivation of $\mathscr L$.
We easily show that the essential spectrum of $\mathscr L$ is given by
\begin{align*}
	\sigma_{\mathrm{ess}}(\mathscr L)=\i(-\infty,-\min\{\omega,s\}]\cup\i[\min\{\omega,s\},\infty).
\end{align*}
Moreover, $\sigma(\mathscr L)$ is symmetric with respect to the real and imaginary axis: if $\lambda\in\sigma(\mathscr L)$, then $-\lambda,\pm\lambda^*\in\sigma(\mathscr L)$,
 where the superscript `$\ast$' represents the complex conjugate.

\begin{dfn}
	The solitary wave \eqref{eqn:soliton2} in \eqref{eqn:CNLS} is \emph{spectrally stable} if the spectrum $\sigma(\mathscr L)$ of $\mathscr L$ is contained in the closed left half plane.
	Otherwise, it is \emph{spectrally unstable}.
\end{dfn}

For the fundamental and bifurcated solitary waves,
 we have the following theorems.

\begin{thm}\label{thm:fun_stab}
For any $\mu\in\Rset$, the fundamental solitary wave \eqref{eqn:fun_soliton} is orbitally and spectrally stable if $\partial_\omega\|U_0\|_{L^2}^2>0$, and it is orbitally and spectrally unstable if $\partial_\omega\|U_0\|_{L^2}^2<0$.
\end{thm}

\begin{thm}\label{thm:bif_orb_stab}
	Let $\ell$ be the number of zeros of $V_1(x)$.
	For $\epsilon>0$ sufficiently small, the bifurcated solitary waves \eqref{eqn:bif_soliton} are orbitally and spectrally stable if $\partial_\omega\|U_0\|_{L^2}^2>0$ and $\ell=0$, and it is orbitally and spectrally unstable if $\partial_\omega\|U_0\|_{L^2}^2<0$.
\end{thm}

Proofs of Theorem~\ref{thm:fun_stab} and \ref{thm:bif_orb_stab}
 are given in Sections~5 and 6.1, respectively.
Note that Theorem~\ref{thm:bif_orb_stab} says nothing about the stability of the bifurcated solitary waves with $\ell>0$ when $\partial_\omega\|U_0\|_{L^2}^2>0$.
We assume that $\partial_\omega\|U_0\|_{L^2}^2>0$
 and discuss the stability of bifurcated solitary waves in the rest of this section.

Let $\mathscr L_0$ and $\mathscr L_\epsilon$, respectively,
 denote the linearized operator around the fundamental and bifurcated solitary waves.
To determine the spectral stability of the bifurcated solitary waves,
 we need to discuss how purely imaginary eigenvalues of $\mathscr L_0$ change
 when $\epsilon>0$.
Let $E(\lambda,\epsilon)$ denote the \emph{Evans function} for $\mathscr L_\epsilon$,
 which allows us to describe the changes of these eigenvalues for $\epsilon>0$.
See Section~4.3 for its definition.
Especially, a zero of $E(\lambda,\epsilon)$ is typically an eigenvalue 
 of $\mathscr L_\epsilon$ but is not necessarily so,
 and it is called a \emph{resonance pole} if not (e.g., Chapter~9 of \cite{KaPr13}).
We also see that the Evans function $E(\lambda,\epsilon)$ may not be analytic at
 $\lambda\in\{\pm\i\omega,\pm\i s\}$,
 which is a branch point for a square root function
 in \eqref{eqn:mu_def1} or \eqref{eqn:mu_def2} below (see also Section~4.3)
 and at which the corresponding Jost solutions may lose their analyticity in $\lambda$.
Near 
 $\lambda=\lambda_\mathrm{br}$ with $\lambda_\mathrm{br}\in\{\pm\i\omega,\pm\i s\}$,
 using the transformation $\gamma^2=\pm\i(\lambda-\lambda_\mathrm{br})$, we have
\begin{align*}
\widetilde E(\gamma,\epsilon)=E(\mp\i\gamma^2+\lambda_\mathrm{br},\epsilon)
\end{align*}
with $\gamma\approx 0$,
 where the upper or lower sign is taken simultaneously,
 and refer to $\widetilde E(\gamma,\epsilon)$ as the \emph{extended Evans function}.
See Sections~4.3 and 4.4 for more details.

Here we compute the Evans function $E(\lambda,0)$ for $\mathscr L_0$.
The eigenvalue problem $\mathscr L_0\psi=\lambda\psi$ is separated into three parts:
(A) the first and second components
\begin{align}
	\begin{aligned}
		&-p''+\omega p-\partial_1^2F(U_0(x)^2,0)U_0(x)^2(p+q)-\partial_1F(U_0(x)^2,0)p=\i\lambda p,\\
		&-q''+\omega q-\partial_1^2F(U_0(x)^2,0)U_0(x)^2(p+q)-\partial_1F(U_0(x)^2,0)q=-\i\lambda q;
	\end{aligned}
	\label{eqn:lin_eq_1}
\end{align}
(B) the fourth component
\begin{align}
	-p''+sp-\partial_2F(U_0(x)^2,0;\mu)p=-\i\lambda p;
	\label{eqn:lin_eq_3}
\end{align}
and (C) the third component
\begin{align}
	-p''+sp-\partial_2F(U_0(x)^2,0;\mu)p=\i\lambda p.
	\label{eqn:lin_eq_2}
\end{align}
Let $(p_j(x,\lambda),q_j(x,\lambda))$, $j=1,2,5,6$, denote the Jost solutions
 to \eqref{eqn:lin_eq_1} such that
\begin{equation}
	\begin{split}
		&\lim_{x\to-\infty}\left\lvert\e^{-\nu_1x}\!\begin{pmatrix}p_1(x,\lambda)\\q_1(x,\lambda)\\p_1'(x,\lambda)\\q_1'(x,\lambda)\end{pmatrix}-\begin{pmatrix}1\\0\\\nu_1\\0\end{pmatrix}\right\rvert
		=\lim_{x\to-\infty}\left\lvert\e^{-\nu_2x}\!\begin{pmatrix}p_2(x,\lambda)\\q_2(x,\lambda)\\p_2'(x,\lambda)\\q_2'(x,\lambda)\end{pmatrix}-\begin{pmatrix}0\\1\\0\\\nu_2\end{pmatrix}\right\rvert\\
		&=\lim_{x\to+\infty}\left\lvert\e^{\nu_1x}\!\begin{pmatrix}p_5(x,\lambda)\\q_5(x,\lambda)\\p_5'(x,\lambda)\\q_5'(x,\lambda)\end{pmatrix}-\begin{pmatrix}1\\0\\-\nu_1\\0\end{pmatrix}\right\rvert
		=\lim_{x\to+\infty}\left\lvert\e^{\nu_2x}\!\begin{pmatrix}p_6(x,\lambda)\\q_6(x,\lambda)\\p_6'(x,\lambda)\\q_6'(x,\lambda)\end{pmatrix}-\begin{pmatrix}0\\1\\0\\-\nu_2\end{pmatrix}\right\rvert=0,
	\end{split}
	\label{eqn:caseA_asym}
\end{equation}
where
\begin{align}
\nu_1(\lambda)=\sqrt{\omega-\i\lambda},\quad
\nu_2(\lambda)=\sqrt{\omega+\i\lambda}.
\label{eqn:mu_def1}
\end{align}
Actually, by a fundamental result on linear differential equations
 (e.g., Section~3.8 of \cite{CoLe55}),
 these solutions exist since $(p,q)=(e^{\pm\nu_1 x},0)$ and $(0,e^{\pm\nu_2 x})$
 are solutions to
\[
-p''+\omega p=\i\lambda p,\quad
-q''+\omega q=-\i\lambda q,
\]
to which \eqref{eqn:lin_eq_1} converges as $x\to\pm\infty$.
Moreover, one of $(p_1,q_1)$ and $(p_2,q_2)$ (resp.\ $(p_5,q_5)$ and $(p_6,q_6)$)
 is uniquely determined but the other is not,
 depending on which has a bigger real part, $\nu_1$ or $\nu_2$.
On the other hand, let $p_4(x,\lambda)$ and $p_8(x,\lambda)$
 (resp.\ $p_3(x,\lambda)$ and $p_7(x,\lambda)$) denote the Jost solutions
 to \eqref{eqn:lin_eq_3} (resp.\ to \eqref{eqn:lin_eq_2}) such that
\begin{align}
		\lim_{x\to-\infty}\left\lvert\e^{-\nu_4x}\!\begin{pmatrix}p_4(x,\lambda)\\p_4'(x,\lambda)\end{pmatrix}-\begin{pmatrix}1\\\nu_4\end{pmatrix}\right\rvert
		=\lim_{x\to+\infty}\left\lvert\e^{\nu_4x}\!\begin{pmatrix}p_8(x,\lambda)\\p_8'(x,\lambda)\end{pmatrix}-\begin{pmatrix}1\\-\nu_4\end{pmatrix}\right\rvert=0
		\label{eqn:caseC_asym}
\end{align}
(resp.
\begin{align}
		\lim_{x\to-\infty}\left\lvert\e^{-\nu_3x}\!\begin{pmatrix}p_3(x,\lambda)\\p_3'(x,\lambda)\end{pmatrix}-\begin{pmatrix}1\\\nu_3\end{pmatrix}\right\rvert
		=\lim_{x\to+\infty}\left\lvert\e^{\nu_3x}\!\begin{pmatrix}p_7(x,\lambda)\\p_7'(x,\lambda)\end{pmatrix}-\begin{pmatrix}1\\-\nu_3\end{pmatrix}\right\rvert=0
		\label{eqn:caseB_asym}
\end{align}
) and
\begin{align*}
	p_4(x,\lambda)=p_8(-x,\lambda),\quad
	p_3(x,\lambda)=p_4(x,-\lambda),\quad
	p_7(x,\lambda)=p_8(x,-\lambda),
\end{align*}
where
\begin{align}
\nu_3(\lambda)=\sqrt{s-\i\lambda},\quad
\nu_4(\lambda)=\sqrt{s+\i\lambda}.
	\label{eqn:mu_def2}
\end{align}
These solutions uniquely exist since $p=e^{\pm\nu_4 x}$ (resp.\ $p=e^{\pm\nu_3 x}$)
 are solutions to
\[
-p''+s p=\i\lambda p\quad
(\mbox{resp.\ }-p''+s p=\i\lambda q),
\]
to which \eqref{eqn:lin_eq_3} (resp.\ \eqref{eqn:lin_eq_2}) converges as $x\to\pm\infty$.
See Section~5 for more details on the Jost solutions.
We choose the branch cut for $\nu_j(\lambda)$, $j=1,\ldots,4$, such that
\[
-\pi<\arg(\lambda\pm\i\omega),\arg(\lambda\pm\i s)<\pi,
\]
i.e.,
\begin{gather}
	-\frac{3\pi}{4}<\arg\nu_1(\lambda),\arg\nu_3(\lambda)<\frac{\pi}{4},\quad
	-\frac{\pi}{4}<\arg\nu_2(\lambda),\arg\nu_4(\lambda)<\frac{3\pi}{4}.
	\label{eqn:br_cut}
\end{gather}
We refer to $\lambda=\pm\i\omega$ and $\pm\i s$ as the \emph{branch points}.

We now obtain the Evans function for $\mathscr L_0$ as
\begin{align}
	E(\lambda,0)=E_\mathrm{A}(\lambda)E_\mathrm{B}(\lambda)E_\mathrm{C}(\lambda),
	\label{eqn:E_unperturbed}
\end{align}
where
\begin{align}
	&E_\mathrm{A}(\lambda)
	\defeq\det\!\begin{pmatrix}
		p_1&p_2&p_5&p_6\\
		q_1&q_2&q_5&q_6\\
		p_1'&p_2'&p_5'&p_6'\\
		q_1'&q_2'&q_5'&q_6'
	\end{pmatrix}\!(0,\lambda),
	\label{eqn:E1}\\
	&E_\mathrm{B}(\lambda)
	\defeq\det\!\begin{pmatrix}
		p_4&p_8\\
		p_4'&p_8'
	\end{pmatrix}\!(0,\lambda),
	\label{eqn:E3}
\end{align}
and
\begin{align}
	E_\mathrm{C}(\lambda)
	\defeq\det\!\begin{pmatrix}
		p_3&p_7\\
		p_3'&p_7'
	\end{pmatrix}\!(0,\lambda).
	\label{eqn:E2}
\end{align}
Near the branch points $\lambda=\lambda_\mathrm{br}\in\{\pm\i\omega,\pm\i s\}$,
 by $\gamma^2=\pm\i(\lambda-\lambda_\mathrm{br})$,
 we also modify \eqref{eqn:E_unperturbed} to
\begin{align*}
\widetilde E(\gamma,0)=\widetilde E_\mathrm{A}(\gamma)
 \widetilde E_\mathrm{B}(\gamma)\widetilde E_\mathrm{C}(\gamma)
\end{align*}
with $\gamma\approx 0$,
 where $\widetilde E(\gamma,0)$ is the extended Evans function
 for $\mathscr L_0$ and
\[
\widetilde E_\mathrm{A,B,C}(\gamma)
 =E_\mathrm{A,B,C}(\mp\i\gamma^2+\lambda_\mathrm{br}).
\]
Here the upper or lower sign has been taken simultaneously.

Let $\lambda_0\in\i\Rset\setminus\{0\}$ be a zero of $E(\lambda,0)$.
Since $p_3(x,\lambda)=p_4(x,-\lambda)$, $p_7(x,\lambda)=p_8(x,-\lambda)$,
 we have $E_\mathrm{C}(\lambda)=E_\mathrm{B}(-\lambda)$.
From the Sturm-Liouville theory (see, e.g., Section~2.3.2 of \cite{KaPr13})
 we see that \eqref{eqn:lin_eq_3} has only a finite number of eigenvalues,
 all of which are simple and belong to $\i(-\infty,s)$.
So we only have to consider zeros on $\i(-\infty,s)$ for $E_\mathrm{B}(\lambda)$
 since eigenvalues of $\mathscr L_\epsilon$ related to them
 remain near the segment for $\epsilon>0$ sufficiently small.
Under this observation, 
 we see that the following six cases occur:
\begin{enumerate}
	\item[(I)] $\lambda_0\in\i(-\min\{\omega,s\},\min\{\omega,s\})$;
	\item[(II)] $\lambda_0=\pm\i\min\{\omega,s\}$;
	\item[(III)] $\lambda_0$ is a zero of $E_\mathrm{B}(\lambda)$
	and belongs to $\i(\omega,s)$ when $\omega<s$;
	\item[(IV)] $\lambda_0$ is a zero of $E_\mathrm{B}(\lambda)$ and belongs to $\i(-\infty,-\min\{\omega,s\})$;
	\item[(V)] $\lambda_0=\i s$ is a zero of $E_\mathrm{B}(\lambda)$ when $\omega<s$;
	\item[(VI)] $\lambda_0=\pm\i\omega$ is a zero of $E_\mathrm{A}(\lambda)$ when $s<\omega$.
\end{enumerate}
In cases (I) and (II)
 $\lambda_0$ is a zero of either $E_\mathrm{A}(\lambda)$, $E_\mathrm{B}(\lambda)$, or $E_\mathrm{C}(\lambda)$.
We can restate cases~(III)-(V) with $E_\mathrm{C}(\lambda)$ as
\begin{enumerate}
	\item[(III')] $\lambda_0$ is a zero of $E_\mathrm{C}(\lambda)$ and belongs to $\i(-s,-\omega)$ when $\omega<s$;
	\item[(IV')] $\lambda_0$ is a zero of $E_\mathrm{C}(\lambda)$ and belongs to $\i(\min\{\omega,s\},\infty)$;
	\item[(V')] $\lambda_0=-\i s$ is a zero of $E_\mathrm{C}(\lambda)$ when $\omega<s$.
\end{enumerate}
Obviously, $\lambda_0$ is an isolated eigenvalue of $\mathscr L_0$ in case (I),
 while it is an endpoint of $\sigma_\mathrm{ess}(\mathscr L_0)$ in case (II).
For the other cases $\lambda_0$ is not an endpoint of $\sigma_\mathrm{ess}(\mathscr L_0)$ and it is an embedded eigenvalue or a resonance pole of $\mathscr L_0$.
Actually, $\lambda_0$ is a resonance pole in cases (V) and (VI).
In the following, we state our results in these six cases separately.
Proofs of the results are provided in Section~6
 except for one of Proposition~\ref{prop:2a}.
For simplicity,
 we essentially assume that $\lambda_0$ is a simple zero of $E(\lambda,0)$
 in cases (I), (III), and (IV).
In cases (II), (V), and (VI),
 we treat the extended Evans function $\widetilde E(\gamma,0)$,
 and assume that its zero $\gamma=0$ is simple.
If these assumptions do not hold,
 then more complicated analyses are required
 to determine the fate of the eigenvalue $\lambda_0$ when $\epsilon>0$.
Note that the zero $\lambda=0$ of $E(\lambda,0)$ is not simple
 by the symmetry \eqref{eqn:transl} and \eqref{eqn:gauge}
 but the dimension of the generalized kernel of $\mathscr{L}$ does not change
 after the pitchfork bifurcations, as stated at the end of Section~\ref{sec:lin_op}.
Hence, the zero $\lambda=0$ does not cause the spectral instability
 of bifurcated solitary waves.


We begin with case (I).
\begin{prop}
\label{prop:2a}
Suppose that the hypotheses of Theorem~\ref{thm:cnls_bif} hold
 and $\lambda=\lambda_0\in\i(-\min\{\omega,s\},\min\{\omega,s\})$
 is a simple zero of $E(\lambda,0)$. 
Then there exists an eigenvalue of $\mathscr L_\epsilon$ close to $\lambda=\lambda_0$
 on the imaginary axis for $\epsilon>0$ sufficiently small.
\end{prop}

\begin{proof}
The statement of this proposition
 immediately follows from the symmetry of $\sigma(\mathscr L_\epsilon)$.
\end{proof}

For case (II) we have the following (see Section~6.2 for its proof).
%

\begin{thm}\label{thm:main_0}
	Suppose that the hypotheses of Theorem~\ref{thm:cnls_bif} hold,
	$\lambda_0=\i\min\{\omega,s\}$ (resp.\ $\lambda_0=-\i\min\{\omega,s\}$)
	is a zero of $E(\lambda,0)$,
         and $\gamma=0$ is a simple zero of $\widetilde E(\gamma,0)$
         with $\lambda_\mathrm{br}=\lambda_0$. 
	Then for $\epsilon>0$ sufficiently small, one of the following three cases occurs:
	\begin{enumerate}
		\setlength{\leftskip}{-1.6em}
		\item $\mathscr L_\epsilon$ has a unique simple eigenvalue
		in $\lambda_0+\i(-\infty,0)$ (resp.\ $\lambda_0+\i(0,\infty)$)
		and no resonance pole near $\lambda=\lambda_0$;
		\item $\mathscr L_\epsilon$
		has a unique resonance pole
		in $\lambda_0+\i(-\infty,0)$ (resp.\ $\lambda_0+\i(0,\infty)$)
		and no eigenvalue near $\lambda=\lambda_0$;
		\item $\lambda=\lambda_0$ remains as an eigenvalue or a resonance pole
		of $\mathscr L_\epsilon$.
	\end{enumerate}
\end{thm}


We turn to case (III) and assume that $\omega<s$ and $E_\mathrm{B}(\lambda_0)=0$
 for $\lambda_0\in\i(\omega,s)$.
Define
\begin{align}
&
\Delta_j(\lambda)
\defeq p_j(0,\lambda)p_j'(0,\lambda)+q_j(0,\lambda)q_j'(0,\lambda),
\label{eqn:Deltaj}\\
&
\mathcal I_j(\lambda)
\defeq\int_{-\infty}^\infty \partial_1\partial_2F(U_0(x)^2,0;0)U_0(x)V_1(x)\bigl(p_j(x,\lambda)+q_j(x,\lambda)\bigr)p_4(x,\lambda)\,\d x
\label{eqn:Ij}
\end{align}
for $j=1,2$.
We note that $(p_1,q_1)$ is uniquely determined when $\lambda_0\in\i(\omega,s)$
 and that $(p_2,q_2)$ is also uniquely determined by Lemma~\ref{lem:Jost_normalize} below
 if $\Delta_1(\lambda_0)\neq 0$.
We have the following result (see Section~6.3 for its proof).

\begin{thm}\label{thm:main_1}
	Let $\omega<s$.
	Suppose that the hypotheses of Theorem~\ref{thm:cnls_bif} hold,
	$\lambda_0\in\i(\omega,s)$ is a simple zero of $E_\mathrm{B}(\lambda)$
	with $E_\mathrm{A}(\lambda_0),E_\mathrm{C}(\lambda_0)\ne0$,
	and $\Delta_1(\lambda_0),\mathcal I_2(\lambda_0)\neq 0$. 
	Then for $\epsilon>0$ sufficiently small, $\mathscr L_\epsilon$ has no eigenvalue
	in a neighborhood of $\lambda=\lambda_0$.
\end{thm}

\begin{rmk}
	In Theorem~\ref{thm:main_1}, the neighborhood of $\lambda=\lambda_0$ can be taken independently of $\epsilon$.
	This also holds in Theorems \ref{thm:main_3} and \ref{thm:main_4} below.
\end{rmk}


We next consider case (IV) partially and assume that $E_\mathrm{B}(\lambda_0)=0$
 for $\lambda_0\in\i(-\infty,-\omega)$.
So the case of $s<\omega$ and $\lambda_0\in\i[-\omega,-s)$ is excluded.
%
As in case (III), $(p_2,q_2)$ is uniquely determined
 when $\lambda_0\in\i(-\infty,-\min\{\omega,s\})$
 and that $(p_1,q_1)$ is also uniquely determined by Lemma~\ref{lem:Jost_normalize} below
 if $\Delta_2(\lambda_0)\neq 0$.
We have the following result (see Section~6.4 for its proof).

\begin{thm}\label{thm:main_2}
	Suppose that the hypotheses of Theorem~\ref{thm:cnls_bif} holds
	and $\lambda_0\in\i(-\infty,-\omega)$ is a simple zero of $E_\mathrm{B}(\lambda)$
	with $E_\mathrm{A}(\lambda_0),E_\mathrm{C}(\lambda_0)\ne0$,
	and $\Delta_2(\lambda_0),\mathcal I_1(\lambda_0)\neq 0$. 
	Then for $\epsilon>0$ sufficiently small,
	$\mathscr L_\epsilon$ has an eigenvalue with a positive real part
	in an $O(\epsilon^2)$-neighborhood of $\lambda=\lambda_0$.
\end{thm}

\begin{rmk}\label{rmk_main2_deg1}
	Assume that $s<\omega$ and $\lambda_0\in\i[-\omega,-s)$.
	Then we have $\Re\lambda(\epsilon)=O(\epsilon^3)$,
	as shown in Remark~\ref{rmk:main2_deg2} below,
	where $\lambda(\epsilon)$ is the eigenvalue of $\mathscr L_\epsilon$ near $\lambda_0$.
	Hence, more tremendous treatments
	than those in the proof of Theorem~\ref{thm:main_2} are required
	to estimate an eigenvalue of $\mathscr L_\epsilon$ near $\lambda=\lambda_0$
	for $\epsilon>0$ sufficiently small.
	This situation does not occur for the cubic nonlinearity \eqref{eqn:cubic}
	even if $s<\omega$ as seen in Section~7.
\end{rmk}


For case~(V)
 we have the following (see Section~6.5 for its proof).

\begin{thm}\label{thm:main_3}
	Let $\omega<s$.
	Suppose that the hypotheses of Theorem~\ref{thm:cnls_bif} hold,
	$E_\mathrm{B}(\i s)=0$,
	$\gamma=0$ is a simple zero of $\widetilde E_\mathrm{B}(\gamma)$
	with $\lambda_\mathrm{br}=\i s$ with $E_\mathrm{A}(\i s),E_\mathrm{C}(\i s)\ne0$
	and $\Delta_1(\i s),\mathcal I_2(\i s)\neq 0$. 
	Then for $\epsilon>0$ sufficiently small,
	$\mathscr L_\epsilon$ has no eigenvalue in a neighborhood of $\lambda=\i s$
	and no resonance pole at $\lambda=\i s$.
\end{thm}


Finally we consider case~(VI) and assume that $s<\omega$.
We make the following assumption.

\begin{itemize}
\item[\textbf{(A4)}]
	There exists a bounded solution $(\hat p(x),\hat q(x))$ to \eqref{eqn:lin_eq_1}
	with $\lambda=\i\omega$ such that
	\begin{align*}
		\lim_{x\to\pm\infty}\hat p(x)=0,\quad
		\lim_{x\to-\infty}\hat q(x)\ne0,\quad
		\lim_{x\to+\infty}\hat q(x)\ne0.
	\end{align*}
\end{itemize}

An equivalent statement to \textbf{(A4)} is given in Remark~\ref{rmk:A6_equiv} below.
We have the following result (see Section~6.6 for its proof).

\begin{thm}\label{thm:main_4}
	Let $s<\omega$.
	Suppose that the hypotheses of Theorem~\ref{thm:cnls_bif} hold,
	$E_\mathrm{A}(\i\omega)=0$,
	$\gamma=0$ is a simple zero of $\widetilde E_\mathrm{A}(\gamma)$
	with $\lambda_\mathrm{br}=i\omega$
	and $E_\mathrm{B}(\i\omega),E_\mathrm{C}(\i\omega)\ne0$,
	and $\Delta_1(\i\omega)\neq 0$. 
	Then, under assumption~{\rm\textbf{(A4)}},
	$\mathscr L_\epsilon$ has no eigenvalue in a neighborhood of $\lambda=\pm\i\omega$
	and no resonance pole at $\lambda=\pm\i\omega$ for $\epsilon>0$ sufficiently small.
\end{thm}

\begin{rmk}
	Let $\lambda_0\in\i\Rset$ be a simple eigenvalue of $\mathscr L_0$ which is a zero of $E_\mathrm{B}(\lambda)$.
	We see that $\lambda_0$ has a positive (resp.\ negative) Krein signature if it belongs to the upper (resp.\ lower) half plane (see Remark~\ref{rmk:krein_sig} below).
	See Section~7.1 of \cite{KaPr13} or Section~\ref{sec:KHam} for the definition of the Krein signature.
	Note that a similar statement is also found in Section~3.1 of \cite{PeYa05}.
	It follows from Theorems \ref{thm:main_1}, \ref{thm:main_3}, and \ref{thm:main_4} that embedded eigenvalues of $\mathscr L_0$ with positive Krein signatures as well as resonance poles at the endpoint of $\sigma_\mathrm{ess}(\mathscr L_0)$ disappear when $\epsilon>0$,
	provided that $\Delta_1(\lambda_0),\mathcal I_2(\lambda_0)\neq 0$.
	On the other hand, Theorem \ref{thm:main_2} states that
	embedded eigenvalues of $\mathscr L_0$ with negative Krein signatures
	have nonzero real parts when $\epsilon>0$,
	provided that $\Delta_2(\lambda_0),\mathcal I_1(\lambda_0)\neq 0$.
	These claims completely agree with the generic results obtained by Pelinovsky and Yang \cite{PeYa05} and Cuccagna et al. \cite{CPV05} stated in Section~1.
\end{rmk}


\section{Bifurcations and Approximate Homoclinic Solutions}

In this section, we review  the Melnikov analysis of \cite{BlYa12}
 in a sufficient setting for our purpose
 and prove Theorem~\ref{thm:cnls_bif} by applying it to \eqref{eqn:UV}.

\subsection{Melnikov Analysis}

We begin with the Melnikov analysis of \cite{BlYa12}
 on pitchfork bifurcations of homoclinic orbits
 in a class of two-degree-of-freedom Hamiltonian systems.
Our setting is simpler but sufficient for the analysis of \eqref{eqn:UV} in Section~3.2.

Consider two-degree-of-freedom Hamiltonian systems of the form
\begin{equation}
\xi'=J\D_\xi H(\xi,\eta;\mu)^\T,\quad
\eta'=J\D_\eta H(\xi,\eta;\mu)^\T,\quad
(\xi,\eta)\in\Rset^2\times\Rset^2,
\label{eqn:Hsys}
\end{equation}
where the Hamiltonian function $H:\Rset^2\times\Rset^2\times\Rset\to\Rset$ is $C^{r+1}$ ($r\ge 4$),
 the superscript `{\scriptsize T}' represents the transpose operator
 and $J$ is the $2\times 2$ symplectic matrix given by
\[
J=
\begin{pmatrix}
0 & 1\\
-1 & 0
\end{pmatrix}.
\]
We make the following assumptions.

\begin{itemize}
\setlength{\leftskip}{-0.8em}
\item[\textbf{(M1)}]
$\D_\xi H(0,0;\mu)=\D_\eta H(0,0;\mu)=0$ for any $\mu\in\Rset$.
\item[\textbf{(M2)}]
$\D_\eta H(\xi,0;0)=0$ for any $\xi\in\Rset^2$.
\end{itemize}

\textbf{(M1)} means that the origin $(\xi,\eta)=(0,0)\,(=O)$
 is an equilibrium of \eqref{eqn:Hsys} for any $\mu\in\Rset$,
 and \textbf{(M2)} means that the $\xi$-plane, $\{(\xi,\eta)\mid\eta=0\}$, is invariant
 under the flow of \eqref{eqn:Hsys} at $\mu=0$.
In particular, the system of \eqref{eqn:Hsys} restricted on the $\xi$-plane at $\mu=0$,
\begin{equation}
\xi'=J\D_\xi H(\xi,0;0)^\T,
\label{eqn:xi}
\end{equation}
has an equilibrium at $\xi=0$.
Moreover, $\D_\xi^j\D_\eta H(\xi,0;0)=0$, $j=1,2,\ldots$, for any $\xi\in\Rset^2$.

\begin{itemize}
\setlength{\leftskip}{-0.8em}
\item[\textbf{(M3)}]
At $\mu=0$ the equilibrium $\xi=0$ in \eqref{eqn:xi} is a hyperbolic saddle, i.e.,
 $J\D_\xi^2 H(0,0;0)$ has a pair of positive and negative real eigenvalues
 $\pm\lambda_1$ ($\lambda_1>0$), and has a homoclinic orbit $\xi^\h(x)$.
\item[\textbf{(M4)}]
$J\D_\eta^2 H(0,0;0)$ has a pair of positive and negative real eigenvalues
 $\pm\lambda_2$ ($\lambda_2>0$).
\end{itemize}

\textbf{(M3)} and \textbf{(M4)} mean that the equilibrium $O$ is a hyperbolic saddle in \eqref{eqn:Hsys}
 and has a homoclinic orbit $(\xi,\eta)=(\xi^\h(x),0)$.
The hyperbolic saddle $O$ has two-dimensional stable and unstable manifolds
 which intersect along the homoclinic orbit $(\xi^\h(x),0)$.

The \textit{variational equation} (VE) of \eqref{eqn:Hsys}
 around the homoclinic orbit $(\xi^\h(x),0)$ at $\mu=0$ is given by
\begin{equation}
\delta\xi'=J\D_\xi^2 H(\xi^\h(x),0;0)\delta\xi,\quad
\delta\eta'=J\D_\eta^2 H(\xi^\h(x),0;0)\delta\eta.
\label{eqn:ve}
\end{equation}
The first equation of \eqref{eqn:ve} has a nonzero bounded solution $\delta\xi=\xi^{\h\prime}(x)$
 satisfying $\lim_{x\to\pm\infty}\delta\xi(x)=0$, where $\xi^{\h\prime}(x)=\frac{\d}{\d x}\xi^\h(x)$.
We assume the following.
\begin{itemize}
\setlength{\leftskip}{-0.8em}
\item[\textbf{(M5)}]
The second equation of \eqref{eqn:ve} also has a nonzero bounded solution $\delta\eta=\phi(x)$
 such that $\lim_{x\to\pm\infty}\phi(x)=0$.
\end{itemize}
When \eqref{eqn:Hsys} is analytic, it was shown in \cite{BlYa12} that 
 the VE \eqref{eqn:ve} is integrable in the meaning of differential Galois theory \cite{PuSi03} 
 if \textbf{(M5)} holds.

The adjoint equation for \eqref{eqn:ve},
 which we call the \textit{adjoint variational equation} (AVE), is given by
\begin{equation}
\delta\hat{\xi}'=\D_\xi^2 H(\xi^\h(x),0;0)J\delta\hat{\xi},\quad
\delta\hat{\eta}'=\D_\eta^2 H(\xi^\h(x),0;0)J\delta\hat{\eta}
\label{eqn:ave}
\end{equation}
since $\D_\xi^2 H(\xi^\h(x),0;0),\D_\eta^2 H(\xi^\h(x),0;0)$ are symmetric matrices
 and $J^\T=-J$.
Hence, $\delta\hat{\xi}=-J\xi^{\h\prime}(x)=\D_\xi H(\xi^\h(x),0)$ and $\delta\hat{\eta}=-J\phi(x)$
 are, respectively, nonzero bounded solutions of the first and second equations of \eqref{eqn:ave}
 since $J^2=-I_2$ and $J^{-1}=-J$, where $I_n$ is the $n\times n$ identity matrix.

We also assume that the Hamiltonian system \eqref{eqn:Hsys} is $\Zset_2$-equivariant.
\begin{itemize}
\setlength{\leftskip}{-0.8em}
\item[\textbf{(M6)}]
There exist two $2\times 2$ matrices $S_1,S_2$, such that either $S_1\ne I_2$ or $S_2\ne I_2$, $S_j^2=I_2$, $j=1,2$, and
\begin{align*}
&
S_1J\D_\xi H(\xi,\eta;\mu)^\T=J\D_\xi H(S_1\xi,S_2\eta;\mu)^\T,\\
&
S_2J\D_\eta H(\xi,\eta;\mu)^\T=J\D_\eta H(S_1\xi,S_2\eta;\mu)^\T.
\end{align*}
\end{itemize}
It follows from \textbf{(M6)} that
 if $(\xi,\eta)=(\bar{\xi}(x),\bar{\eta}(x))$ is a solution to \eqref{eqn:Hsys},
 then so is $(\xi,\eta)=(S_1\bar{\xi}(x),S_2\bar{\eta}(x))$.
See, e.g., Section~7.4 of \cite{Ku04} for more details on $\Zset_2$-equivariant systems.
In particular, the space $\Rset^4$ can be decomposed into a direct sum as
\[
\Rset^4=X^+\oplus X^-,
\]
where $X^\pm=\{(\xi,\eta)\in\Rset^4\mid(S_1\xi,S_2\eta)=\pm(\xi,\eta)\}$.
Under \textbf{(M6)}, if $(\xi^\h(x),0)\in X^+$ for any $x\in\Rset$,
 then the VE \eqref{eqn:ve} and AVE \eqref{eqn:ave} are also $\Zset_2$-equivariant since
\begin{align*}
&
S_1J\D_\xi^2 H(\xi,\eta;\mu)=J\D_\xi^2 H(S_1\xi,S_2\eta;\mu)S_1,\\
&
S_2J\D_\eta^2 H(\xi,\eta;\mu)=J\D_\eta^2 H(S_1\xi,S_2\eta;\mu)S_2
\end{align*}
for any $\xi,\eta\in\Rset^2$.
Finally, we assume the following.
\begin{itemize}
\setlength{\leftskip}{-0.8em}
\item[\textbf{(M7)}]
For any $x\in\Rset$, $(\xi^\h(x),0),(J\xi^{\h\prime}(x),0)\in X^+$ and $(0,\phi(x)),(0,J\phi(x))\in X^-$.
\end{itemize}

Let
\[
	p(\xi,\eta)=J\D_\xi H(\xi,\eta;0)^\T,\quad
	q(\xi,\eta)=J\D_\eta H(\xi,\eta;0)^\T,
\]
and let $p_j(\xi,\eta)$ and $q_j(\xi,\eta)$ be, respectively, the $j$-th order terms of the Taylor expansions
 of $p(\xi,\eta)$ and $q(\xi,\eta)$ around $\eta=0$ with respect to $\eta$, i.e., $p_j(\xi,\eta),q_j(\xi,\eta)=O(|\eta|^j)$, $j=0,1,2,3$, and
\begin{align*}
p(\xi,\eta)=& p_0(\xi,\eta)+p_1(\xi,\eta)+p_2(\xi,\eta)+p_3(\xi,\eta)+O(|\eta|^4),\\
q(\xi,\eta)=& q_0(\xi,\eta)+q_1(\xi,\eta)+q_2(\xi,\eta)+q_3(\xi,\eta)+O(|\eta|^4).
\end{align*}
Note that $q_0(\xi,\eta),p_1(\xi,\eta)=0$ by \textbf{(M2)}.
Let $\xi=\tilde{\xi}^\mu(x)$ and $\tilde{\xi}^\alpha(x)$, respectively, denote nonzero bounded solutions to
\begin{equation}
	\xi'=J\D_\xi^2H(\xi^\h(x),0;0)\xi+J\D_\mu\D_\xi H(\xi^\h(x),0;0)^\T
	\label{eqn:ximu}
\end{equation}
and
\begin{equation}
	\xi'=J\D_\xi^2H(\xi^\h(x),0;0)\xi+p_2(\xi^\h(x),\phi(x))
	\label{eqn:xia}
\end{equation}
with $\langle\D_\xi H(\xi^\h(0),0;0)^\T,\xi(0)\rangle=0$,
 where $\langle\xi,\eta\rangle$ represents the inner product of $\xi,\eta\in\Rset^2$.
Define the definite integrals
\begin{equation}
\begin{split}
	&\bar{a}_2=-\int_{-\infty}^\infty\langle J\phi(x),J\D_\mu\D_\eta^2 H(\xi^\h(x),0;0)\phi(x)
	 +\tilde{q}_1(x,\tilde{\xi}^\mu(x),\phi(x))\rangle\d x,\\
	&\bar{b}_2=-\int_{-\infty}^\infty\langle J\phi(x),q_3(\xi^\h(x),\phi(x))
	 +\tilde{q}_1(x,\tilde{\xi}^\alpha(x),\phi(x))\rangle\d x,
\end{split}
\label{eqn:ab2}
\end{equation}
where
\[
	\tilde{q}_1(x,\xi,\eta)=\D_\xi q_1(\xi^\h(x),\eta(x))\xi.
\]
In the present situation, we state Theorem~2.16 of \cite{BlYa12} as follows.

\begin{thm}
\label{thm:Mel}
Under {\rm\textbf{(M1)}-\textbf{(M7)}}, suppose that $\bar{a}_2,\bar{b}_2\neq 0$.
Then a pitchfork bifurcation of homoclinic orbits occurs at $\mu=0$.
Moreover, it is supercritical or subcritical, depending on whether $\bar{a}_2\bar{b}_2<0$ or $>0$.
\end{thm}

See \cite{BlYa12} for the proof of Theorem \ref{thm:Mel}.
More general $n$-dimensional systems ($n\ge 4$) containing non-Hamiltonian systems
 and saddle-node bifurcations of homoclinic orbits were treated there.
Similar results were recently obtained for reversible systems in \cite{Ya20}.

\begin{rmk}\label{rmk:Mel}
	Homoclinic orbits created at the pitchfork bifurcation detected by Theorem \ref{thm:Mel} are written as
	\begin{align*}
		(\xi,\eta)=(\xi^\h(x),\alpha\phi(x))+O(\alpha^2),
	\end{align*}
	where $\alpha$ is a small parameter satisfying
	\begin{align}
		\bar a_2\mu+\bar b_2\alpha^2=0
		\label{eqn:bif_rel}
	\end{align}
	up to $O(\sqrt{\alpha^8+\mu^4})$.
	See \cite{BlYa12} for the details.
\end{rmk}

\subsection{Proof of Theorem~\ref{thm:cnls_bif}}

Using Theorem~\ref{thm:Mel} and Remark~\ref{rmk:Mel} for \eqref{eqn:UV},
 we now prove Theorem~\ref{thm:cnls_bif}.

\begin{proof}[Proof of Theorem~\ref{thm:cnls_bif}]
 Letting $\xi=(U,U')^\T$ and $\eta=(V,V')^\T$,
  we rewrite \eqref{eqn:UV} in the form of \eqref{eqn:Hsys} as
\begin{align}
	\begin{aligned}
		&\xi_1'=\xi_2,\quad
		\xi_2'=\omega\xi_1-\partial_1F(\xi_1^2,\eta_1^2;\mu)\xi_1,\\
		&\eta_1'=\eta_2,\quad
		\eta_2'=s\eta_1-\partial_2F(\xi_1^2,\eta_1^2;\mu)\eta_1,
	\end{aligned}
	\label{eqn:Hsysex}
\end{align}
which is a two-degree-of-freedom Hamiltonian system with Hamiltonian
\begin{align*}
	H(\xi,\eta;\mu)=\frac{1}{2}\bigl(-\omega\xi_1^2-s\eta_1^2+F(\xi_1^2,\eta_1^2;\mu)+\xi_2^2+\eta_2^2\bigr).
\end{align*}
Under \textbf{(A1)} and \textbf{(A2)}, we easily see that \textbf{(M1)}-\textbf{(M4)} and \textbf{(M6)} in Section~3.1 hold, where $\lambda_1=\sqrt\omega$, $\lambda_2=\sqrt s$, $S_1=I_2$, $S_2=-I_2$, and
\begin{align}
	\xi^\h(x)=(U_0(x),U_0'(x))^\T.
	\label{eqn:homo}
\end{align}
Note that $X^+=\{(\xi,\eta)\in\Rset^2\times\Rset^2\mid\eta=0\}$ and $X^-=\{(\xi,\eta)\in\Rset^2\times\Rset^2\mid\xi=0\}$.
For \eqref{eqn:Hsysex} and \eqref{eqn:homo} the variational equation \eqref{eqn:ve} becomes
\begin{align}
	\begin{aligned}
		&
		\xi_1'=\xi_2,\quad
		\xi_2'=\omega\xi_1-\bigl(2\partial_1^2F(U_0(x)^2,0)U_0(x)+\partial_1F(U_0(x)^2,0)\bigr)\xi_1,\\
		&
		\eta_1'=\eta_2,\quad
		\eta_2'=s\eta_1-\partial_2F(U_0(x)^2,0;\mu)\eta_1.
	\end{aligned}
	\label{eqn:veex}
\end{align}
Letting $\eta_1=\delta V$ and $\eta_2=\delta V'$, the $\eta$-component of \eqref{eqn:veex} reduces to \eqref{eqn:VE2}.
When \textbf{(A3)} holds, so do \textbf{(M5)} and \textbf{(M7)} with $\phi(x)=(V_1(x),V_1'(x))^\T$
 and \eqref{eqn:VE2} has no bounded solution which is linearly independent of $V_1(x)$.
Using Theorem~\ref{thm:Mel} and Remark~\ref{rmk:Mel},
 we obtain the desired result with $\bar a_2$ and $\bar b_2$ given in \eqref{eqn:ab2}.

It remains to obtain the expressions \eqref{eqn:a2ex} and \eqref{eqn:b2ex}
 for $\bar a_2$ and $\bar b_2$ as well as
 the approximation \eqref{eqn:bif_UV} for the bifurcated homoclinic solutoions.
Its direct substitution into \eqref{eqn:UV} shows that the expression \eqref{eqn:bif_UV} is valid
 if $\delta U=U_2(x)$ is a bounded solution to the nonhomogeneous linear equation
\begin{align}
	-\delta U''+\omega\delta  U-\bigl(2\partial_1^2F(U_0(x)^2,0)&U_0(x)^2+\partial_1F(U_0(x)^2,0)\bigr)\delta U\notag\\
	=&\partial_1\partial_2F(U_0(x)^2,0;0)U_0(x)V_1(x)^2,
	\label{eqn:aprx1}
\end{align}
such that $\lim_{x\to\pm\infty}U_2(x)=0$.
We easily see that it exists uniquely and is an even function.
So we only have to show that $U_2(x)$ is expressed as \eqref{eqn:U2} for the latter.

On the other hand,
 noting that $\partial_\mu\D_\xi H(\xi,0;0)=0$ for \eqref{eqn:Hsysex} by \textbf{(A1)},
 we obtain $\tilde\xi^\mu(x)=0$.
Since \eqref{eqn:aprx1} is equivalent to \eqref{eqn:xia} for \eqref{eqn:Hsysex},
 we have $\tilde{\xi}^\alpha(x)=(U_2(x),U_2'(x))^\T$.
Let $\Phi(x)=(\phi_{ij}(x))_{1\le i,j\le2}$ be a fundamental matrix solution to the homogeneous part of \eqref{eqn:xia} such that $\Phi(0)=I_2$.
Since $\delta U=U_0'(x)$ is a solution to the homogeneous part of \eqref{eqn:aprx1}, we have
\begin{align}
	\phi_{12}(x)=U_0'(x)/U_0''(0)
	\label{eqn:phi12}
\end{align}
and $U_0''(0)=(\omega-\partial_1F(\zeta_0^2,0))\zeta_0<0$ by Remark~\ref{rmk:U0V1} (i).
Moreover, applying the reduction of order to \eqref{eqn:aprx1}
 and letting $\delta U=U_0'(x)\varphi$, we have
\begin{align*}
	U_0'(x)\varphi''+2U_0''(x)\varphi'=0,
\end{align*}
to which a general solution is given by
\begin{align*}
	\varphi(x)=C_1\int_{x_0}^x\frac{\d y}{U_0'(y)^2}+C_2,
\end{align*}
where $x_0>0,C_1,C_2$ are arbitrary constants and $x>0$.
This yields a general solution to the homogeneous part of \eqref{eqn:aprx1} as
\begin{align}
	\delta U=U_0'(x)\!\left(C_1\int_{x_0}^x\frac{\d y}{U_0'(y)^2}+C_2\right).
	\label{eqn:dU}
\end{align}
Since by l'H\^{o}pital's rule
\begin{align*}
	\lim_{x\to 0}U_0'(x)\int_{x_0}^x\frac{\d y}{U_0'(y)^2}
	=&\lim_{x\to 0}\frac{1}{1/U_0'(x)}\int_{x_0}^x\frac{\d y}{U_0'(y)^2}\\
	=&-\lim_{x\to 0}\frac{1}{U_0''(x)/U_0'(x)^2}\frac{1}{U_0'(x)^2}
	=-\frac{1}{U_0''(0)}
\end{align*}
as well as $\phi_{11}(0)=1$ and $\phi_{11}'(0)=\phi_{21}(0)=0$, we set
\begin{align*}
	C_1=-U_0''(0),\quad
	C_2=\frac{1}{U_0'(x_0)}+\int_0^{x_0}\frac{U_0''(y)-U_0''(0)}{U_0'(y)^2}\,\d y
\end{align*}
in \eqref{eqn:dU} to obtain
\begin{align}
	\phi_{11}(x)=U_0'(x)\!\left(-U_0''(0)\int_{x_0}^x\frac{\d y}{U_0'(y)^2}+\frac{1}{U_0'(x_0)}+\int_0^{x_0}\frac{U_0''(y)-U_0''(0)}{U_0'(y)^2}\,\d y\right)
	\label{eqn:phi11}
\end{align}
for $x>0$ and $\phi_{11}(x)=\phi_{11}(-x)$ for $x<0$.
In particular, $\lim_{x\to\pm\infty}\phi_{11}(x)=-\infty$.
Using the variation of constants formula,
 we have a bounded solution of \eqref{eqn:aprx1} satisfying $\lim_{x\to\pm\infty}\tilde{\xi}_j(x)=0$
 as \eqref{eqn:U2}.
Since
\begin{align*}
	J\partial_\mu\D_\eta^2H(\xi,0;0)=
	\begin{pmatrix}
		0 & 0\\
		-\partial_2\partial_\mu F(\xi_1^2,0;0) & 0
	\end{pmatrix},\quad
	q_3(\xi,\eta)
	=\begin{pmatrix}
		0\\
		-\partial_2^2F(\xi_1^2,0;0)\eta_1^3
	\end{pmatrix}
\end{align*}
and
\begin{align*}
	\tilde{q}_1(x,\xi,\eta)
	=\begin{pmatrix}
		0\\
		-2\partial_1\partial_2F(U_0(x)^2,0;0)U_0(x)\xi_1\eta_1
	\end{pmatrix},
\end{align*}
we express \eqref{eqn:ab2} as \eqref{eqn:a2ex} and \eqref{eqn:b2ex} for \eqref{eqn:Hsysex}.
\end{proof}


\section{Hamiltonian-Krein Index and Evans Function}

In this section, we present the two main tools used in the next two sections,
 the Hamiltonian-Krein index theory and Evans function technique,
 for the stability problem of solitary waves of the form
\begin{align}
	(u,v)=(\e^{\i\omega t}U(x;\omega,s),\e^{\i st}V(x;\omega,s))
	\label{eqn:soliton4}
\end{align}
in the CNLS equations \eqref{eqn:CNLS},
 to which \eqref{eqn:soliton} reduces when $c,x_0,\theta,\phi=0$ as stated in Section~1.
Here $(U,V)=(U(x;\omega,s),V(x;\omega,s))$ is a homoclinic solution to \eqref{eqn:UV} such that $U(x;\omega,s),V(x;\omega,s)\to0$ as $x\to\pm\infty$, where its dependence on $\omega,s$ is explicitly provided although it will be frequently dropped below.
We also fix the value of $\mu$ and suppress the dependence of $F$ and $(U,V)$ on $\mu$.
See \cite{KKS04,KKS05,KaPr13,Pe05} for more details on these methods.
We easily see that there exist positive constants $\rho,C$ such that
\begin{align}
	|U(x)|+|V(x)|\le C\e^{-\rho|x|},\quad x\in\Rset.
	\label{eqn:UV_expdecay}
\end{align}


\subsection{Linearized Operator}\label{sec:lin_op}


We first rewrite \eqref{eqn:CNLS} by means of $(u,v)$ and their complex conjugates $(w,z)\defeq(u^*,v^*)$ as
\begin{align}
	\begin{aligned}
		&\mathrm{i}\partial_tu+\partial_x^2u+\partial_1F(uw,vz)u=0,\quad
		\mathrm{i}\partial_tv+\partial_x^2v+\partial_2F(uw,vz)v=0,\\
		&\mathrm{i}\partial_tw-\partial_x^2w-\partial_1F(uw,vz)w=0,\quad
		\mathrm{i}\partial_tz-\partial_x^2z-\partial_2F(uw,vz)z=0,
	\end{aligned}
	\label{eqn:uwvz}
\end{align}
in which the solitary wave \eqref{eqn:soliton4} in \eqref{eqn:CNLS} corresponds to
\begin{align}
	(u,w,v,z)=(\e^{\i\omega t}U(x),\e^{-\i\omega t}U(x),\e^{\i st}V(x),\e^{-\i st}V(x)).
\label{eqn:soliton3}
\end{align}
Letting
\begin{align*}
	&u=\e^{\i\omega t}(U(x)+\delta\xi_1),\quad w=\e^{-\i\omega t}(U(x)+\delta\xi_2),\\
	&v=\e^{\i st}(V(x)+\delta\xi_3),\quad z=\e^{-\i st}(V(x)+\delta\xi_4),
\end{align*}
we obtain the linearized problem for \eqref{eqn:uwvz} around the solitary wave \eqref{eqn:soliton3} as
\begin{align*}
	\partial_t\delta\xi=\mathscr L\delta\xi,
\end{align*}
where $\delta\xi=(\delta\xi_1,\ldots,\delta\xi_4)^\T$ and the linearized operator $\mathscr L$ is given by \eqref{eqn:lin_op}.
Substituting the separated variables ansatz $\delta\xi(t,x)=\e^{\lambda t}\psi(x)$ into the above linearized problem yields the eigenvalue problem in a standard form,
\begin{align}
	\mathscr L\psi=\lambda\psi,\quad\psi=(\psi_1,\ldots,\psi_4)^\T.
	\label{eqn:eigprob}
\end{align}
By the symmetries \eqref{eqn:transl} and \eqref{eqn:gauge}, $\Ker\mathscr L$ contains
\begin{align}
	\begin{aligned}
		&\varphi_1(x)=(\partial_xU(x;\omega,s),\partial_xU(x;\omega,s),\partial_xV(x;\omega,s),\partial_xV(x;\omega,s))^\T,\\
		&\varphi_2(x)=(\i U(x;\omega,s),-\i U(x;\omega,s),0,0)^\T,\\
		&\varphi_3(x)=(0,0,\i V(x;\omega,s),-\i V(x;\omega,s))^\T.
	\end{aligned}
	\label{eqn:ker}
\end{align}
Moreover, 
\begin{align}
	\begin{aligned}
		&\chi_1(x)=(-\i xU(x;\omega,s)/2,\i xU(x;\omega,s)/2,-\i xV(x;\omega,s)/2,\i xV(x;\omega,s)/2)^\T,\\
		&\chi_2(x)=(\partial_\omega U(x;\omega,s),\partial_\omega U(x;\omega,s),\partial_\omega V(x;\omega,s),\partial_\omega V(x;\omega,s))^\T,\\
		&\chi_3(x)=(\partial_s U(x;\omega,s),\partial_s U(x;\omega,s),\partial_s V(x;\omega,s),\partial_s V(x;\omega,s))^\T
	\end{aligned}
	\label{eqn:gker}
\end{align}
satisfy $\mathscr L\chi_j=\varphi_j$, $j=1,2,3$, whenever they exist.
For the fundamental solitary wave \eqref{eqn:fun_soliton}, 
$\varphi_3(x)$ and $\chi_3(x)$ vanish since $(U,V)=(U_0(x),0)$.
Moreover, by \textbf{(A3)},
 $\psi=(0,0,V_1(x),0)^\T$ and $(0,0,0,V_1(x))^\T$ are contained in $\Ker\mathscr L$ at $\mu=0$.
Hence, 
 $\dim\gKer\mathscr L$ for the bifurcated solitary waves
 does not change after the pitchfork bifurcations
 detected by Theorem~\ref{thm:cnls_bif}.


\subsection{Hamiltonian-Krein Index}\label{sec:KHam}

We introduce the transformation
\begin{align}
	\psi\mapsto((\psi_1+\psi_2)/2,(\psi_3+\psi_4)/2,(\psi_1-\psi_2)/2\i,(\psi_3-\psi_4)/2\i)^\T.
	\label{eqn:to_reim}
\end{align}
Under this transformation, the eigenvalue problem \eqref{eqn:eigprob} is rewritten as
\begin{align}
	\mathcal J\mathcal L\psi=\lambda\psi,\quad\psi=(\psi_1,\ldots,\psi_4)^\T,
	\label{eqn:eigprob_ri}
\end{align}
where
\begin{align}
	\mathcal J=\begin{pmatrix}O_2&I_2\\-I_2&O_2\end{pmatrix},\quad
	\mathcal L=\begin{pmatrix}\mathcal L_+&O_2\\O_2&\mathcal L_-\end{pmatrix}
	\label{eqn:JL_def}
\end{align}
with
\begin{align}
	\begin{aligned}
		&\mathcal L_-=\begin{pmatrix}-\partial_x^2+\omega-\partial_1F(U^2,V^2)&0\\0&-\partial_x^2+s-\partial_2F(U^2,V^2)\end{pmatrix},
		\\
		&\mathcal L_+=\mathcal L_--\begin{pmatrix}2\partial_1^2F(U^2,V^2)U^2&2\partial_1\partial_2F(U^2,V^2)UV\\2\partial_1\partial_2F(U^2,V^2)UV&2\partial_2^2F(U^2,V^2)V^2\end{pmatrix}.\\
	\end{aligned}
	\label{eqn:L_pm}
\end{align}
We easily see that the spectrum does not change under this transformation.
Moreover, \eqref{eqn:ker} and \eqref{eqn:gker} are transformed into
\begin{align*}
	(U'(x),V'(x),0,0)^\T,\quad
	(0,0,U(x),0)^\T,\quad
	(0,0,0,V(x))^\T.
\end{align*}
and
\begin{align*}
	(0,0,-xU(x)/2,-xV(x)/2)^\T,\ \ 
	(\partial_\omega U(x),\partial_\omega V(x),0,0)^\T,\ \ 
	(\partial_s U(x),\partial_s V(x),0,0)^\T,
\end{align*}
respectively.
Let $\mathcal L_{\pm ij}$ be the $(i,j)$-components of $\mathcal L_\pm$ for $i,j=1,2$.
Note that \textbf{(A3)} holds if and only if $\mathcal L_{\pm 22}$
 for the fundamental solitary wave \eqref{eqn:fun_soliton} have a nontrivial kernel.

Define a $3\times 3$ real matrix $D$ with entries
\begin{align*}
	D_{ij}
	\defeq\langle\chi_i,\i\Sigma_3\varphi_j\rangle_{L^2}
	=\langle\chi_i,\i\Sigma_3\mathscr L\chi_j\rangle_{L^2}
\end{align*}
for $i,j=1,2,3$, i.e.,
\begin{align}
	D=\begin{pmatrix}
		(\|U\|_{L^2}^2+\|V\|_{L^2}^2)/2&0&0\\
		0&-\partial_\omega\|U\|_{L^2}^2&-\partial_\omega\|V\|_{L^2}^2\\
		0&-\partial_s\|U\|_{L^2}^2&-\partial_s\|V\|_{L^2}^2
	\end{pmatrix}.\label{eqn:D}
\end{align}
Since $\i\Sigma_3\mathscr L$ is self-adjoint, we see that $D$ is symmetric, so that $\partial_\omega\|V\|_{L^2}^2=\partial_s\|U\|_{L^2}^2$.
Let $\mathrm n(\mathcal A)$ denote the number of negative eigenvalues
 for a self-adjoint operator $\mathcal A$.


\begin{dfn}
	Let $k_\mathrm r$ be the number of real positive eigenvalues of $\mathcal J\mathcal L$, and let $k_\mathrm c$ be the number of its complex eigenvalues in the first open quadrant of $\Cset$.
	Let $H^\lambda$ be the linear map induced by the bilinear form $\langle v,\mathcal Lw\rangle_{L^2}$ restricted to the generalized eigenspace $\gKer(\mathcal J\mathcal L-\lambda)$, and let $k_\mathrm i^-$ be the sum of $\mathrm n(H^\lambda)$ when $\lambda$ runs through all purely imaginary eigenvalues of $\mathcal J\mathcal L$ with positive imaginary parts.
	Then the sum $K_\mathrm{Ham}=k_\mathrm r+2k_\mathrm c+2k_\mathrm i^-$ is called the \emph{Hamiltonian-Krein index}.
	Moreover, a purely imaginary, nonzero eigenvalue $\lambda$ of $\mathcal J\mathcal L$ is said to have a \emph{positive} $($resp.\ \emph{negative}$)$ \emph{Krein signature} if $H^\lambda$ is positive $($resp.\ negative$)$ definite on $\gKer(\mathcal J\mathcal L-\lambda)$.
\end{dfn}

We have the following theorem, which is frequently referred to as the index theorem.

\begin{thm}[Kapitula et al.\ \cite{KKS04,KKS05}, Pelinovsky \cite{Pe05}]\label{thm:KHam}
	If $D$ is nonsingular, then $K_\mathrm{Ham}=\mathrm n(\mathcal L)-\mathrm n(D)$.
\end{thm}

Using Theorem \ref{thm:KHam} and the Stability and Instability Theorems of \cite{GSS90}
 (see also \cite{KaPr13}), we obtain the following result. 

\begin{thm}
\label{thm:orb}
Suppose that $D$ is nonsingular.
If $K_\mathrm{Ham}=0$ (resp.\ $K_\mathrm{Ham}$ is odd),
 then the solitary wave \eqref{eqn:soliton4} is orbitally stable (resp.\ unstable).
\end{thm}
\begin{rmk}\label{rmk:D0}
	For the fundamental solitary wave \eqref{eqn:fun_soliton} we have $\varphi_3(x),\chi_3(x)\equiv 0$.
	Replacing \eqref{eqn:D} with
	\begin{align}
		D=\begin{pmatrix}
			\|U_0\|_{L^2}^2/2&0\\
			0&-\partial_\omega\|U_0\|_{L^2}^2
		\end{pmatrix},\label{eqn:D2}
	\end{align}
	the statements of Theorems~\ref{thm:KHam} and \ref{thm:orb} also hold for \eqref{eqn:fun_soliton}.
\end{rmk}


\subsection{Evans Functions}\label{sec:Evans}

We rewrite the eigenvalue problem \eqref{eqn:eigprob} as a system of first-order ODEs,
\begin{align}
	Y'=A(x,\lambda)Y,
	\label{eqn:eigprob2}
\end{align}
where $Y=(\psi_1,\ldots,\psi_4,\psi_1',\ldots,\psi_4')^\T$.
Define
\begin{align}
	A_\infty(\lambda)\defeq\lim_{x\to\pm\infty}A(x,\lambda)
	=\left(
		\begin{array}{c|c}
			O_4&I_4\\\hline
			\begin{matrix}\omega-\i\lambda&0&0&0\\0&\omega+\i\lambda&0&0\\0&0&s-\i\lambda&0\\0&0&0&s+\i\lambda\end{matrix}&O_4
		\end{array}
	\right).
	\label{eqn:A_infty}
\end{align}
The matrix $A_\infty(\lambda)$ has eigenvalues
 given by \eqref{eqn:mu_def1}, \eqref{eqn:mu_def2}, and
\[
\nu_{j+4}(\lambda)=-\nu_j(\lambda),\quad j=1,\ldots,4.
\]
Let
\begin{align*}
	\Omega=\Omega_1\cap\Omega_2\cap\Omega_3\cap\Omega_4,
\end{align*}
where
\begin{equation}
\begin{split}
	&\Omega_1=\{\lambda\in\Cset\mid\Re\lambda>0\ \text{or}\ \Im\lambda>-\omega\},\\
	&\Omega_2=\{\lambda\in\Cset\mid\Re\lambda>0\ \text{or}\ \Im\lambda<\omega\},\\
	&\Omega_3=\{\lambda\in\Cset\mid\Re\lambda>0\ \text{or}\ \Im\lambda>-s\},\\
	&\Omega_4=\{\lambda\in\Cset\mid\Re\lambda>0\ \text{or}\ \Im\lambda<s\},
\end{split}
\label{eqn:Omega}
\end{equation}
and let
\begin{align*}
	\widehat\Omega=\Cset\setminus\{\lambda\in\Cset\mid\Re\lambda\le0\ \text{and}\ \Im\lambda\in\{\pm\omega,\pm s\}\}.
\end{align*}
Note that $\nu_j(\lambda)$, $j=1,\ldots,4$, are analytic in $\widehat\Omega$,
 by the definition of their branch cut \eqref{eqn:br_cut}.
Moreover, $\nu_j(\lambda)$ (resp.\ $\nu_{j+4}(\lambda)$) has a positive (resp.\ negative) real part in $\Omega_j$ for $j=1,\ldots,4$.
Let
\begin{align}
	\Omega_\mathrm{e}=\left\{\lambda\in\widehat\Omega\;\middle|\;\min_{1\le j\le4}\Re\nu_j(\lambda)>-\frac{\rho}{2}\right\},
	\label{eqn:Omega_e}
\end{align}
which is an open neighborhood of $\Omega$ in $\widehat\Omega$, where $\rho$ is given in \eqref{eqn:UV_expdecay}.
The regions $\Omega$ and $\Omega_\mathrm{e}$ are displayed as the shaded areas in Figs.\ \ref{fig:Omega} (a) and (b), respectively.
We have the following lemma (see, e.g., Theorem 9.2.3 of \cite{KaPr13} and Proposition 1.2 of \cite{PeWe92}, for the proof).

\begin{figure}
	\begin{minipage}{0.495\textwidth}
		\includegraphics[scale=0.75]{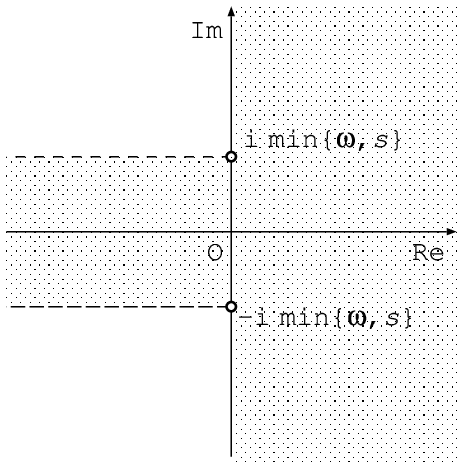}
		\begin{center}
			(a) $\Omega$
		\end{center}
	\end{minipage}
	\begin{minipage}{0.495\textwidth}
		\includegraphics[scale=0.75]{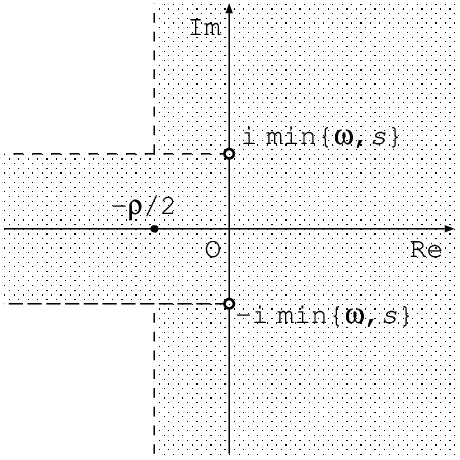}
		\begin{center}
			(b) $\Omega_\mathrm{e}$
		\end{center}
	\end{minipage}
	\caption{
		Regions $\Omega$ and $\Omega_\mathrm{e}$.
		\label{fig:Omega}
	}
\end{figure}

\begin{lem}\label{lem:Jost_sol}
	One can take an eigenvector $v_j(\lambda)$ of $A_\infty(\lambda)$ for the eigenvalue $\nu_j(\lambda)$ such that it is analytic in $\Omega_\e$ and continuous at the branch points $\lambda=\pm\i\omega,\pm\i s$ for $j=1,\ldots,8$.
	Moreover, for $j=1,\ldots,8$, there exist two solutions $Y_j^\pm(x,\lambda)$ to \eqref{eqn:eigprob2} and a positive constant $C_j(\lambda)$ such that
	\begin{align}
		\begin{aligned}
			&|\e^{-\nu_j(\lambda)x}Y_j^-(x,\lambda)-v_j(\lambda)|\le C_j(\lambda)\e^{-\rho|x|/2}&\text{as $x\to-\infty$},\\
			&|\e^{-\nu_j(\lambda)x}Y_j^+(x,\lambda)-v_j(\lambda)|\le C_j(\lambda)\e^{-\rho|x|/2}&\text{as $x\to+\infty$},
		\end{aligned}
		\label{eqn:Y_asymp}
	\end{align}
	and $Y_j^\pm(x,\lambda)$ are analytic with respect to $\lambda\in\Omega_{\mathrm e}$ and continuous at the branch points for any fixed $x\in\Rset$, where $C_j(\lambda)$ is bounded on any compact subset of $\Omega_\e$.
\end{lem}

We refer to $Y_j^\pm(x,\lambda)$, $j=1,\ldots,8$, as the \emph{Jost solutions} for \eqref{eqn:eigprob2}.
Note that for \eqref{eqn:eigprob2}, both of $Y_j^-(x,\lambda)$, $j=1,\ldots,4$, and $Y_j^+(x,\lambda)$, $j=5,\ldots,8$, give bases of the spaces of solutions which decay as $x\to-\infty$ and $x\to+\infty$, respectively, when $\lambda\in\Omega$.
We define the \emph{Evans function} $E(\lambda)$ for the eigenvalue problem \eqref{eqn:eigprob} as
\begin{align}
	E(\lambda)=\det\bigl(Y_1^-\;\;Y_2^-\;\;Y_3^-\;\;Y_4^-\;\;Y_5^+\;\;Y_6^+\;\;Y_7^+\;\;Y_8^+\bigr)(x,\lambda),
	\label{eqn:def_evans}
\end{align}
when $\lambda\in\Omega_\e\cup\{\pm\i\omega,\pm\i s\}$.
Since $\trace A(x,\lambda)=0$, the right hand side of \eqref{eqn:def_evans} is independent of $x$.
Although the Jost solutions are not uniquely determined in general, we easily see that zeros of $E(\lambda)$ do not depend on a particular choice of them.
We have the following fundamental properties of the Evans function (see, e.g., Theorem 10.2.2 of \cite{KaPr13}, for the proof).

\begin{prop}\label{prop:ev_prop}
	The Evans function $E(\lambda)$ is analytic in $\Omega_\e$ and continuous at the branch points $\lambda=\pm\i\omega,\pm\i s$.
	When $\lambda_0\in\Omega\cup\sigma_\mathrm{ess}(\mathscr L)$, it is sufficient for $E(\lambda_0)=0$ that $\lambda_0$ is an eigenvalue of $\mathscr L$.
	When $\lambda_0\in\Omega$, the condition is also necessary.
	Moreover, the algebraic multiplicity of the eigenvalue equals the multiplicity $d$ of the zero for $E(\lambda)$, for which
	\begin{align*}
		E(\lambda_0)=\partial_\lambda E(\lambda_0)=\cdots=\partial_\lambda^{d-1}E(\lambda_0)=0,\quad
		\partial_\lambda^dE(\lambda_0)\ne0.
	\end{align*}
\end{prop}

Note that the Evans function $E(\lambda)$ may not be analytic at the branch points.
Near the branch points, we modify the Evans function as follows.
Let $\lambda_\mathrm{br}=\i\omega$ or $\i s$.
Substituting the relation $\lambda=\lambda_\mathrm{br}-\i\gamma^2$ into \eqref{eqn:eigprob2},
 we have
\begin{align*}
	Y'=A(x,\lambda_\mathrm{br}-\i\gamma^2)Y.
\end{align*}
All eigenvalues of $A_\infty(\lambda_\mathrm{br}-\i\gamma^2)$ are analytic near $\gamma=0$.
We use $\gamma$ instead of $\lambda$
 and apply the above arguments to obtain the Jost solutions and Evans function.
We write this Evans function as $\widetilde E(\gamma)$ near $\gamma=0$
 and call it the \emph{extended Evans function}. 
Especially, $\widetilde E(\gamma)$ is analytic on a Riemann surface
 that is a two-sheeted cover of $\Cset$.
For $\lambda_\mathrm{br}=-\i\omega$ or $-\i s$, using the transformation $\gamma^2=-\i(\lambda-\lambda_\mathrm{br})$, we define the extended Evans function similarly. 
So we immediately obtain the following proposition.

\begin{prop}\label{prop:ev_gamma_prop}
	The extended Evans function $\widetilde E(\gamma)$ is analytic
	near $\gamma=0$.
	Let $\lambda_\mathrm{br}=\i\omega$ or $\i s$.
	Suppose that 
	$\lambda=\lambda_\mathrm{br}-\i\gamma^2\in\Omega_\e$.
	Then it is necessary and sufficient for $E(\lambda)=0$
	that $\widetilde E(\gamma)=0$ and $-\pi/4<\arg\gamma<3\pi/4$.
	Moreover, if $\lambda=\lambda_\mathrm{br}-\i\gamma^2\in\Omega$,
	then the condition is also necessary and sufficient for $\lambda$
	to be an eigenvalue of $\mathscr L$.
	For $\lambda_\mathrm{br}=-\i\omega$ or $-\i s$, similar statements also hold.
\end{prop}


%

We easily see that \eqref{eqn:eigprob} has a bounded non-decaying solution only if $\lambda\in\sigma_\mathrm{ess}(\mathscr L)$ is a resonance pole.


\subsection{Parameter-Dependent Solitary Waves}\label{sec:Evans2}

We now consider the case in which the homoclinic solution to \eqref{eqn:UV} smoothly depends on a small parameter $\epsilon$ and write the linearized operator $\mathscr L$ as $\mathscr L_\epsilon$.
Let $E(\lambda,\epsilon)$ and $\widetilde E(\gamma,\epsilon)$,
 respectively,  denote the original and extended Evans functions.
From Propositions~\ref{prop:ev_prop} and \ref{prop:ev_gamma_prop}
 we obtain the following.
 
\begin{prop}\label{prop:E_form}
	Suppose that the Evans function $E(\lambda,0)$ has a simple zero at $\lambda=\lambda_0\in\Omega_\e$.
	Then there exists a unique smooth function $\lambda(\epsilon)$ near $\epsilon=0$ such that $E(\lambda(\epsilon),\epsilon)=0$ and $\lambda(0)=\lambda_0$.
	Moreover,
	\begin{equation}
		\begin{split}
		&
		\partial_\epsilon\lambda(0)
		=-\frac{\partial_\epsilon E(\lambda_0,0)}{\partial_\lambda E(\lambda_0,0)},\\
		&
		\partial_\epsilon^2\lambda(0)
		=-\frac{\partial_\lambda^2E(\lambda_0,0)\partial_\epsilon\lambda(0)^2
		+2\partial_\lambda\partial_\epsilon E(\lambda_0,0)\partial_\epsilon\lambda(0)
		+\partial_\epsilon^2E(\lambda_0,0)}{\partial_\lambda E(\lambda_0,0)}.
		\end{split}
		\label{eqn:dlambda}
	\end{equation}
	If $\lambda(\epsilon)\in\Omega$, then it is an eigenvalue of $\mathscr L_\epsilon$.
\end{prop}

\begin{proof}
	Applying the implicit function theorem to $E(\lambda,\epsilon)=0$,
	we obtain the first and second parts.
	The last part immediately follows from Proposition \ref{prop:ev_prop}.
\end{proof}

\begin{prop}\label{prop:gamma_move}
	Let $\lambda_\mathrm{br}=\i\omega$ or $\i s$ (resp.\ $-\i\omega$ or $-\i s$).
	Suppose that $\widetilde E(\gamma,0)$ has a simple zero at $\gamma=0$.
	Then there exists a unique smooth function $\gamma(\epsilon)$ near $\epsilon=0$ such that $\widetilde E(\gamma(\epsilon),\epsilon)=0$ and $\gamma(0)=0$.
	Moreover,
	\begin{align}
		\begin{aligned}
			&\partial_\epsilon\gamma(0)
			=-\frac{\partial_\epsilon\widetilde E(0,0)}{\partial_\gamma\widetilde E(0,0)},\\
			&\partial_\epsilon^2\gamma(0)
			=-\frac{\partial_\gamma^2\widetilde E(0,0)\partial_\epsilon\gamma(0)^2
			+2\partial_\gamma\partial_\epsilon\widetilde E(0,0)\partial_\epsilon\gamma(0)
			+\partial_\epsilon^2\widetilde E(0,0)}{\partial_\gamma\widetilde E(0,0)}.
		\end{aligned}
		\label{eqn:dgamma}
	\end{align}
	If $\Re\gamma(\epsilon)>0$ and $\Im\gamma(\epsilon)\ge0$ (resp.\ $\Re\gamma(\epsilon)>0$ and $\Im\gamma(\epsilon)\le0$), then $\lambda_\mathrm{br}-\i\gamma(\epsilon)^2$ (resp.\ $\lambda_\mathrm{br}+\i\gamma(\epsilon)^2$) is a zero of $E(\lambda,\epsilon)$.
	Especially, if $\Re\gamma(\epsilon)>0$ and $\Im\gamma(\epsilon)>0$ (resp.\ $\Re\gamma(\epsilon)>0$ and $\Im\gamma(\epsilon)<0$), then the zero is an eigenvalue of $\mathscr L_\epsilon$ with a positive real part.
	If $\Re\gamma(\epsilon)<0$ or $\Im\gamma(\epsilon)<0$ (resp.\ $\Re\gamma(\epsilon)<0$ or $\Im\gamma(\epsilon)>0$), then $\mathscr L_\epsilon$ has no eigenvalue near $\lambda=\lambda_\mathrm{br}$.
\end{prop}

\begin{proof}
	The first and second parts immediately follow from the implicit function theorem.
	Let $\lambda_\mathrm{br}=\i\omega$ or $\i s$ and let $\lambda(\epsilon)=\lambda_\mathrm{br}-\i\gamma(\epsilon)^2$.
	We prove the remaining parts.
	Since $\Re\lambda(\epsilon)=2\,\Re\gamma(\epsilon)\,\Im\gamma(\epsilon)$ and $\Im\lambda(\epsilon)=\Im\lambda_\mathrm{br}-(\Re\gamma(\epsilon))^2+(\Im\gamma(\epsilon))^2$, we see that if $\Re\gamma(\epsilon)>0$ and $\Im\gamma(\epsilon)\ge0$, then $\lambda(\epsilon)$ is a zero of $E(\lambda,\epsilon)$ by Proposition \ref{prop:ev_gamma_prop}.
	On the other hand, assume that $\Re\gamma(\epsilon)<0$ or $\Im\gamma(\epsilon)<0$ and that $\mathscr L_\epsilon$ has an eigenvalue $\lambda_1(\epsilon)$ near $\lambda=\lambda_\mathrm{br}$.
	By the symmetry of $\sigma(\mathscr L_\epsilon)$, we can assume $\Re\lambda_1(\epsilon)\ge0$.
	Then $\lambda_1(\epsilon)$ is a zero of $E(\lambda,\epsilon)$ by Proposition \ref{prop:ev_prop}.
	However, $E(\lambda,\epsilon)$ does not have such a zero by Proposition \ref{prop:ev_gamma_prop}.
	This yields a contradiction so that $\mathscr L_\epsilon$ has no eigenvalue near $\lambda=\lambda_\mathrm{br}$.
	Thus, we complete the proof.
\end{proof}


\section{Spectral Stability of the Fundamental Solitary Wave}\label{sec:fun_sol_stab}

In this section, we compute the location of eigenvalues of the linearized operator \eqref{eqn:lin_op}
 around the fundamental solitary wave \eqref{eqn:fun_soliton} by using the techniques of Section~4,
 and give a proof of Theorem~\ref{thm:fun_stab}.



First we fix $\mu$ at some value which is not necessarily zero, and consider the fundamental solitary wave \eqref{eqn:fun_soliton}.
The linearized operator \eqref{eqn:lin_op} becomes
\begin{align}
	\mathscr L_0=&-\i\Sigma_3\left[-\partial_x^2I_4+\begin{pmatrix}\omega I_2&O_2\\O_2&sI_2\end{pmatrix}\right.\nonumber\\
	&\left.-\begin{pmatrix}\partial_1^2F(U_0^2,0)U_0^2(I_2+\sigma_1)+\partial_1F(U_0^2,0)I_2&O_2\\O_2&\partial_2F(U_0^2,0;\mu)I_2\end{pmatrix}\right].
	\label{eqn:L0}
\end{align}
The coefficient matrix $A(x,\lambda)$ in \eqref{eqn:eigprob2} becomes
\begin{align}
	A_0(x,\lambda)=
	\left(
		\begin{array}{c|c}
			O_4&I_4\\\hline
			\begin{matrix}A_{01}(x,\lambda)&O_2\\O_2&A_{02}(x,\lambda)\end{matrix}&O_4
		\end{array}
	\right)
	\label{eqn:A0}
\end{align}
with
\begin{align*}
	&A_{01}(x,\lambda)=\omega I_2-\partial_1^2F(U_0^2,0)U_0^2(I_2+\sigma_1)-\partial_1F(U_0^2,0)I_2-\i\lambda\sigma_3,\\
	&A_{02}(x,\lambda)=sI_2-\partial_2F(U_0^2,0;\mu)I_2-\i\lambda\sigma_3.
\end{align*}
As stated in Section~2,
the eigenvalue problem \eqref{eqn:eigprob} is separated into three parts: (A) the first and second components; (B) the fourth component; and (C) the third component.
In the following we treat the three parts separately and obtain expressions of the Jost functions and Evans functions given by \eqref{eqn:def_evans}.

\subsubsection*{\rm\bf (A)\nopunct}
The first and second components of \eqref{eqn:eigprob} become \eqref{eqn:lin_eq_1},
which are the same as the well-understood eigenvalue problem for $u=\e^{\i\omega t}U_0(x)$ in the single nonlinear Schr\"odinger equation with the nonlinearity $\partial_1F(|u|^2,0)u$.
We summarize fundamental spectral properties of \eqref{eqn:lin_eq_1} as follows
(see, e.g., Section~5.3.1 of \cite{Ya10} for the proof).

\begin{prop}\label{prop:evp_A}
	All eigenvalues of \eqref{eqn:lin_eq_1} are either real or purely imaginary.
	The zero eigenvalue is of geometric multiplicity two and of algebraic multiplicity $\ge4$,
	where the equality holds if and only if $\partial_\omega\|U_0\|_{L^2}^2\ne0$.
	Moreover, \eqref{eqn:lin_eq_1} has a pair of positive and negative eigenvalues with the same modulus if $\partial_\omega\|U_0\|_{L^2}^2<0$ and has no nonzero real eigenvalues if $\partial_\omega\|U_0\|_{L^2}^2>0$.
\end{prop}

Let $(p_j,q_j)$, $j=1,2,5,6$, be the Jost solutions to \eqref{eqn:lin_eq_1}
 satisfying \eqref{eqn:caseA_asym}, as in Section~2.2.
So $(p_j,q_j,p_j',q_j')$, $j=1,2,5,6$, are the Jost solutions
 for the associated components of \eqref{eqn:eigprob2}  satisfying \eqref{eqn:caseA_asym}.
Note that
\begin{align*}
	(1,0,\nu_1,0)^\T,\quad
	(0,1,0,\nu_2)^\T,\quad
	(1,0,-\nu_1,0)^\T,\quad
	(0,1,0,-\nu_2)^\T
\end{align*}
are eigenvectors of
\begin{align*}
	\begin{pmatrix}O_2&I_2\\{\displaystyle\lim_{x\to\pm\infty}A_{01}(x,\lambda)}&O_2\end{pmatrix}
\end{align*}
for the eigenvalues $\nu_1,\nu_2,-\nu_1$, and $-\nu_2$, respectively.
In addition, the Jost solutions are not necessarily unique.
We take a particular choice in Section~6.
The contribution of this part to the Evans function is written as $E_\mathrm{A}(\lambda)$
 given by \eqref{eqn:E1}.
Note that $E_\mathrm{A}(\lambda)$ does not depend on $\mu$.

\subsubsection*{\rm\bf(B)\nopunct}
The fourth component of \eqref{eqn:eigprob} becomes \eqref{eqn:lin_eq_3}.
Let $p_j$, $j=4,8$, be the Jost solutions \eqref{eqn:lin_eq_3}
 satisfying \eqref{eqn:caseC_asym}, as in Section~2.2.
So $(p_j,p_j')$, $j=4,8$, are the Jost solutions
 for the associated components of \eqref{eqn:eigprob2} satisfying \eqref{eqn:caseC_asym}.
For $\lambda\in\Omega_4$ (see \eqref{eqn:Omega}) the Jost solutions are unique since $\Re\nu_4(\lambda)>0$.
The contribution of this part to the Evans function is represented
 as $E_\mathrm{B}(\lambda)$ defined in \eqref{eqn:E3}.
By \textbf{(A3)}, the zero eigenvalue of \eqref{eqn:lin_eq_3} is simple at $\mu=0$,
 so that $\lambda=0$ is a simple zero of $E_\mathrm{B}(\lambda)$.

\subsubsection*{\rm\bf(C)\nopunct}
The third component of \eqref{eqn:eigprob} becomes \eqref{eqn:lin_eq_2},
 which has the same form as \eqref{eqn:lin_eq_3} by replacing $\lambda$ with $-\lambda$.
Hence, if $\lambda$ is an eigenvalue of \eqref{eqn:lin_eq_3}, then $-\lambda$ is an eigenvalue of \eqref{eqn:lin_eq_2}.
Let $p_j$, $j=3,7$, be the Jost solutions \eqref{eqn:lin_eq_2}
 satisfying \eqref{eqn:caseB_asym}, as in Section~2.2.
So $(p_j,p_j')$, $j=3,7$, are the Jost solutions
 for the associated components of \eqref{eqn:eigprob2} satisfying \eqref{eqn:caseB_asym}.
The contribution of this part to the Evans function is represented
 as $E_\mathrm{C}(\lambda)$ defined in \eqref{eqn:E2}.

\subsubsection*{\indent\nopunct}\hspace*{-0.8em}
Thus, the Evans function $E(\lambda,0)$ for $\mathscr L_0$ is represented
 as \eqref{eqn:E_unperturbed},
 and the Jost solutions $Y_{0j}^-(x,\lambda)$, $j=1,\ldots,4$,
 and $Y_{0j}^+(x,\lambda)$, $j=5,\ldots,8$, for $\mathscr L_0$ are given by
\begin{align}
	\begin{aligned}
		Y_{0j}^-(x,\lambda)&=(p_j,q_j,0,0,p_j',q_j',0,0)^\T(x,\lambda),\quad j=1,2,\\
		Y_{03}^-(x,\lambda)&=(0,0,p_3,0,0,0,p_3',0)^\T(x,\lambda),\\
		Y_{04}^-(x,\lambda)&=(0,0,0,p_4,0,0,0,p_4')^\T(x,\lambda),\\
		Y_{0j}^+(x,\lambda)&=(p_j,q_j,0,0,p_j',q_j',0,0)^\T(x,\lambda),\quad j=5,6,\\
		Y_{07}^+(x,\lambda)&=(0,0,p_7,0,0,0,p_7',0)^\T(x,\lambda),\\
		Y_{08}^+(x,\lambda)&=(0,0,0,p_8,0,0,0,p_8')^\T(x,\lambda).
	\end{aligned}
	\label{eqn:Jost_0}
\end{align}
The other Jost solutions $Y_{0j}^+$, $j=1,\ldots,4$, and $Y_{0j}^-$, $j=5,\ldots,8$, are assumed to have the same forms as $Y_{0j}^-$, $j=1,\ldots,4$, and $Y_{0j}^+$, $j=5,\ldots,8$, respectively, in \eqref{eqn:Jost_0}.
So the Jost solutions $Y_{0j}^\pm(x,\lambda)$, $j=1,\ldots,8$, satisfy \eqref{eqn:Y_asymp}
 with the eigenvectors,
\begin{align*}
	&v_1=(1,0,0,0,\nu_1,0,0,0)^\T,\quad
	v_2=(0,1,0,0,0,\nu_2,0,0)^\T,\\
	&v_3=(0,0,1,0,0,0,\nu_3,0)^\T,\quad
	v_4=(0,0,0,1,0,0,0,\nu_4)^\T,\\
	&v_5=(1,0,0,0,-\nu_1,0,0,0)^\T,\quad
	v_6=(0,1,0,0,0,-\nu_2,0,0)^\T,\\
	&v_7=(0,0,1,0,0,0,-\nu_3,0)^\T,\quad
	v_8=(0,0,0,1,0,0,0,-\nu_4)^\T,
\end{align*}
of $A_\infty(\lambda)$ given by \eqref{eqn:A_infty}.
We now prove Theorem~\ref{thm:fun_stab}.

\begin{proof}[Proof of Theorem~\ref{thm:fun_stab}]
	From the above arguments in parts (A), (B), and (C)
	we see that all eigenvalues of $\mathscr L_0$ are either real or purely imaginary.
	Moreover, by Proposition~\ref{prop:evp_A},
	the number of positive eigenvalues of $\mathscr L_0$ is zero or one,
	so that  the fundamental solitary wave \eqref{eqn:fun_soliton} is spectrally stable or unstable,
	depending on whether $\partial_\omega\|U_0\|_{L^2}^2$ is positive or negative.

	It remains to prove the orbital stability and instability of \eqref{eqn:fun_soliton}.
	We define $\mathcal L_\pm$ and $D$ for \eqref{eqn:fun_soliton}
	as in \eqref{eqn:L_pm} and \eqref{eqn:D2}, respectively, i.e.,
	\begin{align*}
		&\mathcal L_-=
		\begin{pmatrix}
			-\partial_x^2+\omega-\partial_1F(U_0^2,0) & 0\\
			0 & -\partial_x^2+s-\partial_2F(U_0^2,0;\mu)
		\end{pmatrix},\\
		&\mathcal L_+=
		\begin{pmatrix}
			-\partial_x^2+\omega-\partial_1F(U_0^2,0)-2\partial_1^2F(U_0^2,0)U_0^2 & 0\\
			0 & -\partial_x^2+s-\partial_2F(U_0^2,0;\mu)
		\end{pmatrix},\\
		&D=\begin{pmatrix}
			\|U_0\|_{L^2}^2/2&0\\
			0&-\partial_\omega\|U_0\|_{L^2}^2
		\end{pmatrix}.
	\end{align*}
	Let $\mathrm n(A)$ be the number of negative eigenvalues for a self-adjoint operator $A$.
	Since
	\begin{align*}
		&\mathcal L_{-11}U_0=(-\partial_x^2+\omega-\partial_1F(U_0^2,0))U_0=0,\\
		&\mathcal L_{+11}U_0'
		=(-\partial_x^2+\omega-\partial_1F(U_0^2,0)-2\partial_1^2F(U_0^2,0)U_0^2)U_0'=0,
	\end{align*}
	the Sturm-Liouville theory (e.g., Section~2.3.2 of \cite{KaPr13}) says that
		\begin{align}
		\begin{aligned}
		\mathrm n(\mathcal L_{-11})=0,\quad
		\mathrm n(\mathcal L_{+11})=1.
		\end{aligned}
		\label{eqn:KHindex_1}
	\end{align}
	Recall that $\mathcal L_{\pm ij}$ is the $(i,j)$-components of $\mathcal L_\pm$ for $i,j=1,2$.
	On the other hand, for $s>0$ sufficiently large, by Remark~\ref{rmk:U0V1} (i) we see that
	\begin{align*}
		\inf_{\substack{\psi\in H^2\\\|\psi\|_{L^2}=1}}\langle\mathcal L_{\pm22}\psi,\psi\rangle_{L^2}
		\ge s-\sup_{0\le\zeta\le\zeta_0}|\partial_2F(\zeta^2,0;\mu)|>0,
	\end{align*}
	which implies
	\begin{align}
		\mathrm n(\mathcal L_{\pm22})=0.
		\label{eqn:KHindex_2}
	\end{align}
	From \eqref{eqn:KHindex_1} and \eqref{eqn:KHindex_2} we obtain $K_\mathrm{Ham}=\mathrm n(\mathcal L_-)+\mathrm n(\mathcal L_+)-\mathrm n(D)=0$ if $\partial_\omega\|U_0\|_{L^2}^2>0$ and $K_\mathrm{Ham}=1$ if $\partial_\omega\|U_0\|_{L^2}^2<0$.
	Using Theorems \ref{thm:KHam} and \ref{thm:orb} along with Remark \ref{rmk:D0},
	we obtain the desired result.
\end{proof}

\begin{rmk}\label{rmk:krein_sig}
	Let $\lambda_0\in\i\Rset$ be a simple eigenvalue of $\mathscr L_0$ which is a zero of $E_\mathrm{B}(\lambda)$.
	We compute the Krein signature of $\lambda_0$.
	Let $\Psi\in L^2(\Rset)$ be the corresponding eigenfunction of \eqref{eqn:lin_eq_3} so that $(0,0,0,\Psi)^\T$ is an eigenfunction of $\mathscr L_0$ associated to $\lambda_0$.
	Using the transformation \eqref{eqn:to_reim} we see that $(0,\Psi/2,0,\i\Psi/2)^\T\in\Ker(J\mathcal L_0-\lambda_0)$, where $J\mathcal L_0$ is the linearized operator \eqref{eqn:JL_def} around the fundamental solitary wave \eqref{eqn:fun_soliton}.
	Hence, it follows from
	\begin{align*}
		&\left\langle\begin{pmatrix}0\\\Psi/2\\0\\\i\Psi/2\end{pmatrix},\mathcal L_0\begin{pmatrix}0\\\Psi/2\\0\\\i\Psi/2\end{pmatrix}\right\rangle
		=\left\langle\begin{pmatrix}0\\\Psi/2\\0\\\i\Psi/2\end{pmatrix},-J^2\mathcal L_0\begin{pmatrix}0\\\Psi/2\\0\\\i\Psi/2\end{pmatrix}\right\rangle\\
		&=\lambda_0\left\langle\begin{pmatrix}0\\\Psi/2\\0\\\i\Psi/2\end{pmatrix},J\begin{pmatrix}0\\\Psi/2\\0\\\i\Psi/2\end{pmatrix}\right\rangle
		=\frac{\Im\lambda_0}{2}\|\Psi\|_{L^2}^2
	\end{align*}
	that $\lambda_0$ has a positive (resp.\ negative) Krein signature if it belongs to the upper (resp.\ lower) half plane.
\end{rmk}


\section{Spectral Stability of the Bifurcated Solitary Waves}

We next consider the bifurcated solitary waves \eqref{eqn:bif_soliton} near $\mu=0$
 and give the remaining proofs of the main results in Section~2 by using the techniques
 of Section~4. 

\subsection{Proof of Theorem~\ref{thm:bif_orb_stab}}
We begin with the following direct consequence
 of the Hamiltonian-Krein index theory described in Section~4.

\begin{prop}
\label{prop:6a}
	Let $\ell$ be the number of zeros of $V_1(x)$.
	For $\epsilon>0$ sufficiently small, the Hamiltonian-Krein index for the bifurcated solitary waves \eqref{eqn:bif_soliton} are given by $K_\mathrm{Ham}=k_\mathrm r+2\ell$, where $k_\mathrm r=0$ if $\partial_\omega\|U_0\|_{L^2}^2>0$ and $k_\mathrm r=1$ if $\partial_\omega\|U_0\|_{L^2}^2<0$.
\end{prop}

\begin{proof}
	Let $\mathcal L_\pm$ and $D$ be given by \eqref{eqn:L_pm} and \eqref{eqn:D}, respectively, for the bifurcated solitary wave \eqref{eqn:bif_soliton}.
	By the Sturm-Liouville theory and the standard perturbation argument, we have
	\begin{align*}
		\mathrm n(\mathcal L_-)=\ell,\qquad
		\mathrm n(\mathcal L_+)=\begin{cases}
			\ell+1&\text{if $\bar b_2>0$};\\
			\ell+2&\text{if $\bar b_2<0$}
		\end{cases}
	\end{align*}
	(see Section 3.1 of \cite{PeYa05} and Section 6.1 of \cite{KaPr13} for details).
	On the other hand, substituting \eqref{eqn:bif_UV} into \eqref{eqn:D} and using the symmetry of $D$, we compute
	\begin{align*}
		&D_{22}=-\partial_\omega\|U_0\|_{L^2}^2+\frac{4\langle U_0,U_2\rangle_{L^2}^2}{\bar b_2}+O(\epsilon^2),\\
		&D_{23}=D_{32}=\frac{2\langle U_0,U_2\rangle_{L^2}\|V_1\|_{L^2}^2}{\bar b_2}+O(\epsilon^2),\quad
		D_{33}=\frac{\|V_1\|_{L^2}^4}{\bar b_2}+O(\epsilon^2).
	\end{align*}
	Thus, $\mathrm n(D)$ equals the number of negative elements in the set $\{-\partial_\omega\|U_0\|_{L^2}^2,\bar b_2\}$.
	Hence, $K_\mathrm{Ham}=\mathrm n(\mathcal L_-)+\mathrm n(\mathcal L_+)-\mathrm n(D)=k_\mathrm r+2\ell$ by Theorem~\ref{thm:KHam}.
	This completes the proof. 
\end{proof}

\begin{proof}[Proof of Theorem~\ref{thm:bif_orb_stab}]
Using Theorem~\ref{thm:orb} and Proposition~\ref{prop:6a},
 we immediately obtain the desired result. 
\end{proof}

The rest of this section is devoted to introduce some notation to prove Theorems~\ref{thm:main_0} to \ref{thm:main_4}.
Let $\mathscr L_\epsilon$ denote the linearized operator $\mathscr L$ defined by \eqref{eqn:lin_op} around the bifurcated solitary waves \eqref{eqn:bif_soliton}.
We have $\mathscr L_\epsilon=\mathscr L_0$ at $\epsilon=0$, where $\mathscr L_0$ is given by \eqref{eqn:L0} and has only eigenvalues on the imaginary axis by Theorem \ref{thm:fun_stab}.
The coefficient matrix $A(x,\lambda)$ in \eqref{eqn:eigprob2} is written as
\begin{align}
	A(x,\lambda,\epsilon)=A_0(x,\lambda)+\epsilon A_1(x)+\epsilon^2A_2(x)+O(\epsilon^3),
	\label{eqn:A_bif}
\end{align}
where $A_0(x,\lambda)$ is given by \eqref{eqn:A0} and
\begin{align}
	&A_1(x)=-a(x)\!\left(
		\begin{array}{c|c}
			O_4&O_4\\\hline
			\begin{matrix}O_2&I_2+\sigma_1\\I_2+\sigma_1&O_2\end{matrix}&O_4
		\end{array}
	\right),\label{eqn:A1_def}\\
	&A_2(x)=-\left(
		\begin{array}{c|c}
			O_4&O_4\\\hline
			\begin{matrix}A_{21}(x)&O_2\\O_2&A_{22}(x)\end{matrix}&O_4
		\end{array}
	\right)\nonumber
\end{align}
with
\begin{align}
	&a(x)=\partial_1\partial_2F(U_0(x)^2,0;0)U_0(x)V_1(x),\label{eqn:a_def}\\
	&A_{21}(x)=2\partial_1^2F(U_0(x)^2,0)U_0(x)U_2(x)(2I_2+\sigma_1)\nonumber\\
	&\qquad\qquad+\partial_1\partial_2F(U_0(x)^2,0;0)V_1(x)^2I_2\nonumber\\
	&\qquad\qquad+2\partial_1^3F(U_0(x)^2,0)U_0(x)^3U_2(x)(I_2+\sigma_1)\nonumber\\
	&\qquad\qquad+\partial_1^2\partial_2F(U_0(x)^2,0;0)U_0(x)^2V_1(x)^2(I_2+\sigma_1),\label{eqn:A21_def}\\
	&A_{22}(x)=\partial_2^2F(U_0(x)^2,0;0)V_1(x)^2(2I_2+\sigma_1)\nonumber\\
	&\qquad\qquad+\bigl(2\partial_1\partial_2F(U_0(x)^2,0;0)U_0(x)U_2(x)+\bar\mu\partial_2\partial_\mu F(U_0(x)^2,0;0)\bigr)I_2.\label{eqn:A22_def}
\end{align}
Let $Y_j^\pm(x,\lambda,\epsilon)$, $j=1,\ldots,8$, be the Jost solutions for $\mathscr L_\epsilon$which are given by Lemma \ref{lem:Jost_sol} 
 and coincide with \eqref{eqn:Jost_0} when $\epsilon=0$, i.e., $Y_j^\pm(x,\lambda,0)=Y_{0j}^\pm(x,\lambda)$.
We write the corresponding Evans function defined by \eqref{eqn:def_evans} as $E(\lambda,\epsilon)$.

\subsection{Proof of Theorem~\ref{thm:main_0}}

We consider case (II) in Section~2.2.

\begin{proof}[Proof of Theorem~\ref{thm:main_0}]
Let $\lambda_\mathrm{br}=\i\min\{\omega,s\}$ and let $\widetilde E(\gamma,\epsilon)$ denote the extended Evans function near the branch point $\lambda_\mathrm{br}$, where $\gamma^2=\i(\lambda-\lambda_\mathrm{br})$, as in Section~\ref{sec:Evans2}.
For any $\gamma,\epsilon\in\Rset$ near $(\gamma,\epsilon)=(0,0)$,
 we can choose the $\Rset$-valued Jost solutions $Y_j^\pm(x,\lambda,\epsilon)$, $j=1,\ldots,8$,
 since $\nu_j(\lambda)=\nu_j(\lambda_\mathrm{br}-\i\gamma^2)\in\Rset$, $j=1,\ldots,8$.
For this choice we have $\widetilde E(\gamma,\epsilon)\in\Rset$.
Assume that $\gamma=0$ is a simple zero of $\widetilde E(\gamma,0)$.
Using the implicit function theorem, we obtain a unique zero $\gamma(\epsilon)\in\Rset$ of $\widetilde E(\gamma,\epsilon)$ with $\gamma(0)=0$.
Using Proposition \ref{prop:gamma_move}, we see that one of the following three cases occurs.
\begin{enumerate}
\setlength{\leftskip}{-1.6em}
	\item[(i)]
		If $\gamma(\epsilon)>0$, then there exists a unique simple eigenvalue of $\mathscr L_\epsilon$ near $\lambda_\mathrm{br}$ and $\mathscr L_\epsilon$ has no resonance pole near $\lambda_\mathrm{br}$.
		The unique eigenvalue belongs to $\lambda_\mathrm{br}+\i(-\infty,0)$.
	\item[(ii)]
		If $\gamma(\epsilon)<0$, then there exists a unique resonance pole of $\mathscr L_\epsilon$ near $\lambda_\mathrm{br}$ and $\mathscr L_\epsilon$ has no eigenvalue near $\lambda_\mathrm{br}$.
		The unique resonance pole belongs to $\lambda_\mathrm{br}+\i(-\infty,0)$.
	\item[(iii)]
		If $\gamma(\epsilon)=0$, then $\lambda_\mathrm{br}$ remains as an eigenvalue or a resonance pole of $\mathscr L_\epsilon$.
\end{enumerate}
This completes the proof.
\end{proof}



\subsection{Proof of Theorem~\ref{thm:main_1}}\label{sec:case3}

We consider case (III) in Section~2.2 and assume that
 $\omega<s$ and $\lambda_0\in\i(\omega,s)$ is a simple zero of $E_\mathrm{B}(\lambda)$
 with $E_\mathrm{A}(\lambda_0),E_\mathrm{C}(\lambda_0)\ne0$.
Especially, $\lambda_0$ is also a simple zero of $E(\lambda,0)$.

\begin{proof}[Proof of Theorem~\ref{thm:main_1}]
We have $\nu_1(\lambda_0),\nu_4(\lambda_0)\in(0,\infty)$ and $\nu_2(\lambda_0)\in\i(0,\infty)$ (see \eqref{eqn:mu_def1} and \eqref{eqn:mu_def2}).
In particular, $p_1(x,\lambda_0)$ and $q_1(x,\lambda_0)$ are uniquely determined due to \eqref{eqn:caseA_asym}, and $p_4(x,\lambda_0)$ and $p_8(x,\lambda_0)$ satisfy
\begin{align}
	p_4(x,\lambda_0)=\tau p_8(x,\lambda_0)
	\label{eqn:p48}
\end{align}
for some $\tau\in\Rset\setminus\{0\}$ since $E_\mathrm{B}(\lambda_0)=0$.
As in Remark \ref{rmk:U0V1} (ii), we see that $\tau$ must be $1$ or $-1$.



Let $\bar p_4(x)$ be a solution to \eqref{eqn:lin_eq_3} with $(\lambda,\mu)=(\lambda_0,0)$ such that it is linearly independent of $p_8(x,\lambda_0)$ and
\begin{align*}
	\det\!\begin{pmatrix}
		\bar p_4(x)&p_8(x,\lambda_0)\\
		\bar p_4'(x)&p_8'(x,\lambda_0)
	\end{pmatrix}=1.
\end{align*}
Let $\bar Y_{04}(x)=(0,0,0,\bar p_4(x),0,0,0,\bar p_4'(x))^\T$ and
\begin{align}
	\mathscr Y_0(x)=\bigl(Y_{01}^-\;\;Y_{02}^-\;\;Y_{03}^-\;\;\bar Y_{04}\;\;Y_{05}^+\;\;Y_{06}^+\;\;Y_{07}^+\;\;Y_{08}^+\bigr)(x,\lambda_0).
	\label{eqn:Y0_def}
\end{align}
Then $\mathscr Y_0$ is a fundamental matrix solution to
\begin{align}
	Y'=A_0(x,\lambda)Y
	\label{eqn:ODE0}
\end{align}
at $\lambda=\lambda_0$.
We consider the adjoint equation
\begin{align}
	Z'=-A_0(x,\lambda)^\H Z
	\label{eqn:adj_eq}
\end{align}
for \eqref{eqn:ODE0}, where the superscript `{\scriptsize H}' stands for the Hermitian conjugate.
If $(\bm p,\bm p')$ with $\bm p:\Rset\to\Cset^4$ is a solution to \eqref{eqn:ODE0}, then $(-\bm p',\bm p)^\H$ is a solution to \eqref{eqn:adj_eq}.
Let $Z_{0j}^\pm(x,\lambda)$ and $\bar Z_{04}(x)$ be solutions to \eqref{eqn:adj_eq} which are obtained in this manner from $Y_{0j}^\pm(x,\lambda)$ and $\bar Y_{04}(x)$, respectively, for $j=1,\ldots,8$, where the lower or upper sign in the superscripts of $Z_{0j}^\pm$ and $Y_{0j}^\pm$ is taken depending on whether $j\le4$ or not.
Recall that $Y_{03}^-$, $Y_{04}^-$, $Y_{07}^+$, and $Y_{08}^+$ are uniquely determined.
If $\Delta_1(\lambda_0)\neq 0$,
 then the Jost solutions can be normalized as follows,
 where $\Delta_1(\lambda)$ was defined in \eqref{eqn:Deltaj} with $j=1$.

\begin{lem}\label{lem:Jost_normalize}
	Let $\lambda\in\Omega_\e$, where $\Omega_\e$ is defined in \eqref{eqn:Omega_e}.
	The Jost solutions $Y_{0j}^\pm(x,\lambda)$, $j=1,2,5,6$, for \eqref{eqn:eigprob2} with \eqref{eqn:A_bif} are uniquely determined by the relations
	\begin{align}
		\begin{aligned}
			&Y_{0,j+4}^\pm(x,\lambda)=
			\begin{pmatrix}I_4&O_4\\O_4&-I_4\end{pmatrix}Y_{0j}^\mp(-x,\lambda),\quad
			j=1,2,\\
			&Y_{02}^-\cdot Z_{05}^+=
			Y_{06}^-\cdot Z_{05}^+=
			Y_{05}^-\cdot Z_{05}^+=
			Y_{05}^-\cdot Z_{06}^+=
			Y_{05}^-\cdot Z_{02}^-=0,
		\end{aligned}
		\label{eqn:Jost_flip}
	\end{align}
	where the dot stands for the standard inner product on $\Cset^8$.
	Moreover, for these Jost solutions we have
	\begin{align}
		&\mathscr Y_0(x)^{-1}\notag\\
		&=\left(\begin{array}{c|c|c|c}
			\begin{matrix}-p_5'/C_1&-q_5'/C_1\\-p_6'/C_2&-q_6'/C_2\end{matrix}&O_2&\begin{matrix}p_5/C_1&q_5/C_1\\p_6/C_2&q_6/C_2\end{matrix}&O_2\\\hline
			O_2&\begin{matrix}p_7'/E_\mathrm{C}&0\\0&p_8'\end{matrix}&O_2&\begin{matrix}-p_7/E_\mathrm{C}&0\\0&-p_8\end{matrix}\\\hline
			\begin{matrix}p_1'/C_1&q_1'/C_1\\p_2'/C_2&q_2'/C_2\end{matrix}&O_2&\begin{matrix}-p_1/C_1&-q_1/C_1\\-p_2/C_2&-q_2/C_2\end{matrix}&O_2\\\hline
			O_2&\begin{matrix}-p_3'/E_\mathrm{C}&0\\0&-\bar p_4'\end{matrix}&O_2&\begin{matrix}p_3/E_\mathrm{C}&0\\0&\bar p_4\end{matrix}
		\end{array}\right)\!(x,\lambda_0),
		\label{eqn:Y0_inv}
	\end{align}
	where
	$C_1\defeq(Y_{01}^-\cdot Z_{05}^+)(\lambda_0)$ and $C_2\defeq(Y_{02}^-\cdot Z_{06}^+)(\lambda_0)$ are nonzero constants.
\end{lem}

A proof of Lemma \ref{lem:Jost_normalize} is given in Appendix \ref{sec:proof_Jost_nomalize}.
Thus, $Y_{0j}^\pm(x,\lambda)$ are now uniquely determined except for $Y_{03}^+$, $Y_{04}^+$, $Y_{07}^-$, and $Y_{08}^-$.
To estimate the zero $\lambda(\epsilon)$ of $E(\lambda,\epsilon)$ with $\lambda(0)=\lambda_0$ for $\epsilon>0$ by Proposition \ref{prop:E_form}, we need some information on $Y_j^\pm(x,\lambda,\epsilon)$.
Since $\partial_\epsilon Y_j^\pm(x,\lambda_0,0)$ satisfies
\begin{align*}
	(\partial_\epsilon Y_j^\pm)'=A_0(x,\lambda_0)\partial_\epsilon Y_j^\pm+A_1(x)Y_{0j}^\pm,
\end{align*}
we have the following lemma.

\begin{lem}\label{lem:deY}
	We have
	\begin{align}
		\begin{aligned}
			\partial_\epsilon Y_j^-(x,\lambda_0,0)&=\mathscr Y_0(x)\int_{-\infty}^x\mathscr Y_0(y)^{-1}A_1(y)Y_{0j}^-(y,\lambda_0)\,\d y,&j=1,2,3,4,\\
			\partial_\epsilon Y_j^+(x,\lambda_0,0)&=-\mathscr Y_0(x)\int_x^\infty\mathscr Y_0(y)^{-1}A_1(y)Y_{0j}^+(y,\lambda_0)\,\d y,&j=5,6,7,8.
		\end{aligned}
		\label{eqn:deY_int}
	\end{align} 
\end{lem}

\begin{proof}
	We only prove the first equation of \eqref{eqn:deY_int}.
	The second equation of \eqref{eqn:deY_int} can be shown by the same argument.
	By the variation of constants formula, we have
	\begin{align}
		Y_j^-(x,\lambda_0,\epsilon)
		&=\mathscr Y_0(x)\mathscr Y_0(x_0)^{-1}Y_j^-(x_0,\lambda_0,\epsilon)\nonumber\\
		&\quad+\epsilon\mathscr Y_0(x)\int_{x_0}^x\mathscr Y_0(y)^{-1}A_1(y)Y_j^-(y,\lambda_0,\epsilon)\,\d y\label{eqn:deY_tmp}
	\end{align}
	for any $x_0\in\Rset$.
	It follows from \eqref{eqn:caseA_asym}, \eqref{eqn:caseC_asym}, \eqref{eqn:caseB_asym}, 
	and \eqref{eqn:Y0_inv}
	that the limit of $\mathscr Y_0(x_0)^{-1}Y_j^-(x_0,\lambda_0,\epsilon)$ as $x_0\to-\infty$ exists and is independent of $\epsilon$.
	Letting $w_0\in\Cset^8$ be the limit and taking $x_0\to-\infty$ in \eqref{eqn:deY_tmp}, we have
	\begin{align*}
		Y_j^-(x,\lambda_0,\epsilon)
		=\mathscr Y_0(x)w_0+\epsilon\mathscr Y_0(x)\int_{-\infty}^x\mathscr Y_0(y)^{-1}A_1(y)Y_j^-(y,\lambda_0,\epsilon)\,\d y.
	\end{align*}
	Differentiating the above equation with respect to $\epsilon$ and putting $\epsilon=0$ yields the desired result.
\end{proof}

Using \eqref{eqn:Jost_0}, \eqref{eqn:A1_def}, and \eqref{eqn:Y0_inv}, we have the forms
\begin{align}
	\begin{aligned}
		&\partial_\epsilon Y_j^-(x,\lambda_0,0)=(0,0,*,*,0,0,*,*)^\T,\quad j=1,2,\\
		&\partial_\epsilon Y_j^-(x,\lambda_0,0)=(*,*,0,0,*,*,0,0)^\T,\quad j=3,4,\\
		&\partial_\epsilon Y_j^+(x,\lambda_0,0)=(0,0,*,*,0,0,*,*)^\T,\quad j=5,6,\\
		&\partial_\epsilon Y_j^+(x,\lambda_0,0)=(*,*,0,0,*,*,0,0)^\T,\quad j=7,8,
	\end{aligned}
	\label{eqn:deY_zero}
\end{align}
for \eqref{eqn:deY_int}.
Moreover,
\begin{align}
	&\partial_\epsilon(Y_4^--\tau Y_8^+)(x,\lambda_0,0)\nonumber\\
	&=\mathscr Y_0(x)\int_{-\infty}^\infty\mathscr Y_0(y)^{-1}A_1(y)Y_{04}^-(y,\lambda_0)\,\d y\nonumber\\
	&=\rho_1Y_{01}^-(x,\lambda_0)+\rho_2Y_{02}^-(x,\lambda_0)+\rho_5Y_{05}^+(x,\lambda_0)+\rho_6Y_{06}^+(x,\lambda_0),
	\label{eqn:rho1}
\end{align}
where $\rho_j$ is the $j$-th component of
\begin{align*}
	\int_{-\infty}^\infty\mathscr Y_0(y)^{-1}A_1(y)Y_{04}^-(y,\lambda_0)\,\d y
\end{align*}
for $j=1,2,5,6$.
Differentiating \eqref{eqn:def_evans} with respect to $\epsilon$ at $(\lambda,\epsilon)=(\lambda_0,0)$ yields
\begin{align*}
	\partial_\epsilon E(\lambda_0,0)
	=\det\bigl(Y_1^-\;\;Y_2^-\;\;Y_3^-\;\;\partial_\epsilon(Y_4^--\tau Y_8^+)\;\;Y_5^+\;\;Y_6^+\;\;Y_7^+\;\;Y_8^+\bigr)(x,\lambda_0,0).
\end{align*}
Substituting \eqref{eqn:rho1} into the above equation, we obtain
\begin{align}
	\partial_\epsilon E(\lambda_0,0)=0.
	\label{eqn:deE3}
\end{align}
Note that $Y_j^\pm(x,\lambda,0)=Y_{0j}^\pm(x,\lambda)$.
So it follows from the first equation of \eqref{eqn:dlambda} that $\partial_\epsilon\lambda(0)=0$.
We use the second equation of \eqref{eqn:dlambda} to estimate $\lambda(\epsilon)$.
For $\partial_\epsilon^2 E(\lambda_0,0)$ and $\partial_\lambda E(\lambda_0,0)$ we have the following lemmas.

\begin{lem}\label{lem:de2E}
	We have
	\begin{align}
		\partial_\epsilon^2E(\lambda_0,0)
		=2\tau E_\mathrm{A}(\lambda_0)E_\mathrm{C}(\lambda_0)\biggl(\frac{2J_1}{C_1}+\frac{2J_2}{C_2}+f_3\biggr)
		\label{eqn:de2E_mid}
	\end{align}
	with
	\begin{align}
		J_j\defeq\int_{-\infty}^\infty\int_{-\infty}^x&a(x)p_8(x,\lambda_0)\bigl(p_{j+4}(x,\lambda_0)+q_{j+4}(x,\lambda_0)\bigr)\nonumber\\
		&a(y)p_8(y,\lambda_0)\bigl(p_j(y,\lambda_0)+q_j(y,\lambda_0)\bigr)\d y\,\d x\label{eqn:J12}
	\end{align}
	for $j=1,2$ and
	\begin{align}
		f_3\defeq\int_{-\infty}^\infty(A_{22}(x))_{22}\,p_4(x,\lambda_0)^2\,\d x,\label{eqn:f3_def}
	\end{align}
	where $a(x)$ and $A_{22}(x)$ are given by \eqref{eqn:a_def} and \eqref{eqn:A22_def}, respectively.
\end{lem}

\begin{lem}\label{lem:dlE}
	We have
	\begin{align}
		\partial_\lambda E(\lambda_0,0)=-\i\tau E_\mathrm{A}(\lambda_0)E_\mathrm{C}(\lambda_0)\|p_4(\cdot,\lambda_0)\|_{L^2}^2.
		\label{eqn:dlE}
	\end{align}
\end{lem}

A proof of Lemma \ref{lem:de2E} is given in Appendix \ref{sec:proof_de2E}.
To prove Lemma \ref{lem:dlE}, we note that $\lambda=\lambda_0$ is a simple eigenvalue of \eqref{eqn:lin_eq_3} and apply Proposition 8.1.4 of \cite{KaPr13} to $E_\mathrm{B}(\lambda)$ with $(p_0,q_0)=(p_4,p_4)$.
We now compute the second equation of \eqref{eqn:dlambda} to obtain the following result.

\begin{lem}\label{lem:Relambda}
	We have
	\begin{align}
		&\Re\lambda(\epsilon)
		=-\frac{\displaystyle|\mathcal I_2(\lambda_0)|^2-\Re\!\biggl(\frac{C_3}{C_2}\mathcal I_2(\lambda_0)^2\biggr)}{2\,\|p_4(\cdot,\lambda_0)\|_{L^2}^2\,\Im\nu_2(\lambda_0)}\epsilon^2+O(\epsilon^3),\label{eqn:Relambda}\\
		&\Im\lambda(\epsilon)=\Im\lambda_0+O(\epsilon^2),\label{eqn:Relambda_im}
	\end{align}
	where $C_3\defeq(Y_{02}^-\cdot Z_{02}^+)(\lambda_0)$
	and $\mathcal I_2(\lambda)$ was defined in \eqref{eqn:Ij} with $j=2$.
	Moreover,
	\begin{align}
		\Re\lambda(\epsilon)
		\le-\frac{|\mathcal I_2(\lambda_0)|^2}{2\,\|p_4(\cdot,\lambda_0)\|_{L^2}^2\,\Im\nu_2(\lambda_0)}\!\left(1-\sqrt{1-\frac{4(\Im\nu_2(\lambda_0))^2}{|C_2|^2}}\right)\!\epsilon^2+O(\epsilon^3).\label{eqn:J2_ineq}
	\end{align}
\end{lem}

\begin{proof}
	Substituting \eqref{eqn:deE3}, \eqref{eqn:de2E_mid}, and \eqref{eqn:dlE} into \eqref{eqn:dlambda}, we obtain
	\begin{align}
		&\lambda(\epsilon)
		=\lambda_0+\!\left(-\i\biggl(\frac{2J_1}{C_1}+\frac{2J_2}{C_2}+f_3\biggr)\middle/\|p_4(\cdot,\lambda_0)\|_{L^2}^2\right)\!\epsilon^2+O(\epsilon^3),\label{eqn:lambda_eps}\\
		&\Re\lambda(\epsilon)
		=\left(\Im\!\biggl(\frac{2J_2}{C_2}\biggr)\middle/\|p_4(\cdot,\lambda_0)\|_{L^2}^2\right)\!\epsilon^2+O(\epsilon^3),\nonumber
	\end{align}
	and \eqref{eqn:Relambda_im} since $p_1(x,\lambda_0),q_1(x,\lambda_0),p_4(x,\lambda_0)\in\Rset$ by $\nu_1(\lambda_0),\nu_4(\lambda_0)\in\Rset$, and consequently $J_1,C_1,f_3\in\Rset$.
	To obtain \eqref{eqn:Relambda}, we only have to prove
	\begin{align}
		\Im\!\biggl(\frac{2J_2}{C_2}\biggr)
		&=-\frac{1}{2\,\Im\nu_2(\lambda_0)}\biggl(|\mathcal I_2(\lambda_0)|^2-\Re\!\biggl(\frac{C_3}{C_2}\mathcal I_2(\lambda_0)^2\biggr)\biggr).
		\label{eqn:J2_eq}
	\end{align}
	Since $(Y_{01}^+\ Y_{02}^+\ Y_{05}^+\ Y_{06}^+)$ and $(Y_{01}^-\ Y_{02}^-\ Y_{05}^-\ Y_{06}^-)$ are $8\times4$ matrices whose nonzero elements construct fundamental matrix solutions to \eqref{eqn:lin_eq_1} (see \eqref{eqn:Jost_0} and a remark thereafter), we see that there exists $S\in\mathrm{GL}(4,\Cset)$ such that
	\begin{align}
		\begin{pmatrix}
			Y_{01}^+&Y_{02}^+&Y_{05}^+&Y_{06}^+
		\end{pmatrix}
		=
		\begin{pmatrix}
			Y_{01}^-&Y_{02}^-&Y_{05}^-&Y_{06}^-
		\end{pmatrix}
		S
		\label{eqn:scat}
	\end{align}
	at $\lambda=\lambda_0$.
	From Lemma \ref{lem:Jost_normalize} we have
	\begin{align*}
		\begin{array}{llll}
			Y_{01}^-\cdot Z_{05}^+=C_1,&
			Y_{01}^-\cdot Z_{06}^+=0,&
			Y_{01}^-\cdot Z_{01}^-=0,&
			Y_{01}^-\cdot Z_{02}^-=0,\\
			Y_{02}^-\cdot Z_{05}^+=0,&
			Y_{02}^-\cdot Z_{06}^+=C_2,&
			Y_{02}^-\cdot Z_{01}^-=0,&
			Y_{02}^-\cdot Z_{02}^-=0,\\
			Y_{05}^-\cdot Z_{05}^+=0,&
			Y_{05}^-\cdot Z_{06}^+=0,&
			Y_{05}^-\cdot Z_{01}^-=-2\nu_1,&
			Y_{05}^-\cdot Z_{02}^-=0,\\
			Y_{06}^-\cdot Z_{05}^+=0,&
			Y_{06}^-\cdot Z_{06}^+=0,&
			Y_{06}^-\cdot Z_{01}^-=0,&
			Y_{06}^-\cdot Z_{02}^-=-2\nu_2.
		\end{array}
	\end{align*}
	Hence, 
	\begin{align}
		S=
		\begin{pmatrix}
			\dfrac{2\nu_1}{C_1}&0&0&0\\
			0&\dfrac{2\nu_2}{C_2}\Bigl(1-\Bigl(\dfrac{C_3}{2\nu_2}\Bigr)^2\Bigr)&0&-\dfrac{C_3}{2\nu_2}\\
			0&0&\dfrac{C_1}{2\nu_1}&0\\
			0&\dfrac{C_3}{2\nu_2}&0&\dfrac{C_2}{2\nu_2}
		\end{pmatrix}\!(\lambda_0).
		\label{eqn:scat2}
	\end{align}
	On the other hand, since $\nu_2(\lambda_0)$ is purely imaginary, $Y_{02}^-(x,\lambda_0)^*$ satisfies the same asymptotic condition as $Y_{06}^-(x,\lambda_0)$ when $x\to-\infty$.
	Since by Lemma \ref{lem:Jost_normalize} $Y_{06}^-$ is uniquely determined and $(Y_{02}^-)^*\cdot Z_{05}^+=(Y_{02}^-\cdot(Z_{05}^+)^*)^*=(Y_{02}^-\cdot Z_{05}^+)^*=0$, we have $Y_{02}^-(x,\lambda_0)^*=Y_{06}^-(x,\lambda_0)$.
	Hence, we use \eqref{eqn:Jost_0}, \eqref{eqn:scat}, and \eqref{eqn:scat2} to obtain
	\begin{align}
		p_6(x)=-\frac{C_3}{2\nu_2}p_2(x)+\frac{C_2}{2\nu_2}p_2(x)^*,\quad
		q_6(x)=-\frac{C_3}{2\nu_2}q_2(x)+\frac{C_2}{2\nu_2}q_2(x)^*
		\label{eqn:p2p6_rel}
	\end{align}
	at $\lambda=\lambda_0$.
	Substituting \eqref{eqn:p2p6_rel} into \eqref{eqn:J12} with $j=2$ yields \eqref{eqn:J2_eq}.

	It remains to show \eqref{eqn:J2_ineq}.
	Putting $x=0$ at \eqref{eqn:p2p6_rel} and using $(p_6(x,\lambda),q_6(x,\lambda))=(p_2(-x,\lambda),q_2(-x,\lambda))$, we have
	\begin{align*}
		(C_3+2\nu_2(\lambda_0))(p_2(0,\lambda_0),\,q_2(0,\lambda_0))
		=C_2(p_2^*(0,\lambda_0),\,q_2^*(0,\lambda_0)),
	\end{align*}
	so that $C_3^2+4(\Im\nu_2(\lambda_0))^2=|C_2|^2$ since $C_3=2\Re\!(p_2{p_2'}^*+q_2{q_2'}^*)(0,\lambda_0)\in\Rset$, $\nu_2(\lambda_0)\in\i\Rset$, and $C_2\ne0$.
	Hence, we compute
	\begin{align*}
		\displaystyle|\mathcal I_2(\lambda_0)|^2-\Re\!\biggl(\frac{C_3}{C_2}\mathcal I_2(\lambda_0)^2\biggr)
		&\ge|\mathcal I_2(\lambda_0)|^2\!\left(1-\left|\frac{C_3}{C_2}\right|\right)\\
		&=|\mathcal I_2(\lambda_0)|^2\!\left(1-\sqrt{1-\frac{4(\Im\nu_2(\lambda_0))^2}{|C_2|^2}}\right),
	\end{align*}
	which yields \eqref{eqn:J2_ineq}.
\end{proof}


If $\mathcal I_2(\lambda_0)\neq 0$, then $\Re\lambda(\epsilon)<0$ by Lemma \ref{lem:Relambda}.
Assume that $\mathscr L_\epsilon$ has an eigenvalue $\lambda_1(\epsilon)$ near $\lambda_0$.
By the symmetry of $\sigma(\mathscr L_\epsilon)$, we can assume $\Re\lambda_1(\epsilon)\ge0$.
Using Proposition \ref{prop:ev_prop}, we see that $E(\lambda_1(\epsilon),\epsilon)=0$.
However, $E(\lambda,\epsilon)$ does not have such a zero since $\lambda(\epsilon)$ is a unique zero of $E(\lambda,\epsilon)$ near $\lambda_0$.
This yields a contradiction, so that $\mathscr L_\epsilon$ has no eigenvalue near $\lambda_0$.
This completes the proof.
\end{proof}


\subsection{Proof of Theorem~\ref{thm:main_2}}

We consider case (IV) in Section~2.2 and assume that
 $\lambda_0\in\i(-\infty,-\min\{\omega,s\})$ is a simple zero of $E_\mathrm{B}(\lambda)$
 with $E_\mathrm{A}(\lambda_0),E_\mathrm{C}(\lambda_0)\ne0$.
Note that $\lambda_0\in\i(-\infty,-\omega)\setminus\{-\i s\}$ or $\lambda_0=-\i s$ (resp.\ $\lambda_0\in\i(-\infty,-\omega)$ or $\lambda_0\in\i[-\omega,-s)$) if $\omega\le s$ (resp.\ $\omega>s$).

\begin{proof}[Proof of Theorem~\ref{thm:main_2}]
We first treat the case of $\lambda_0\in\i(-\infty,-\omega)\setminus\{-\i s\}$, whether $\omega\le s$ or not.
By the choice of branch cuts \eqref{eqn:br_cut} we have $\nu_1(\lambda_0)\in\i(-\infty,0)$ and $\nu_2(\lambda_0),\nu_4(\lambda_0)\in(0,\infty)$.
In particular, the Jost solution $(p_2(x,\lambda),q_2(x,\lambda))$ to \eqref{eqn:lin_eq_1} with $\lambda=\lambda_0$ is uniquely determined by \eqref{eqn:caseA_asym}.


By exchanging the roles of $\nu_1$ and $\nu_2$, the same argument as in case (III) is valid.
Note that $\nu_1(\lambda_0)$ has a negative imaginary part while $\nu_2(\lambda_0)$ has a positive one in case (III).
In particular, the Jost solutions $Y_j^\pm$, $j=1,\ldots,8$, are uniquely determined
 if $\Delta_2(\lambda_0)\neq 0$,
 where $\Delta_2(\lambda)$ was defined in \eqref{eqn:Deltaj} with $j=2$,
 as in Lemma \ref{lem:Jost_normalize}.
As in Lemma \ref{lem:Relambda} we can prove the following.

\begin{lem}\label{lem:Relambda_2}
	We have
	\begin{align}
		&\Re\lambda(\epsilon)
		=-\frac{\displaystyle|\mathcal I_1(\lambda_0)|^2-\Re\!\biggl(\frac{C_4}{C_1}\mathcal I_1(\lambda_0)^2\biggr)}{2\,\|p_4(\cdot,\lambda_0)\|_{L^2}^2\,\Im\nu_1(\lambda_0)}\epsilon^2+O(\epsilon^3),
		\label{eqn:Relambda_2}\\
		&\Im\lambda(\epsilon)=\Im\lambda_0+O(\epsilon^2),\nonumber
	\end{align}
	where $C_4\defeq(Y_{01}^-\cdot Z_{01}^+)(\lambda_0)$
	and $\mathcal I_1(\lambda)$ was defined in \eqref{eqn:Ij} with $j=1$.
	Moreover,
	\begin{align*}
		\Re\lambda(\epsilon)
		\ge-\frac{|\mathcal I_1(\lambda_0)|^2}{2\,\|p_4(\cdot,\lambda_0)\|_{L^2}^2\,\Im\nu_1(\lambda_0)}\!\left(1-\sqrt{1-\frac{4(\Im\nu_1(\lambda_0))^2}{|C_1|^2}}\right)\!\epsilon^2+O(\epsilon^3).
	\end{align*}
\end{lem}


If $\mathcal I_1(\lambda_0)\neq 0$,
 then by Lemma \ref{lem:Relambda_2} $\Re\lambda(\epsilon)>0$,
 so that $\lambda(\epsilon)$ is an eigenvalue of $\mathscr L_\epsilon$
 by Proposition \ref{prop:E_form}.

We next consider the case of $\lambda_0=-\i s$ and $\omega<s$.
Let $\widetilde E(\gamma,\epsilon)$ denote the extended Evans function near the branch point $\lambda_\mathrm{br}=-\i s$, where $\gamma^2=-\i(\lambda-\lambda_\mathrm{br})$.
As easily shown, $\partial_\epsilon^2\widetilde E(0,0)=\partial_\epsilon^2E(-\i s,0)$ is given by \eqref{eqn:de2E_mid} with $\lambda_0=-\i s$ while $\partial_\epsilon\widetilde E(0,0)=0$.
Since $\partial_\gamma=2\i\gamma\partial_\lambda$, it follows from \eqref{eqn:dlE} that
\begin{align*}
	\partial_\gamma\widetilde E(0,0)=0,\quad
	\partial_\gamma^2\widetilde E(0,0)
	=2\tau E_\mathrm{A}(-\i s)E_\mathrm{C}(-\i s)\|p_4(\cdot,-\i s)\|_{L^2}^2.
\end{align*}
Moreover, we differentiate \eqref{eqn:def_evans} with respect to $\epsilon$ and $\gamma$ to obtain
\begin{align*}
	\partial_\epsilon\partial_\gamma\widetilde E(0,0)
	&=\det\bigl(Y_1^-\;\;Y_2^-\;\;\partial_\gamma Y_3^-\;\;\partial_\epsilon Y_4^-\;\;Y_5^+\;\;Y_6^+\;\;Y_7^+\;\;Y_8^+\bigr)(x,-\i s,0)\\
	&\quad+\det\bigl(Y_1^-\;\;Y_2^-\;\;Y_3^-\;\;\partial_\epsilon Y_4^-\;\;Y_5^+\;\;Y_6^+\;\;\partial_\gamma Y_7^+\;\;Y_8^+\bigr)(x,-\i s,0)\\
	&\quad+\det\bigl(Y_1^-\;\;Y_2^-\;\;\partial_\gamma Y_3^-\;\;Y_4^-\;\;Y_5^+\;\;Y_6^+\;\;Y_7^+\;\;\partial_\epsilon Y_8^+\bigr)(x,-\i s,0)\\
	&\quad+\det\bigl(Y_1^-\;\;Y_2^-\;\;Y_3^-\;\;Y_4^-\;\;Y_5^+\;\;Y_6^+\;\;\partial_\gamma Y_7^+\;\;\partial_\epsilon Y_8^+\bigr)(x,-\i s,0)\\
	&=0.
\end{align*}
Here we have used \eqref{eqn:p48} and $\partial_\gamma\partial_\epsilon^kY_j^\pm|_{\gamma=0}=2\i\gamma\partial_\lambda\partial_\epsilon^kY_j^\pm|_{\gamma=0}=0$ for $j\ne3,7$ and $k=0,1$ in the first equality, and \eqref{eqn:Jost_0} and \eqref{eqn:deY_zero} in the second equality.
Since $\widetilde E(\gamma,\epsilon)=\frac{1}{2}\partial_\gamma^2\widetilde E(0,0)\gamma^2+\frac{1}{2}\partial_\epsilon^2\widetilde E(0,0)\epsilon^2+\cdots$ near $(\gamma,\epsilon)=(0,0)$, all zeros of $\widetilde E(\gamma,\epsilon)$ near $(\gamma,\epsilon)=(0,0)$ are parametrized as $\gamma=\pm\gamma(\epsilon)$ with
\begin{align*}
	\gamma(\epsilon)=\frac{\sqrt{-(2J_1/C_1+2J_2/C_2+f_3)}}{\|p_4(\cdot,-\i s)\|_{L^2}}\epsilon+O(\epsilon^2).
\end{align*}
As in the proof of Lemma \ref{lem:Relambda}, we have
\begin{align*}
	\Im\!\biggl(\frac{2J_1}{C_1}+\frac{2J_2}{C_2}+f_3\biggr)
	=-\frac{1}{2\,\Im\nu_1(-\i s)}\biggl(|\mathcal I_1(-\i s)|^2-\Re\!\biggl(\frac{C_4}{C_1}\mathcal I_1(-\i s)^2\biggr)\biggr)>0
\end{align*}
if $\Delta_2(-\i s),\mathcal I_1(-\i s)\neq 0$.
Hence, we obtain $\Re\gamma(\epsilon)>0$ and $\Im\gamma(\epsilon)<0$, so that $\lambda(\epsilon)=-\i s+\i\gamma(\epsilon)^2$ is an eigenvalue of $\mathscr L_\epsilon$ by Proposition \ref{prop:gamma_move}.
This completes the proof.
\end{proof}


\begin{rmk}\label{rmk:main2_deg2}
	Let $s<\omega$.
	Suppose that $\lambda_0\in\i[-\omega,-s)$ is a simple zero of $E_\mathrm{B}(\lambda)$ such that $E_\mathrm{A}(\lambda_0),E_\mathrm{C}(\lambda_0)\ne0$.
	Then, as easily shown, a zero of $E(\lambda,\epsilon)$ near $\lambda=\lambda_0$ satisfies \eqref{eqn:lambda_eps}.
	Since $\nu_1(\lambda_0)\ge0$ and $\nu_2(\lambda_0),\nu_4(\lambda_0)>0$ by \eqref{eqn:mu_def1} and \eqref{eqn:mu_def2}, we have $p_j(x,\lambda_0),q_j(x,\lambda_0),p_4(x,\lambda_0)\in\Rset$ for $j=1,2,$ and $2J_1/C_1,2J_2/C_2,f_3\in\Rset$ in \eqref{eqn:lambda_eps}, so that $\Re\lambda(\epsilon)=O(\epsilon^3)$ as stated in Remark~\ref{rmk_main2_deg1}.
\end{rmk}


\subsection{Proof of Theorem~\ref{thm:main_3}}

We consider case~(V) in Section~2.2 and assume that
 $\omega<s$ and $\lambda_0=\i s$ is a zero of $E_\mathrm{B}(\lambda)$
 with $E_\mathrm{A}(\lambda_0),E_\mathrm{C}(\lambda_0)\ne0$.
 
\begin{proof}[Proof of Theorem~\ref{thm:main_3}]
Let $\widetilde E(\gamma,\epsilon)$ denote the extended Evans functions near the branch point $\lambda_\mathrm{br}=\i s$, where $\gamma^2=\i(\lambda-\lambda_\mathrm{br})$.
We define $\widetilde E_\mathrm{B}(\gamma)$ for $E_\mathrm{B}(\lambda)$ as $\widetilde E(\gamma,\epsilon)$ and assume that $\gamma=0$ is a simple zero of $\widetilde E_\mathrm{B}(\gamma)$.
Since any eigenvalue of \eqref{eqn:p48} is included in $\i(-\infty,s)$, $\lambda_0=\i s$ is a resonance pole of $\mathscr L_0$.
Actually, since $\nu_4(\i s)=0$, by \eqref{eqn:caseC_asym} $p_4(x,\i s)\to1$ ($x\to-\infty$) and $p_8(x,\i s)\to1$ ($x\to+\infty$) while \eqref{eqn:p48} still holds with $\tau=1$ or $-1$.
Like \eqref{eqn:deE3} and \eqref{eqn:de2E_mid} in case (III), we obtain
\begin{align*}
	\partial_\epsilon\widetilde E(0,0)=0,\quad
	\partial_\epsilon^2\widetilde E(0,0)
	=2\tau E_\mathrm{A}(\i s)E_\mathrm{C}(\i s)\biggl(\frac{2J_1}{C_1}+\frac{2J_2}{C_2}+f_3\biggr).
\end{align*}
As in Lemma 9.5.4 of \cite{KaPr13}, we see that $\partial_\gamma\widetilde E_\mathrm{B}(0)=-(1+\tau^2)\tau=-2\tau$ and $\partial_\gamma\widetilde E(0,0)=-2\tau E_\mathrm{A}(\i s)E_\mathrm{C}(\i s)$.
Hence, as in Lemma \ref{lem:Relambda},
 there exists a unique zero $\gamma(\epsilon)$ of $\widetilde E(\gamma,\epsilon)$
 with a positive imaginary part
 if $\Delta_1(\i s),\mathcal I_2(\i s)\neq 0$.
By Proposition \ref{prop:gamma_move},
 $\mathscr L_\epsilon$ has no eigenvalue near $\lambda=\i s$.
This completes the proof.
\end{proof}



\subsection{Proof of Theorem~\ref{thm:main_4}}

We consider case (VI) in Section~2.2 and assume that
 $s<\omega$ and $\lambda_0=\i\omega$ is a zero of $E_\mathrm{A}(\lambda)$
 with $E_\mathrm{B}(\lambda_0),E_\mathrm{C}(\lambda_0)\ne0$.

\begin{proof}[Proof of Theorem~\ref{thm:main_4}]
Let $\widetilde E(\gamma,\epsilon)$ denote the extended Evans functions near the branch point $\lambda_\mathrm{br}=\i\omega$, where $\gamma^2=\i(\lambda-\lambda_\mathrm{br})$.
We define $\widetilde E_\mathrm{A}(\gamma)$ for $E_\mathrm{A}(\lambda)$ as $\widetilde E(\gamma,\epsilon)$ and assume that $\gamma=0$ is a simple zero of $\widetilde E_\mathrm{A}(\gamma)$.
Note that $\nu_1(\i\omega),\nu_3(\i\omega)\in(0,\infty)$, $\nu_2(\i\omega)=0$, and $\nu_4(\i\omega)\in\i(0,\infty)$.
Under \textbf{(A4)}, $\lambda=\i\omega$ is a resonance pole of $\mathscr L_0$.
So we take the Jost solutions $(p_j(x,\i\omega),q_j(x,\i\omega))$, $j=2,6$, to \eqref{eqn:lin_eq_1} such that
\begin{align}
	\begin{aligned}
		&p_2(x,\i\omega)=\tau p_6(x,\i\omega),\quad
		q_2(x,\i\omega)=\tau q_6(x,\i\omega),\\
		&\lim_{x\to-\infty}p_2(x,\i\omega)=\lim_{x\to+\infty}p_6(x,\i\omega)=0,\\
		&\lim_{x\to-\infty}q_2(x,\i\omega)=\lim_{x\to+\infty}q_6(x,\i\omega)=1,
	\end{aligned}
	\label{eqn:p26}
\end{align}
where $\tau=1$ or $-1$.

We can compute the zero $\gamma=\gamma(\epsilon)$ of $\widetilde E(\gamma,\epsilon)$ in a similar fashion to case (III).
So we briefly describe it below.
Since $\nu_1(\i\omega),\nu_3(\i\omega)\in(0,\infty)$, $\nu_2(\i\omega)=0$, and $\nu_4(\i\omega)\in\i(0,\infty)$, the other Jost solutions $(p_j(x,\lambda),q_j(x,\lambda))$, $j=1,5$, and $p_j(x,\lambda)$, $j=3,4,7,8$, are uniquely determined.
Note that the Jost solutions are determined without Lemma \ref{lem:Jost_normalize}.
Let $(\bar p_2(x),\bar q_2(x))$ be a solution to \eqref{eqn:lin_eq_1} with $\lambda=\i\omega$ such that it is linearly independent of $(p_j(x,\i\omega),q_j(x,\i\omega))$, $j=1,5,6$, and
\begin{align}
	\bar p_2p_1'+\bar q_2q_1'=0,\quad
	\bar p_2'p_1+\bar q_2'q_1=0,\quad
	\bar p_2p_2'+\bar q_2q_2'+\bar p_2'p_2+\bar q_2'q_2=1
	\label{eqn:E1scale}
\end{align}
at $(x,\lambda)=(0,\i\omega)$.
Such a solution can always be taken by adding any linear combination
 of $(p_j(x,\i\omega),q_j(x,\i\omega))$, $j=1,5,6$, to it if necessary.
Define
\begin{align*}
	&\bar Y_{02}(x)=(\bar p_2,\bar q_2,0,0,\bar p_2',\bar q_2',0,0)^\T(x),\\
	&\widehat{\mathscr Y}_0(x)=\bigl(Y_{01}^-\;\;\bar Y_{02}\;\;Y_{03}^-\;\;Y_{04}^-\;\;Y_{05}^+\;\;Y_{06}^+\;\;Y_{07}^+\;\;Y_{08}^+\bigr)(x,\i\omega).
\end{align*}
Let $Z_{0j}^\pm(x,\lambda)$, $j=1,\ldots,8$, and $\bar Z_{02}(x)$ be the solutions obtained from $Y_{0j}^\pm(x,\lambda)$, $j=1,\ldots,8$, and $\bar Y_{02}(x)$, respectively, to the adjoint equation \eqref{eqn:adj_eq} as in case (III).
As in Lemma \ref{lem:Jost_normalize}, we have the following.

\begin{lem}
	We have
\footnotesize
	\begin{align*}
		&\widehat{\mathscr Y}_0(x)^{-1}\\
		&=\left(\begin{array}{c|c|c|c}
			\begin{matrix}-p_5'/\widehat C_1&-q_5'/\widehat C_1\\-p_6'&-q_6'\end{matrix}&O_2&\begin{matrix}p_5/\widehat C_1&q_5/\widehat C_1\\p_6&q_6\end{matrix}&O_2\\\hline
			O_2&\begin{matrix}p_7'/E_\mathrm{C}&0\\0&p_8'/E_\mathrm{B}\end{matrix}&O_2&\begin{matrix}-p_7/E_\mathrm{C}&0\\0&-p_8/E_\mathrm{B}\end{matrix}\\\hline
			\begin{matrix}p_1'/\widehat C_1&q_1'/\widehat C_1\\p_2'&q_2'\end{matrix}&O_2&\begin{matrix}-p_1/\widehat C_1&-q_1/\widehat C_1\\-p_2&-q_2\end{matrix}&O_2\\\hline
			O_2&\begin{matrix}-p_3'/E_\mathrm{C}&0\\0&-p_4'/E_\mathrm{B}\end{matrix}&O_2&\begin{matrix}p_3/E_\mathrm{C}&0\\0&p_4/E_\mathrm{B}\end{matrix}
		\end{array}\right)\!(x,\i\omega),
	\end{align*}
	\normalsize
	where $\widehat C_1\defeq(Y_{01}^-\cdot Z_{05}^+)(\i\omega)$ is a nonzero constant.
\end{lem}

\begin{proof}
	The proof is similar to that of the second part in Lemma \ref{lem:Jost_normalize}.
	Let
	\begin{align*}
		\widehat{\mathscr Z}_0(x)=\bigl(Z_{01}^-\;\;\bar Z_{02}\;\;Z_{03}^-\;\;Z_{04}^-\;\;Z_{05}^+\;\;Z_{06}^+\;\;Z_{07}^+\;\;Z_{08}^+\bigr)(x,\i\omega).
	\end{align*}
	Instead of \eqref{eqn:Z0Y0} we have
	\begin{align*}
		\widehat{\mathscr Z}_0(0)^\H\widehat{\mathscr Y}_0(0)
		=\left(\begin{array}{c|c|c|c}
			O_2&O_2&\begin{matrix}-\widehat C_1&0\\0&-1\end{matrix}&O_2\\\hline
			O_2&*&O_2&*\\\hline
			\begin{matrix}\widehat C_1&0\\0&1\end{matrix}&O_2&O_2&O_2\\\hline
			O_2&*&O_2&*
		\end{array}\right)
	\end{align*}
	by \eqref{eqn:E1scale}.
	Taking the determinant in the above identity yields $\widehat C_1\ne0$.
	Using the formula $\widehat{\mathscr Y}_0(x)^{-1}=(\widehat{\mathscr Z}_0(0)^\H\widehat{\mathscr Y}_0(0))^{-1}\widehat{\mathscr Z}_0(x)^\H$, we obtain the desired result.
\end{proof}

Like \eqref{eqn:deE3} and \eqref{eqn:de2E_mid} in case (III), we obtain
\begin{align}
	\partial_\epsilon\widetilde E(0,0)=0\label{eqn:deE2}
\end{align}
and
\begin{align}
	\partial_\epsilon^2\widetilde E(0,0)
	=2E_\mathrm{B}(\i\omega)&E_\mathrm{C}(\i\omega)\biggl(\frac{J_3}{E_\mathrm{C}(\i\omega)}+\frac{J_4}{E_\mathrm{B}(\i\omega)}\nonumber\\
	&-\int_{-\infty}^\infty(p_6(x,\i\omega),\,q_6(x,\i\omega))A_{21}(x)\!\begin{pmatrix}p_2(x,\i\omega)\\q_2(x,\i\omega)\end{pmatrix}\!\d x\biggr),\label{eqn:de2E2}
\end{align}
where
\begin{align}
	J_j\defeq\int_{-\infty}^\infty\int_{-\infty}^x&a(x)\bigl(p_6(x,\lambda_0)+q_6(x,\lambda_0)\bigr)p_{j+4}(x,\lambda_0)\nonumber\\
	&a(y)\bigl(p_6(y,\lambda_0)+q_6(y,\lambda_0)\bigr)p_j(y,\lambda_0)\,\d y\,\d x\label{eqn:J34}
\end{align}
for $j=3,4$.
Recall that $a(x)$ and $A_{21}(x)$ are given by \eqref{eqn:a_def} and \eqref{eqn:A21_def}, respectively.
As in Lemma 9.5.4 of \cite{KaPr13}, we see that
\begin{align}
	\partial_\gamma\widetilde E(0,0)=2E_\mathrm{B}(\i\omega)E_\mathrm{C}(\i\omega).\label{eqn:dlE2}
\end{align}
As in Lemma \ref{lem:Relambda} we have the following lemma.

\begin{lem}\label{lem:Relambda_iomega}
	We have
	\begin{align}
		&\Re\gamma(\epsilon)=O(\epsilon^2),\label{eqn:Relambda_iomega_re}\\
		&\Im\gamma(\epsilon)
		=-\frac{1}{2\,\Im\nu_4(\i\omega)}\biggl(|\mathcal I_2(\i\omega)|^2-\Re\!\biggl(\frac{C_5}{E_\mathrm{B}(\i\omega)}\mathcal I_2(\i\omega)^2\biggr)\biggr)\epsilon^2+O(\epsilon^3)\label{eqn:Relambda_iomega},
	\end{align}
	where $C_5\defeq(Y_{04}^-\cdot Z_{04}^+)(\i\omega)$.
	Moreover,
	\begin{align}
		\Im\gamma(\epsilon)
		\le-\frac{|\mathcal I_2(\i\omega)|^2}{2\,\Im\nu_4(\i\omega)}\!\left(1-\sqrt{1-\frac{4(\Im\nu_4(\i\omega))^2}{|E_\mathrm{B}(\i\omega)|^2}}\right)\!\epsilon^2+O(\epsilon^3).\label{eqn:J4_ineq}
	\end{align}
\end{lem}

\begin{proof}
	Substituting \eqref{eqn:deE2}, \eqref{eqn:de2E2}, and \eqref{eqn:dlE2} into \eqref{eqn:dgamma}, we obtain \eqref{eqn:Relambda_iomega_re} and
	\begin{align*}
		\Im\gamma(\epsilon)=-\Im\!\biggl(\frac{J_4}{E_\mathrm{B}(\i\omega)}\biggr)\epsilon^2+O(\epsilon^3)
	\end{align*}
	since $p_2(x,\i\omega),q_2(x,\i\omega),p_3(x,\i\omega)\in\Rset$ by $\nu_2(\i\omega),\nu_3(\i\omega)\in\Rset$, so that $J_3,E_\mathrm{C}(\i\omega)\in\Rset$.
	To obtain \eqref{eqn:Relambda_iomega}, we only have to prove
	\begin{align}
		\Im\!\biggl(\frac{J_4}{E_\mathrm{B}(\i\omega)}\biggr)
		=\frac{1}{2\,\Im\nu_4(\i\omega)}\biggl(|\mathcal I_2(\i\omega)|^2-\Re\!\biggl(\frac{C_5}{E_\mathrm{B}(\i\omega)}\mathcal I_2(\i\omega)^2\biggr)\biggr).\label{eqn:J4_eq}
	\end{align}
	Since $(Y_{04}^+\ Y_{08}^+)$ and $(Y_{04}^-\ Y_{08}^-)$ are $8\times2$ matrices whose nonzero elements construct fundamental matrix solution to \eqref{eqn:lin_eq_3}, we see that there exists $\widehat S\in GL(2,\Cset)$ such that
	\begin{align}
		\begin{pmatrix}
			Y_{04}^+&Y_{08}^+
		\end{pmatrix}
		=
		\begin{pmatrix}
			Y_{04}^-&Y_{08}^-
		\end{pmatrix}
		\widehat S
		\label{eqn:scat3}
	\end{align}
	at $\lambda=\i\omega$.
	Using \eqref{eqn:caseC_asym}, we have
	\begin{align}
		\widehat S=
		\begin{pmatrix}
			-\dfrac{2\nu_4}{E_\mathrm{B}}\Bigl(1-\Bigl(\dfrac{C_5}{2\nu_4}\Bigr)^2\Bigr)&-\dfrac{C_5}{2\nu_4}\\
			\dfrac{C_5}{2\nu_4}&-\dfrac{E_\mathrm{B}}{2\nu_4}
		\end{pmatrix}\!(\i\omega)
		\label{eqn:scat4}
	\end{align}
	as in \eqref{eqn:scat2}.
	On the other hand, since $\nu_4(\i\omega)$ is purely imaginary, $Y_{04}^-(x,\i\omega)^*$ satisfies the same asymptotic condition as $Y_{08}^-(x,\i\omega)$ when $x\to-\infty$.
	Since $Y_{08}^-$ is uniquely determined, we have $Y_{04}^-(x,\i\omega)^*=Y_{08}^-(x,\i\omega)$.
	Hence, we use \eqref{eqn:Jost_0}, \eqref{eqn:scat3}, and \eqref{eqn:scat4} to obtain
	\begin{align}
		p_8(x,\i\omega)=-\frac{1}{2\nu_4(\i\omega)}\bigl(C_5p_4(x,\i\omega)+E_\mathrm{B}(\i\omega)p_4(x,\i\omega)^*\bigr).
		\label{eqn:p4p8_rel}
	\end{align}
	Substituting \eqref{eqn:p4p8_rel} into \eqref{eqn:J34} with $j=4$ yields \eqref{eqn:J4_eq}.

	It remains to show \eqref{eqn:J4_ineq}.
	Since $C_5=2\Re\!(p_4p_4'^*)(0,\i\omega)\in\Rset$ and $\nu_4(\i\omega)\in\i\Rset$, we obtain $|E_\mathrm{B}(\i\omega)|^2=C_5^2+4(\Im\nu_4(\i\omega))^2$, from which \eqref{eqn:J4_ineq} immediately follows.
\end{proof}

From Proposition \ref{prop:gamma_move} and Lemma \ref{lem:Relambda_iomega} we see that $\mathscr L_\epsilon$ has no eigenvalue near $\lambda=\i\omega$ when $\epsilon>0$.
This completes the proof.
\end{proof}


\begin{rmk}\label{rmk:A6_equiv}
	{\rm\textbf{(A4)}} holds if and only if there exists a bounded solution $(\hat p(x),\hat q(x))\not\in L^2(\Rset)\times L^2(\Rset)$ to \eqref{eqn:lin_eq_1} with $\lambda=\i\omega$.
	The necessity is obvious.
	Conversely, by Lemma~\ref{lem:Jost_sol} any bounded solution $(\hat p(x),\hat q(x))$ to \eqref{eqn:lin_eq_1} with $\lambda=\i\omega$ satisfies one of the following asymptotics:
	\begin{enumerate}
		\renewcommand{\labelenumi}{(\roman{enumi})}
		\item
			$\displaystyle
			\lim_{x\to\pm\infty}\hat p(x)=0,\quad
			\lim_{x\to-\infty}\hat q(x)=0,\quad
			\lim_{x\to+\infty}\hat q(x)=0,$
		\item
			$\displaystyle
			\lim_{x\to\pm\infty}\hat p(x)=0,\quad
			\lim_{x\to-\infty}\hat q(x)=0,\quad
			\lim_{x\to+\infty}\hat q(x)\ne0,$
		\item
			$\displaystyle
			\lim_{x\to\pm\infty}\hat p(x)=0,\quad
			\lim_{x\to-\infty}\hat q(x)\ne0,\quad
			\lim_{x\to+\infty}\hat q(x)=0,$
		\item
			$\displaystyle
			\lim_{x\to\pm\infty}\hat p(x)=0,\quad
			\lim_{x\to-\infty}\hat q(x)\ne0,\quad
			\lim_{x\to+\infty}\hat q(x)\ne0.$
	\end{enumerate}
	If case (ii) or (iii) holds, then $(\hat p(-x),\hat q(-x))$ is also a solution to \eqref{eqn:lin_eq_1} with $\lambda=\i\omega$ which is linearly independent of $(\hat p(x),\hat q(x))$.
	Since the zero $\gamma=0$ of $\widetilde E_\mathrm{A}(\gamma)$ is simple, we see that only case (i) or (iv) occurs.
	Hence, if $(\hat p(x),\hat q(x))\not\in L^2(\Rset)\times L^2(\Rset)$ then case (iv) must occur.
\end{rmk}


\section{Example: Coupled Cubic Nonlinear Schr\"odinger Equations}

In this section we apply the theories of Section~2 to the CNLS equations \eqref{eqn:CNLS} with the particular nonlinearity \eqref{eqn:cubic}, i.e.,
\begin{align}
	\mathrm{i}\partial_t u+\partial_{x}^2u+(|u|^2+\beta_1|v|^2)u=0,\quad
	\mathrm{i}\partial_t v+\partial_{x}^2v+(\beta_1|u|^2+\beta_2|v|^2)v=0,
	\label{eqn:CNLS_cubic}
\end{align}
and 
 illustrate the theoretical results with some numerical computations.
Here we take $\beta_1\in\Rset$ as the control parameter. 
Throughout this section, we set $\omega=1$
 by replacing $\omega t$ and $s/\omega$ with $t$ and $s$, respectively, if necessary
 without loss of generality.


\subsection{Bifurcation Analysis}

For \eqref{eqn:CNLS_cubic} we easily see that \textbf{(A1)} and \textbf{(A2)} hold
 with $U_0(x)=\sqrt 2\,\sech x$ for any $\beta_1,\beta_2\in\Rset$.
The NVE \eqref{eqn:VE2} becomes
\begin{equation}
	\delta V''-(s-2\beta_1\sech^2x)\delta V=0.
	\label{eqn:exVE2}
\end{equation}
We first discuss bifurcations of the fundamental solitary wave.
As shown in Section 5 of \cite{BlYa12}, when
\begin{align}
	\beta_1=\beta_1^{(\ell)}\defeq\frac{(\sqrt s+\ell)(\sqrt s+\ell+1)}{2},\quad
	\ell\in\Zset_{\ge0}\defeq\{n\in\Zset\mid n\ge 0\},
	\label{eqn:beta1c}
\end{align}
\textbf{(A3)} holds and \eqref{eqn:exVE2} has a bounded solution 
given by
\begin{align}
	V_1^{(\ell)}(x)
	&=2^{\sqrt s}\Gamma(\sqrt s+1)P_{\sqrt s+\ell}^{-\sqrt s}(\tanh x)\notag\\
	&=\begin{cases}
		\sech^{\sqrt s}x\,\pFq{2}{1}{-\ell',\sqrt s+\ell'+\tfrac{1}{2}}{\sqrt s+1}{\sech^2x}&\text{if $\ell=2\ell'$};\\
		\sech^{\sqrt s}x\tanh x\,\pFq{2}{1}{-\ell',\sqrt s+\ell'+\tfrac{3}{2}}{\sqrt s+1}{\sech^2x}&\text{if $\ell=2\ell'+1$},
	\end{cases}
	\label{eqn:V1_cubic}
\end{align}
where $\ell'\in\Zset_{\ge0}$, $P^\mu_\nu(z)$ is the associated Legendre function of the first kind
\begin{align*}
	P^\mu_\nu(z)=\frac{1}{\Gamma(1-\mu)}\biggl(\frac{1+z}{1-z}\biggr)^{\mu/2}\pFq{2}{1}{-\nu,\nu+1}{1-\mu}{\frac{1-z}{2}}
\end{align*}
and ${}_p F_q$ is the (generalized) hypergeometric function
\begin{align}
	\pFq{p}{q}{a_1,\ldots,a_p}{b_1,\ldots,b_q}{z}=\sum_{j=0}^\infty\frac{(a_1)_j\cdots(a_p)_j}{j!\,(b_1)_j\cdots(b_q)_j}z^j
	\label{eqn:def_hypgeo}
\end{align}
with $(x)_j\defeq\Gamma(x+j)/\Gamma(x)$.
See \cite{AbSt64,Sl66} for necessary information on the associated Legendre functions and hypergeometric functions.
See also Appendix A of \cite{YaSt20}.
Small errors were contained in (52) and (53) of \cite{BlYa12}, which correspond to \eqref{eqn:V1_cubic}.
From \eqref{eqn:phi12} and  \eqref{eqn:phi11} we have
\begin{align}
	\phi_{11}(x)=\frac{1}{2}\sech x\,(3-\cosh^2x-3x\tanh x),\quad
	\phi_{12}(x)=\sech x\tanh x.
	\label{eqn:phi11_12}
\end{align}
Recall that $\phi_{jl}(x)$ represents the $(j,l)$-element of the fundamental matrix solution $\Phi(x)$ to the homogeneous part of \eqref{eqn:aprx1} with $\Phi(0)=I_2$ for $j,l=1,2$.
Using \eqref{eqn:U2}, we express \eqref{eqn:a2ex} and \eqref{eqn:b2ex} as
\begin{align}
	&\bar a_2=-2\int_{-\infty}^\infty V_1^{(\ell)}(x)^2\sech^2x\,\d x,\label{eqn:a2_tmp}\\
	&\bar b_2=8\bigl(\beta_1^{(\ell)}\bigr)^2\int_{-\infty}^\infty\phi_{11}(x)V_1^{(\ell)}(x)^2\sech x\biggl(\int_x^\infty\phi_{12}(y)V_1^{(\ell)}(y)^2\sech y\,\d y\biggr)\d x\notag\\
	&\qquad-\beta_2\int_{-\infty}^\infty V_1^{(\ell)}(x)^4\,\d x.\label{eqn:b2_tmp}
\end{align}
Here we have used the fact that $\phi_{11}(x)$ and $V_1^{(\ell)}(x)^2$ are even functions
 and $\phi_{12}(x)$ and
\begin{align*}
	\int_0^x\phi_{11}(y)V_1^{(\ell)}(y)^2\sech y\,\d y
\end{align*}
are odd functions.
Especially, we easily see by \eqref{eqn:a2_tmp} and \eqref{eqn:b2_tmp}
 that $\bar{a}_2$ is always negative
 and $\bar{b}_2$ is positive (resp.\ negative) if $\beta_2$ is negative (resp.\ positive)
 with sufficiently large magnitude.
Applying Theorem \ref{thm:cnls_bif}, we obtain the following for \eqref{eqn:CNLS_cubic}.

\begin{thm}
\label{thm:bifex}
For $\ell\in\Zset_{\ge 0}$,
 a pitchfork bifurcation of the fundamental solitary wave \eqref{eqn:fun_soliton}
 with $U_0(x)=\sqrt{2}\sech x$ occurs at $\beta_1=\beta_1^{(\ell)}$ given by \eqref{eqn:beta1c}
 if $\bar{b}_2\neq 0$.
In addition, it is supercritical or subcritical, depending on whether $\bar{b}_2>0$ or $<0$.
Moreover, the bifurcated solitary waves are expressed
 as \eqref{eqn:bif_soliton} with \eqref{eqn:bif_UV},
 where $\epsilon>0$ is a small parameter such that $\beta_1=\beta_1^{(\ell)}+O(\epsilon^2)$.
\end{thm}


In what follows, we obtain tractable expressions for $\bar a_2$ and $\bar b_2$.
We begin with two auxiliary results.

\begin{lem}\label{lem:V1_3F2}
	We have
	\begin{align}
		V_1^{(\ell)}(x)^2=\sech^{2\sqrt s}x\,\pFq{3}{2}{-\ell,2\sqrt s+\ell+1,\sqrt s+\frac{1}{2}}{2\sqrt s+1,\sqrt s+1}{\sech^2x}
	\label{eqn:lemV1_3F2a}
	\end{align}
	and
	\begin{align}
		V_1^{(\ell)}(x)^2=\sech^{2\sqrt s}x\sum_{j=0}^\ell C_j^{(\ell)}\sech^{2j}x,
	\label{eqn:lemV1_3F2b}
	\end{align}
	where
	\begin{align*}
		C_j^{(\ell)}
		=(-1)^j\begin{pmatrix}\ell\\j\end{pmatrix}\frac{(2\sqrt s+\ell+1)_j(\sqrt s+\frac{1}{2})_j}{(2\sqrt s+1)_j(\sqrt s+1)_j},
	\end{align*}
	and $\begin{pmatrix}\ell\\j\end{pmatrix}$ represents the binomial coefficient for $\ell,j\in\Zset_{\ge0}$.
\end{lem}

\begin{proof}
	When $\ell$ is even, \eqref{eqn:lemV1_3F2a} directly follows from Clausen's theorem (see, e.g., (2.5.7) of \cite{Sl66})
	\begin{align*}
		\pFq{2}{1}{a,b}{a+b+\tfrac{1}{2}}{z}^2
		=\pFq{3}{2}{2a,2b,a+b}{2a+2b,a+b+\tfrac{1}{2}}{z}
	\end{align*}
	for $(a,b)=(-\ell/2,\sqrt s+\ell/2+1/2)$.
	When $\ell=2k+1$ is odd, we first apply the formula
	\begin{align*}
		(1-z)^{a+b-c}\pFq{2}{1}{a,b}{c}{z}=\pFq{2}{1}{c-a,c-b}{c}{z}
	\end{align*}
	(see, e.g., 15.3.3 of \cite{AbSt64})
	and then use Clausen's theorem to obtain \eqref{eqn:lemV1_3F2a}.
	By noting
	\begin{align*}
		\frac{(-\ell)_j}{j!}=(-1)^j\begin{pmatrix}\ell\\j\end{pmatrix},
	\end{align*}
	\eqref{eqn:lemV1_3F2b} easily follows from \eqref{eqn:def_hypgeo} and \eqref{eqn:lemV1_3F2a}.
\end{proof}

\begin{lem}\label{lem:IJ}
	Let
	\begin{align*}
		\mathcal K_r\defeq\int_{-\infty}^\infty\sech^rx\,\d x,\quad r>0.
	\end{align*}
	We have
	\begin{align}
		\mathcal K_r=\frac{\sqrt\pi\,\Gamma(r/2)}{\Gamma((r+1)/2)},\quad
		\mathcal K_{r+2}=\frac{r}{r+1}\mathcal K_r.
	\label{eqn:lemIJ1}
	\end{align}
	Moreover,
	\begin{align}
		\int_{-\infty}^\infty\phi_{11}(x)\sech^{r+1}x\,\d x
		=\frac{r-1}{r+2}\mathcal K_r
		=\frac{\sqrt\pi(r-1)\Gamma(r/2)}{(r+2)\Gamma((r+1)/2)}.
	\label{eqn:lemIJ2}
	\end{align}
\end{lem}

\begin{proof}
	The first equality in \eqref{eqn:lemIJ1} is well known (see, e.g., p.\ 71 of \cite{PeWe92}).
	The second equality immediately follows from the first one.
	Finally we compute
	\begin{align*}
		\int_{-\infty}^\infty\phi_{11}(x)\sech^{r+1}x\,\d x
		&=\frac{3}{2}\mathcal K_{r+2}-\frac{1}{2}\mathcal K_r-\frac{3}{2}\int_{-\infty}^\infty x\sech^{r+2}x\tanh x\,\d x\\
		&=\frac{3}{2}\mathcal K_{r+2}-\frac{1}{2}\mathcal K_r-\frac{3}{2(r+2)}\mathcal K_{r+2},
	\end{align*}
	which yields \eqref{eqn:lemIJ2} along with \eqref{eqn:lemIJ1}.
\end{proof}

Using Lemmas~\ref{lem:V1_3F2} and \ref{lem:IJ}, we obtain the following formulas of $\bar a_2$ and $\bar b_2$.

\begin{prop}\label{prop:a2b2_cubic}
	For $s>0$ and $\ell\in\Zset_{\ge0}$, we have
	\begin{align}
		&\bar a_2
		=-\frac{\sqrt\pi\,\Gamma(\sqrt s)}{\Gamma\bigl(\sqrt s+\frac{1}{2}\bigr)}\frac{2\,\ell!\sqrt s}{\bigl(\sqrt s+\ell+\frac{1}{2}\bigr)(2\sqrt s+1)_\ell},\label{eqn:a2_cubic}\\
		&\bar b_2
		=\frac{\sqrt\pi\,\Gamma(2\sqrt s)}{\Gamma\bigl(2\sqrt s+\frac{1}{2}\bigr)}\sum_{j_1=0}^\ell\sum_{j_2=0}^\ell
		C_{j_1}^{(\ell)} C_{j_2}^{(\ell)}\frac{(2\sqrt s)_{j_1+j_2}}{\bigl(2\sqrt s+\frac{1}{2}\bigr)_{j_1+j_2}}\nonumber\\
		&\qquad\qquad\qquad\times\biggl(\frac{(\sqrt s+\ell)^2(\sqrt s+\ell+1)^2(2\sqrt s+j_1+j_2)}{(\sqrt s+j_1+1)(2\sqrt s+j_1+j_2+2)}-\beta_2\biggr).\label{eqn:b2_cubic}
	\end{align}
\end{prop}

\begin{proof}
	Using \eqref{eqn:def_hypgeo} and the formula
	\begin{align}
		\Gamma(x)\Gamma(y)=\Gamma(x+y)B(x,y)\label{eqn:Gamma_Beta}
	\end{align}
	(see, e.g., 6.2.2 of \cite{AbSt64}), where $B(x,y)$ is the beta function, we have
	\begin{align}
		\begin{aligned}
			&\pFq{p+1}{q+1}{a_1,\ldots,a_p,c}{b_1,\ldots,b_q,d}{z}\\
			&=\frac{\Gamma(d)}{\Gamma(c)\Gamma(d-c)}\int_0^1t^{c-1}(1-t)^{d-c-1}\pFq{p}{q}{a_1,\ldots,a_p}{b_1,\ldots,b_q}{tz}\d t.
		\end{aligned}
		\label{eqn:int_pFq}
	\end{align}
Using \eqref{eqn:def_hypgeo}, \eqref{eqn:lemV1_3F2a} and \eqref{eqn:int_pFq}, we compute
	\begin{align*}
		\int_{-\infty}^\infty V_1^{(\ell)}(x)^2\sech^2x\,\d x
		&=\frac{\Gamma(\sqrt s+1)\Gamma\bigl(\frac{1}{2}\bigr)}{\Gamma\bigl(\sqrt s+\frac{3}{2}\bigr)}\pFq{4}{3}{-\ell,2\sqrt s+\ell+1,\sqrt s+\frac{1}{2},\sqrt s+1}{2\sqrt s+1,\sqrt s+1,\sqrt s+\frac{3}{2}}{1}\\
		&=\frac{\Gamma(\sqrt s+1)\Gamma\bigl(\frac{1}{2}\bigr)}{\Gamma\bigl(\sqrt s+\frac{3}{2}\bigr)}\pFq{3}{2}{-\ell,2\sqrt s+\ell+1,\sqrt s+\frac{1}{2}}{2\sqrt s+1,\sqrt s+\frac{3}{2}}{1}\\
		&=\frac{\Gamma(\sqrt s+1)\Gamma\bigl(\frac{1}{2}\bigr)}{\Gamma\bigl(\sqrt s+\frac{3}{2}\bigr)}\frac{(-\ell)_\ell\bigl(\sqrt s+\frac{1}{2}\bigr)_\ell}{(2\sqrt s+1)_\ell\bigl(-\ell-\sqrt s-\frac{1}{2}\bigr)_\ell}\\
		&=\frac{\Gamma(\sqrt s)\Gamma\bigl(\frac{1}{2}\bigr)}{\Gamma\bigl(\sqrt s+\frac{1}{2}\bigr)}\frac{\ell!\sqrt s}{(2\sqrt s+1)_\ell\bigl(\sqrt s+\ell+\frac{1}{2}\bigr)},
	\end{align*}
	which yields \eqref{eqn:a2_cubic} by \eqref{eqn:a2_tmp}.
	In the third equality we have used Saalsch\"utz's theorem stating that
	\begin{align}
		\pFq{3}{2}{-n,b,c}{d,e}{1}=\frac{(d-b)_n(d-c)_n}{(d)_n(d-b-c)_n}
		\label{eqn:Saal}
	\end{align}
	if $n\in\Zset_{\ge0}$ and $-n+b+c+1=d+e$ (see, e.g., (2.3.1.4) of \cite{Sl66}).

	We turn to \eqref{eqn:b2_tmp}.
	Using \eqref{eqn:phi11_12} and \eqref{eqn:lemV1_3F2b}, we obtain
	\begin{align*}
		\int_x^\infty\phi_{12}(y)V_1^{(\ell)}(y)^2\sech y\,\d y
		=\sum_{j_1=0}^\ell\frac{C_{j_1}^{(\ell)}}{2(\sqrt s+j_1+1)}\sech^{2(\sqrt s+j_1+1)}x.
	\end{align*}
	Substituting this equality and using \eqref{eqn:lemIJ1}, we compute
	\begin{align*}
		&\int_{-\infty}^\infty\phi_{11}(x)V_1^{(\ell)}(x)^2\sech x\biggl(\int_x^\infty\phi_{12}(y)V_1^{(\ell)}(y)^2\sech y\,\d y\biggr)\d x\\
		&=\sum_{j_1=0}^\ell\sum_{j_2=0}^\ell\frac{C_{j_1}^{(\ell)} C_{j_2}^{(\ell)}}{2(\sqrt s+j_1+1)}\frac{4\sqrt s+2j_1+2j_2+1}{4\sqrt s+2j_1+2j_2+4}\mathcal K_{4\sqrt s+2j_1+2j_2+2}\\
		&=\sum_{j_1=0}^\ell\sum_{j_2=0}^\ell\frac{C_{j_1}^{(\ell)} C_{j_2}^{(\ell)}}{2(\sqrt s+j_1+1)}\frac{2\sqrt s+j_1+j_2}{2\sqrt s+j_1+j_2+2}\mathcal K_{4\sqrt s+2j_1+2j_2}\\
		&=\sum_{j_1=0}^\ell\sum_{j_2=0}^\ell\frac{C_{j_1}^{(\ell)} C_{j_2}^{(\ell)}}{2(\sqrt s+j_1+1)}\frac{2\sqrt s+j_1+j_2}{2\sqrt s+j_1+j_2+2}\frac{\sqrt\pi\,(2\sqrt s)_{j_1+j_2}\Gamma(2\sqrt s)}{\bigl(2\sqrt s+\frac{1}{2}\bigr)_{j_1+j_2}\Gamma\bigl(2\sqrt s+\frac{1}{2}\bigr)}.
	\end{align*}
	On the other hand, using \eqref{eqn:lemV1_3F2b} and \eqref{eqn:lemIJ1}, we compute
\begin{align*}
\int_{-\infty}^\infty V_1^{(\ell)}(x)^4\d x
=&\sum_{j_1=0}^\ell\sum_{j_2=0}^\ell
 C_{j_1}^{(\ell)} C_{j_2}^{(\ell)}\mathcal K_{4\sqrt{s}+2j_1+2j_2}\\
 =&
 \sum_{j_1=0}^\ell\sum_{j_2=0}^\ell C_{j_1}^{(\ell)} C_{j_2}^{(\ell)}
 \frac{\sqrt\pi\,(2\sqrt s)_{j_1+j_2}\Gamma(2\sqrt s)}
 {\bigl(2\sqrt s+\frac{1}{2}\bigr)_{j_1+j_2}\Gamma\bigl(2\sqrt s+\frac{1}{2}\bigr)}.
\end{align*}
Thus, we obtain \eqref{eqn:b2_cubic}.
\end{proof}

Further computation of $\bar b_2$ are carried out and its closed-form expressions for $\ell\le 4$
 are given in Appendix~B.

\begin{figure}[t]
	\begin{minipage}{0.495\textwidth}
	\includegraphics[scale=0.55]{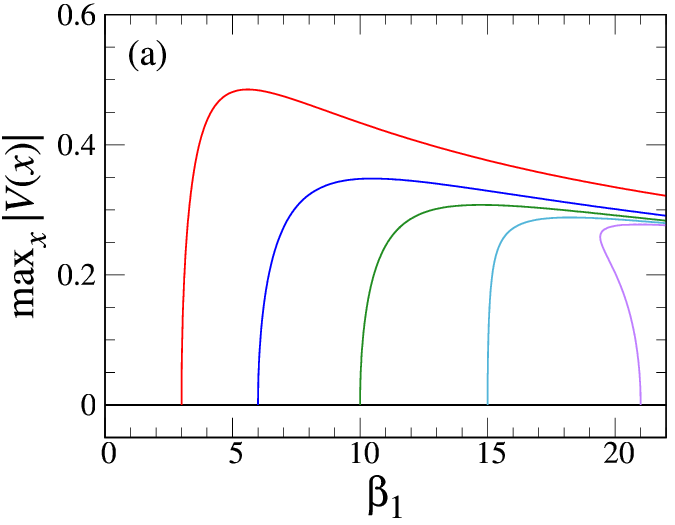}
	\end{minipage}
	\begin{minipage}{0.495\textwidth}
	\begin{center}
	\includegraphics[scale=0.45]{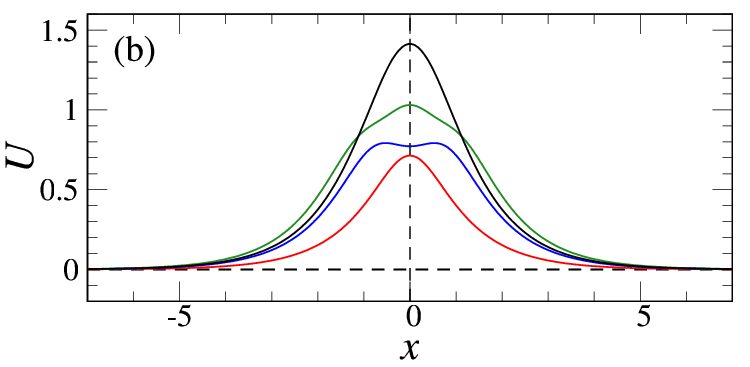}\\[2ex]
	\includegraphics[scale=0.45]{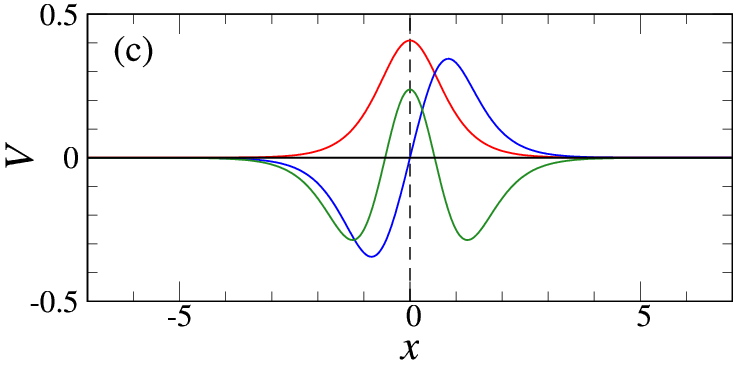}
	\end{center}
	\end{minipage}
	\caption{
		Bifurcations of solitary waves in \eqref{eqn:CNLS}
		for $(\omega,s,\beta_2)=(1,4,2)$:
		(a) Bifurcation diagram;
		(b) and (c): profiles of the corresponding homoclinic solutions
		to \eqref{eqn:UV} on the first three branches at $\beta_1=12$.
		In plate~(a), the red, blue, green, light blue, and purple lines
		represent the branches born at the first, second, third, fourth and fifth
		bifurcations (at $\beta_1=3,6,10,15$, and $21$), respectively,
		while the black line represents the branch of the fundamental solitary wave
		\eqref{eqn:fun_soliton}.
		In plates~(b) and (c), the homoclinic solutions along with $(U,V)=(U_0(x),0)$
		are plotted as the same color lines as the corresponding branches in plate~(a).
		\label{fig:3a}
	}
\end{figure}

Figure~\ref{fig:3a}(a) displays a numerically computed bifurcation diagram for homoclinic orbits in \eqref{eqn:Hsysex} for $(\omega,s,\beta_2)=(1,4,2)$.
Profiles of the homoclinic orbits on the first three bifurcation points
 at $\beta_1=12$ are plotted in Figs.~\ref{fig:3a}(b) and (b).
To obtain these results, we used a similar approach as in Section 5 of \cite{BlYa12}
 and the computation tool \texttt{AUTO} \cite{DO12}.
We observe that pitchfork bifurcations occur at $\beta_1=3,6,10,15,21$
 and they are supercritical except that the fifth one ($\ell=4$) is subcritical,
 as predicted by the theory (see also Fig.~\ref{fig:bifbdry}(a)).
Note that a pair of symmetric branches about $V=0$ are born at each bifurcation point.
Moreover, we see that a saddle-node bifurcation of the solitary waves
 on the fifth branches occurs at $\beta_1\approx 19.4$
 (more precisely, $19.41626\ldots$).
The $V$-component of the homoclinic orbit on the $(\ell+1)$th branch
 have exactly $\ell$ zeros for $\ell=0,1,2$.
The details on the approach and further numerical results are given in \cite{YaYa20}.


\subsection{Stability Analysis}\label{sec:cubic_stab}

We turn to the spectral stability of the fundamental and bifurcated solitary waves.
For the fundamental solitary wave, noting that $\|U_0\|_{L^2}^2=4\sqrt\omega$ for $\omega>0$ and applying Theorem \ref{thm:fun_stab}, we immediately obtain the following result, which was also obtained in Remark 2.6 of \cite{Oh96} when $\beta_2=1$.

\begin{thm}
	The fundamental solitary wave \eqref{eqn:fun_soliton} with $U_0(x)=\sqrt 2\sech x$ is orbitally and spectrally stable.
\end{thm}

We next consider the bifurcated solitary waves \eqref{eqn:bif_soliton} near $\beta_1=\beta_1^{(\ell)}$.
We first give some explicit expressions of the Jost solutions and Evans function for $\mathscr L_0$ given in \eqref{eqn:L0}. 
The eigenvalue problem for $\mathscr L_0$ is separated into three parts as stated in Section 3.1.

\subsubsection*{\rm\bf (A)\nopunct}

\eqref{eqn:lin_eq_1} becomes
\begin{align}
	-p''+p-2\sech^2x\,(2p+q)=\i\lambda p,\quad
	-q''+q-2\sech^2x\,(p+2q)=-\i\lambda q.
	\label{eqn:lin_eq_1_cubic}
\end{align}
The Jost solutions $(p_j,q_j)$, $j=1,2,5,6$, which satisfy our normalization conditions stated in Lemma \ref{lem:Jost_normalize}, for \eqref{eqn:lin_eq_1_cubic} are explicitly given by
\begin{align}
	\begin{aligned}
		&\begin{pmatrix}p_1(x,\lambda)\\q_1(x,\lambda)\end{pmatrix}
		=\frac{\e^{\nu_1x}}{(\nu_1+1)^2}\begin{pmatrix}(\nu_1-\tanh x)^2\\-\sech^2x\end{pmatrix},\quad
		\begin{pmatrix}p_2(x,\lambda)\\q_2(x,\lambda)\end{pmatrix}
		=\begin{pmatrix}q_1(x,-\lambda)\\p_1(x,-\lambda)\end{pmatrix},\\
		&\begin{pmatrix}p_{j+4}(x,\lambda)\\q_{j+4}(x,\lambda)\end{pmatrix}
		=\begin{pmatrix}p_j(-x,\lambda)\\q_j(-x,\lambda)\end{pmatrix},\quad
		j=1,2,
	\end{aligned}
	\label{eqn:p1_cubic}
\end{align}
where $\nu_1=\nu_1(\lambda)$ is given in \eqref{eqn:mu_def1} (see Section III of \cite{Ka90} and Section~10.4.1 of \cite{KaPr13}).
We compute the Evans function \eqref{eqn:E1} for \eqref{eqn:lin_eq_1_cubic} as
\begin{align*}
	E_\mathrm{A}(\lambda)
	=\frac{4\nu_1(\lambda)\nu_2(\lambda)(\nu_1(\lambda)-1)^2(\nu_2(\lambda)-1)^2}{(\nu_1(\lambda)+1)^2(\nu_2(\lambda)+1)^2}
\end{align*}
This implies that \eqref{eqn:lin_eq_1_cubic} has only one eigenvalue $\lambda=0$ and it is of algebraic multiplicity four, and has resonance poles $\lambda=\pm\i$.

\subsubsection*{\rm\bf (B)\nopunct}

\eqref{eqn:lin_eq_3} becomes
\begin{align}
	-p''+sp-2\beta_1\sech^2(x)p=-\i\lambda p.
	\label{eqn:lin_eq_3_cubic}
\end{align}
The Jost solutions $p_j$, $j=4,8$, for \eqref{eqn:lin_eq_3_cubic} are given by
\begin{align}
	p_4(x,\lambda)
	&=\Gamma(\nu_4+1)P_\kappa^{-\nu_4}(-\tanh x)\nonumber\\
	&=\biggl(\frac{\sech x}{2}\biggr)^{\nu_4}\pFq{2}{1}{\nu_4-\kappa,\nu_4+\kappa+1}{\nu_4+1}{\frac{1+\tanh x}{2}},\label{eqn:p4ex}\\
	p_8(x,\lambda)&=p_4(-x,\lambda),\nonumber
\end{align}
where $\nu_4=\nu_4(\lambda)$ is given in \eqref{eqn:mu_def2} and $\kappa=(-1+\sqrt{1+8\beta_1})/2$ (see, e.g., Problem 5 in Section 23 of \cite{LaLi77}).
Note that $\kappa=\sqrt s+\ell$ if $\beta_1=\beta_1^{(\ell)}$.
We compute the Evans function \eqref{eqn:E3} for \eqref{eqn:lin_eq_3_cubic} as
\begin{align}
	E_\mathrm{B}(\lambda)
	=-\frac{2\,\Gamma(\nu_4(\lambda)+1)^2}{\Gamma(\nu_4(\lambda)-\kappa)\Gamma(\nu_4(\lambda)+\kappa+1)},
	\label{eqn:E3_cubic}
\end{align}
where we have used the formulas
\begin{align*}
	\pFq{2}{1}{a,b}{\frac{a+b+1}{2}}{\frac{1}{2}}
	=\frac{\sqrt\pi\,\Gamma\bigl(\frac{a+b+1}{2}\bigr)}{\Gamma\bigl(\frac{a+1}{2}\bigr)\Gamma\bigl(\frac{b+1}{2}\bigr)},\quad
	\Gamma\biggl(\frac{z}{2}\biggr)\Gamma\biggl(\frac{z+1}{2}\biggr)=2^{1-z}\sqrt\pi\,\Gamma(z)
\end{align*}
(see, e.g., Chapter 15 of \cite{AbSt64}).
From \eqref{eqn:E3_cubic} we observe the following.
Let $\beta_1\ge0$.
All eigenvalues of \eqref{eqn:lin_eq_3_cubic} are given by
\begin{equation}
\lambda=\i(s-(\kappa-k)^2),\quad
k\in\{0,1,\ldots,\lfloor\kappa\rfloor\}\setminus\{\kappa\}
\label{eqn:evB}
\end{equation}
and they are simple,
 since $\lambda\in\i(-\infty,s)$ if $\lambda$ is an eigenvalue of \eqref{eqn:lin_eq_3_cubic},
 as stated in Section~2.
Hence, if $\kappa$ is an integer, $\lambda=\i s$ is a resonance pole.
On the other hand, for $\beta_1<0$, \eqref{eqn:lin_eq_3_cubic} has no eigenvalue since $\Re\kappa\in(-1/2,0)$ and $\nu_4(\lambda)>0$ if $\lambda\in\i(-\infty,s)$.

\subsubsection*{\rm\bf (C)\nopunct}

\eqref{eqn:lin_eq_2} becomes
\begin{align}
	-p''+sp-2\beta_1\sech^2(x)p=\i\lambda p.
	\label{eqn:lin_eq_2_cubic}
\end{align}
As stated in Section \ref{sec:fun_sol_stab}, if $\lambda$ is an eigenvalue of \eqref{eqn:lin_eq_3_cubic}, then $-\lambda$ is an eigenvalue of \eqref{eqn:lin_eq_2_cubic}, and $E_\mathrm{C}(\lambda)=E_\mathrm{B}(-\lambda)$.

\begin{rmk}
Using Proposition~\ref{prop:ev_prop} and \eqref{eqn:E_unperturbed}
 along with parts (A), (B), and (C) above,
 we see that $\mathscr L_0$ has  the eigenvalues $\lambda=0$ and
\begin{equation}
\lambda=\pm\i(s-(\kappa-k)^2),\quad
k\in\{0,1,\ldots,\lfloor\kappa\rfloor\}\setminus\{\kappa\}.
\label{eqn:ev}
\end{equation}
The zero eigenvalue is of geometric multiplicity two and of algebraic multiplicity four except at the pitchfork bifurcation points $\beta_1=\beta_1^{(\ell)}$,
 and the associated eigenfunctions and generalized eigenfunctions are given
 by $\varphi_j(x)$ and $\chi_j(x)$, $j=1,2$, with $(U,V)=(U_0(x),0)$
 in \eqref{eqn:ker} and \eqref{eqn:gker}.
Especially, the Jost solutions \eqref{eqn:p1_cubic} with $\lambda=0$ are expressed
 as linear combinations of $\varphi_1(x)$ and $\varphi_2(x)$.
At $\beta_1=\beta_1^{(\ell)}$, the zero eigenvalue is of geometric multiplicity four and of algebraic multiplicity six since $(0,0,V_1^{(\ell)}(x),0)^\T$ and $(0,0,0,V_1^{(\ell)}(x))^\T$ are contained
 by $\Ker\mathscr L_0$ in addition to $\varphi_1(x)$ and $\varphi_2(x)$.
On the other hand, for the eigenvalue $\lambda=\i(s-(\kappa-k)^2)$,
 it follows from \eqref{eqn:p4ex} that
 the associated eigenfunction is given by $(0,0,0,\Psi(x))$, where
\begin{align*}
	\Psi(x)=\biggl(\frac{\sech x}{2}\biggr)^{\kappa-k}\pFq{2}{1}{-\frac{k}{2},\kappa+\frac{-k+1}{2}}{\kappa-k+1}{\sech^2x}
\end{align*}
if $k$ is even, and
\begin{align*}
	\Psi(x)=-\biggl(\frac{\sech x}{2}\biggr)^{\kappa-k}\pFq{2}{1}{-\frac{k-1}{2},\kappa+\frac{-k+2}{2}}{\kappa-k+1}{\sech^2x}\tanh x
\end{align*}
if $k$ is odd.
Similarly, the eigenfunction associated with the eigenvalue $\lambda=-\i(s-(\kappa-k)^2)$
 is given by $(0,0,\Psi(x),0)$.
Eigenfunctions and generalized eigenfunctions of $J\mathcal L$ are obtained
 from these ones by the transformation \eqref{eqn:to_reim}.
\end{rmk}

We set $\beta_1=\beta_1^{(\ell)}-({\bar b_2}/{\bar a_2})\epsilon^2$ as in \eqref{eqn:mu_eps} and consider the linearized operator $\mathscr L_\epsilon$ with \eqref{eqn:cubic} around the bifurcated solitary waves \eqref{eqn:bif_soliton} near $\beta_1=\beta_1^{(\ell)}$.
From \eqref{eqn:p1_cubic} we easily see that
\begin{align*}
	p_j(0,\lambda)p_j'(0,\lambda)+q_j(0,\lambda)q_j'(0,\lambda)
	=\frac{\nu_j(\lambda)(\nu_j(\lambda)-1)^2}{(\nu_j(\lambda)+1)^2},\quad j=1,2,
\end{align*}
so that $\Delta_1(\lambda)\neq 0$ for any $\lambda\in\i\Rset\setminus\{-\i\}$
 and $\Delta_2(\lambda)\neq 0$ for any $\lambda\in\i\Rset\setminus\{\i\}$.
Moreover, \textbf{(A4)} holds with $\bar p(x)=p_2(x,\i)$ and $\bar q(x)=q_2(x,\i)$.
Letting
\begin{align}
	\begin{aligned}
		\mathcal I^{(\ell)}(\mu,\nu)\defeq
		\int_{-\infty}^\infty&\e^{\nu x}V_1^{(\ell)}(x)\bigl((\nu-\tanh x)^2-\sech^2x\bigr)\sech^{\mu+1}x\\
		&\pFq{2}{1}{\mu-\sqrt s-\ell,\mu+\sqrt s+\ell+1}{\mu+1}{\frac{1+\tanh x}{2}}\d x,
	\end{aligned}
	\label{eqn:Iell_def}
\end{align}
we express \eqref{eqn:Ij} as
\begin{align*}
&
	\mathcal I_1(\lambda)=\frac{2^{-\nu_4(\lambda)+\frac{1}{2}}\beta_1^{(\ell)}}{(\nu_1(\lambda)+1)^2}\mathcal I^{(\ell)}(\nu_4(\lambda),\nu_1(\lambda)),\\
&
	\mathcal I_2(\lambda)=\frac{2^{-\nu_4(\lambda)+\frac{1}{2}}\beta_1^{(\ell)}}{(\nu_2(\lambda)+1)^2}\mathcal I^{(\ell)}(\nu_4(\lambda),\nu_2(\lambda)).
\end{align*}
Thus, $\mathcal I_j(\lambda)\neq 0$ is equivalent
 to $\mathcal I^{(\ell)}(\nu_4(\lambda),\nu_j(\lambda))\ne0$ for $j=1,2$.

Let $\lambda_0$ be a zero of $E_\mathrm{B}(\lambda)$ at $\beta_1=\beta_1^{(\ell)}$.
Noting that $\kappa=\sqrt s+\ell$ at $\beta_1=\beta_1^{(\ell)}$ 
 and using \eqref{eqn:ev}, we can write $\lambda_0=-\i k(2\sqrt s+k)$ and $\nu_4(\lambda_0)=\sqrt s+k$ with $k\in\{-\lfloor\sqrt s\rfloor,-\lfloor\sqrt s\rfloor+1,\ldots,\ell\}$.
Here $k=0,1,\ldots,\lfloor\sqrt s\rfloor+\ell$
 have been replaced with  $k=\ell,\ell-1,\ldots,-\lfloor\sqrt s\rfloor$.
Moreover,
\begin{align}
	p_4(x,\lambda_0)=\biggl(\frac{\sech x}{2}\biggr)^{\sqrt s+k}\pFq{2}{1}{-\frac{\ell-k}{2},\sqrt s+\frac{\ell+k+1}{2}}{\sqrt s+k+1}{\sech^2x}
	\label{eqn:p4_cubic_even}
\end{align}
if $\ell-k$ is even, and
\begin{align*}
	p_4(x,\lambda_0)=-\biggl(\frac{\sech x}{2}\biggr)^{\sqrt s+k}\pFq{2}{1}{-\frac{\ell-k-1}{2},\sqrt s+\frac{\ell+k+2}{2}}{\sqrt s+k+1}{\sech^2x}\tanh x
\end{align*}
if $\ell-k$ is odd.
$\lambda_0$ may not be a simple zero of $E(\lambda,0)$ although it is simple for $E_\mathrm{B}(\lambda)$ as stated above.
The following lemma gives necessary and sufficient conditions for $\lambda_0\ne0$ to be simple for $E(\lambda,0)$.
Note that if $\lambda=\lambda_0$ is a zero of $E_\mathrm{B}(\lambda)$, then $\lambda=-\lambda_0$ is a zero of $E_\mathrm{C}(\lambda)$ and that $\lambda=\pm\i$ are zeros of $E_\mathrm{A}(\lambda)$.

\begin{lem}\label{lem:E_simple}
	Let $\ell\in\Zset_{\ge0}$, $\beta_1=\beta_1^{(\ell)}$, and $\lambda_0(k)=-\i k(2\sqrt s+k)$ for $k\in\{-\lfloor\sqrt s\rfloor,-\lfloor\sqrt s\rfloor+1,\ldots,\ell\}$.
	The following statements hold.
	\begin{itemize}
	\setlength{\leftskip}{-1.8em}
	\item[(i)]
		When $k\in\{-\lfloor\sqrt s\rfloor,\ldots,-1\}$, $\lambda_0(k)=-\lambda_0(k')$ for some $k'\in\{-\lfloor\sqrt s\rfloor,-\lfloor\sqrt s\rfloor+1,\ldots,\ell\}$ if and only if
		\begin{align*}
			\sqrt s=-\frac{k^2+(k')^2}{2(k+k')}\quad\text{and}\quad
			k'\in\bigl\{-\lfloor(\sqrt 2-1)k\rfloor,\ldots,\min\{-k-1,\ell\}\bigr\}.
		\end{align*}
	\item[(ii)]
		When $k\in\{1,\ldots,\ell\}$,
		$\lambda_0(k)=-\lambda_0(k')$ for some $k'\in\{-\lfloor\sqrt s\rfloor,-\lfloor\sqrt s\rfloor+1,\ldots,\ell\}$ if and only if
		\begin{align*}
			\sqrt s=-\frac{k^2+(k')^2}{2(k+k')}\quad\text{and}\quad
			k'\in\{-\lfloor(\sqrt 2+1)k\rfloor,\ldots,-k-1\}.
		\end{align*}
	\item[(iii)]
		$\lambda_0(k)=\i$ if and only if $s=1$ and $k=-1$.
	\item[(iv)]
		$\lambda_0(k)\ne-\i$ for any values of $\ell,s$, and $k$.
	\end{itemize}
\end{lem}



The proof of Lemma \ref{lem:E_simple} is given in Appendix \ref{sec:proof_E_simple}.
Applying Theorems \ref{thm:main_1}, \ref{thm:main_2}, and \ref{thm:main_3} and using Lemma \ref{lem:E_simple}, we obtain the following proposition.

\begin{prop}\label{prop:cubic_main}
	Let $\omega=1$, $s>0$, and $\ell\in\Zset_{\ge0}$.
	Let $\lambda_0$ be a zero of $E_\mathrm{B}(\lambda)$ at $\beta_1=\beta_1^{(\ell)}$ and suppose that $\bar b_2\ne0$.
	For $\epsilon>0$ sufficiently small, the following statements hold.
	\begin{itemize}
	\setlength{\leftskip}{-2em}
	\item[(i)]
		Let $k\in\{-\lfloor\sqrt s\rfloor,\ldots,-1\}$.
		Suppose that
		\begin{align*}
			\sqrt s\notin\left\{-\frac{k^2+(k')^2}{2(k+k')}\;\middle|\;k'\in\bigl\{-\lfloor(\sqrt 2-1)k\rfloor,\ldots,\min\{-k-1,\ell\}\bigr\}\right\}
		\end{align*}
		if $k\ne-\sqrt s$, and $s\ne1$ otherwise.
		If $\mathcal I^{(\ell)}(\nu_4(\lambda_0),\nu_2(\lambda_0))\ne0$, then $\mathscr L_\epsilon$ has neither eigenvalues nor resonance poles near $\lambda_0=-\i k(2\sqrt s+k)$.
	\item[(ii)]
		Let $k\in\{1,\ldots,\ell\}$.
		Suppose that
		\begin{align*}
			\sqrt s\notin\left\{-\frac{k^2+(k')^2}{2(k+k')}\;\middle|\;k'\in\{-\lfloor(\sqrt 2+1)k\rfloor,\ldots,-k-1\}\right\}.
		\end{align*}
		If $\mathcal I^{(\ell)}(\nu_4(\lambda_0),\nu_1(\lambda_0))\ne0$, then there exists a simple eigenvalue $\lambda(\epsilon)$ of $\mathscr L_\epsilon$ with $\lambda(0)=\lambda_0=-\i k(2\sqrt s+k)$ such that
		\begin{align}
			&\Re\lambda(\epsilon)=-\frac{|\mathcal I_-(\lambda_0)|^2}{2\,\Im\nu_1(\lambda_0)\|p_4(\cdot,\lambda_0)\|_{L^2(\Rset)}^2}\epsilon^2+O(\epsilon^3),\label{eqn:cubic_relambda}\\
			&\Im\lambda(\epsilon)=-k(2\sqrt s+k)+O(\epsilon^2).\nonumber
		\end{align}
	\end{itemize}
\end{prop}

\begin{proof}
	For part (i), using Lemma \ref{lem:E_simple} (i), (iii), and (iv), we see that $\lambda=\lambda_0$ is a simple zero of $E(\lambda,0)$.
	Hence, from Theorems \ref{thm:main_1} and \ref{thm:main_2} we immediately obtain the desired result.
	For part (ii), by Lemma \ref{lem:E_simple} (ii), (iii), and (iv), we see that $\lambda_0$ is simple.
	Hence, by Theorem \ref{thm:main_2}, there exists a zero $\lambda(\epsilon)$ of $\mathscr L_\epsilon$ with a positive real part.
	\eqref{eqn:cubic_relambda} follows from \eqref{eqn:Relambda_2} since $C_4=0$ as easily checked by \eqref{eqn:p1_cubic}.
\end{proof}

\begin{rmk}\label{rmk:cubic_main}
At $\beta_1=\beta_1^{(\ell)}$, the eigenvalues \eqref{eqn:ev} becomes
\begin{align}
	\lambda=&\pm\i(s-(\sqrt{s}+\ell-k)^2),\quad
	k\in\{0,1,\ldots,\lfloor\sqrt{s}\rfloor+\ell\}\setminus\{\sqrt{s}+\ell\}.
	\label{eqn:lambda}
\end{align}
From Proposition~\ref{prop:cubic_main}
 and the relation $E_\mathrm{C}(\lambda)=E_\mathrm{B}(-\lambda)$
 we see that these eigenvalues immediately disappear
 (resp.\ become a pair of eigenvalues with nonzero real parts)
  for $k>\ell$ (resp.\ $k<\ell$) when $\epsilon>0$.
For $k=\ell$ the eigenvalue $\lambda=0$ remains the origin even when $\epsilon>0$ since $\dim\Ker\mathscr L_\epsilon=3$ for $\epsilon>0$.
A basis of $\Ker\mathscr L_\epsilon$ is given by \eqref{eqn:ker}.
\end{rmk}


Let $\mathcal K_r(\alpha)$ be the Fourier tranform of $\sech^rx$, i.e.,
\begin{align*}
	\mathcal K_r(\alpha)\defeq\int_{\Rset}\e^{-\i\alpha x}\sech^rx\,\d x,
\end{align*}
where $r>0$ and $\alpha\in\Rset$.
The change of variables $x=\frac{1}{2}\log t$ yields
\begin{align*}
	\mathcal K_r(\alpha)
	=2^{r-1}\int_0^\infty\frac{t^{\frac{1}{2}(r-\i\alpha)-1}}{(t+1)^r}\,\d t
	=2^{r-1}B\biggl(\frac{r-\i\alpha}{2},\frac{r+\i\alpha}{2}\biggr)
	=2^{r-1}\frac{\Gamma\bigl(\frac{r-\i\alpha}{2}\bigr)\Gamma\bigl(\frac{r+\i\alpha}{2}\bigr)}{\Gamma(r)}
\end{align*}
along with \eqref{eqn:Gamma_Beta}.
Since $\Gamma(z^*)=\Gamma(z)^*$, we obtain
\begin{align}
	\mathcal K_r(\alpha)
	=\frac{2^{r-1}}{\Gamma(r)}\Bigl\lvert\Gamma\Bigl(\frac{r\pm\i\alpha}{2}\Bigr)\Bigr\rvert^2.\label{eqn:Kr_Fourier}
\end{align}
We have useful expressions of $\mathcal I^{(\ell)}(\nu_4(\lambda_0),\nu_j(\lambda_0))$, $\ell\in\Zset_{\ge0}$, $j=1,2$, with $\mathcal K_r(\alpha)$ for some special values of $\lambda_0$ as follows.

\begin{lem}\label{lem:concrete_comp}\ 
	\begin{itemize}
	\setlength{\leftskip}{-1.8em}
	\item[(i)]
		Let $\lambda_0=-\i k(2\sqrt s+k)$ with $k\in\{-\lfloor\sqrt s\rfloor,\ldots,-1\}$.
		We have
		\begin{align*}
			&\mathcal I^{(0)}(\nu_4(\lambda_0),\nu_2(\lambda_0))\\
			&=\begin{cases}
				\displaystyle\frac{(-1)^{k/2}k(2\sqrt s+k)}{(\sqrt s+k+1)_{-k}}\biggl(\prod_{j=0}^{-k/2-1}\biggl(\biggl(j+\frac{1}{2}\biggr)^2+\frac{\nu^2}{4}\biggr)\biggr)\mathcal K_{2\sqrt s+k+1}(\nu)&\text{if $k$ is even};\\
				\displaystyle\frac{\i\nu}{2}\frac{(-1)^{(k-1)/2}k(2\sqrt s+k)}{(\sqrt s+k+1)_{-k}}\biggl(\prod_{j=1}^{-(k+1)/2}\biggl(j^2+\frac{\nu^2}{4}\biggr)\biggr)\mathcal K_{2\sqrt s+k+1}(\nu)&\text{if $k$ is odd},
			\end{cases}
		\end{align*}
		where $\nu=\sqrt{-k(2\sqrt s+k)-1}$.
	\item[(ii)]
		Let $\ell\in\Nset$ and $\lambda_0=-\i\ell(2\sqrt s+\ell)$.
		We have
		\begin{align*}
			&\mathcal I^{(\ell)}(\nu_4(\lambda_0),\nu_1(\lambda_0))\\
			&=\begin{cases}
				\displaystyle\frac{(-1)^{\ell/2+1}\ell(2\sqrt s+\ell)}{(\sqrt s+1)_\ell}\biggl(\prod_{j=0}^{\ell/2-1}\biggl(\biggl(j+\frac{1}{2}\biggr)^2+\frac{\nu^2}{4}\biggr)\biggr)\mathcal K_{2\sqrt s+\ell+1}(\nu)&\text{if $\ell$ is even};\\
				\displaystyle\frac{\i\nu}{2}\frac{(-1)^{(\ell+1)/2}\ell(2\sqrt s+\ell)}{(\sqrt s+1)_\ell}\biggl(\prod_{j=1}^{(\ell-1)/2}\biggl(j^2+\frac{\nu^2}{4}\biggr)\biggr)\mathcal K_{2\sqrt s+\ell+1}(\nu)&\text{if $\ell$ is odd},
			\end{cases}
		\end{align*}
		where $\nu=\sqrt{\ell(2\sqrt s+\ell)-1}$.
	\item[(iii)]
		Let $s\in(0,1)$.
		We have
		\begin{align*}
			\mathcal I^{(0)}(\nu_4(\i),\nu_2(\i))
			=-\frac{\sqrt\pi\,\Gamma(1+\i\nu)\Gamma\bigl(\frac{\sqrt s+1+\i\nu}{2}\bigr)\Gamma\bigl(\frac{\sqrt s+1-\i\nu}{2}\bigr)}{\Gamma(\sqrt s+1)\Gamma\bigl(\frac{\sqrt s+2+\i\nu}{2}\bigr)\Gamma\bigl(\frac{-\sqrt s+1+\i\nu}{2}\bigr)},
		\end{align*}
		where $\nu=\sqrt{1-s}$.
	\end{itemize}
\end{lem}

The proof of Lemma \ref{lem:concrete_comp} is given in Appendix \ref{sec:proof_comp}.
Using Proposition \ref{prop:cubic_main} and Lemma \ref{lem:concrete_comp}, we obtain the following theorem.

\begin{thm}\label{thm:cubic_main}
	For each $\ell\in\Zset_{\ge 0}$, suppose that $\bar b_2\ne0$
	and consider the solitary wave \eqref{eqn:bif_soliton} bifurcated at $\beta_1=\beta_1^{(\ell)}$.
	Let $\mathscr L_\epsilon$ be the linearized operator \eqref{eqn:lin_op} with \eqref{eqn:cubic} around the bifurcated solitary waves.
	\begin{itemize}
	\setlength{\leftskip}{-2em}
	\item[(i)]
		Let $\ell=0$.
		The bifurcated solitary waves are orbitally and spectrally stable.
		Moreover, for $\epsilon>0$ sufficiently small, there exists $\delta>0$ such that $\mathscr L_\epsilon$ has neither eigenvalues nor resonance poles except for $\{0\}\cup\i(-\min\{1,s\}-\delta,-\min\{1,s\}+\delta)\cup\i(\min\{1,s\}-\delta,\min\{1,s\}+\delta)$.
	\item[(ii)]
		Let $\ell>0$.
		If
		\begin{align*}
			\sqrt s\notin\left\{-\frac{\ell^2+k^2}{2(\ell+k)}\;\middle|\;k\in\{-\lfloor(\sqrt 2+1)\ell\rfloor,\ldots,-\ell-1\}\right\},
		\end{align*}
		then for $\epsilon>0$ sufficiently small, $\mathscr L_\epsilon$ has a pair of eigenvalues with positive and negative real parts near $\lambda=\pm\i\ell(2\sqrt s+\ell)$, i.e., the bifurcated solitary waves are spectrally unstable.
	\end{itemize}
\end{thm}

\begin{proof}
	The orbital and spectral stability of the bifurcated solitary waves with $\ell=0$ immediately follows from Theorem~\ref{thm:bif_orb_stab}.
	Using Proposition \ref{prop:cubic_main} (i) and Lemma \ref{lem:concrete_comp} (i), we see that $\mathscr L_\epsilon$ has neither eigenvalues nor resonance poles near zeros of $E_\mathrm{B}(\lambda)$ except for $\lambda=0,\pm\i\min\{1,s\}$, since $\mathcal K_{2\sqrt s+k+1}(\nu)\ne0$ with $\nu=\sqrt{-k(2\sqrt s+k)-1}$, $k\in\{-\lfloor\sqrt s\rfloor,\ldots,-1\}$, by \eqref{eqn:Kr_Fourier}.
	Using Theorem \ref{thm:main_4} and Lemma \ref{lem:concrete_comp} (iii), we see that $\mathscr L_\epsilon$ also has neither eigenvalues nor resonance poles near zeros of $E_\mathrm{A}(\lambda)$ except for $\lambda=0,\pm\i\min\{1,s\}$.
	Moreover, the zeros $\lambda=\pm\i\min\{1,s\}$ of $E(\lambda,0)$ remains purely imaginary by Theorem~\ref{thm:bif_orb_stab}.
	Thus we obtain part (i).
	For part (ii), using Proposition \ref{prop:cubic_main} (ii) and Lemma \ref{lem:concrete_comp} (ii), we see that $\mathscr L_\epsilon$ has an eigenvalue with a positive real part near a zero of $E_\mathrm{B}(\lambda)$ with $k=\ell$, since $\mathcal K_{2\sqrt s+\ell+1}(\nu)\ne0$ with $\nu=\sqrt{\ell(2\sqrt s+\ell)-1}$ by \eqref{eqn:Kr_Fourier}.
\end{proof}

\begin{rmk}
	The nondegenerate condition $\mathcal I^{(\ell)}\ne0$ in Proposition \ref{prop:cubic_main} is not always true.
	For example, a direct calculation shows
	\begin{align*}
		\mathcal I^{(2)}(\nu_4(\lambda_0),\nu_2(\lambda_0))=-\frac{4(\sqrt s-1)(\sqrt s-3)}{s(\sqrt s+2)(\sqrt s+3)(2\sqrt s-1)}\mathcal K_{2\sqrt s-1}\Bigl(\sqrt{4\sqrt s-5}\Bigr),
	\end{align*}
	where $\lambda_0=-\i k(2\sqrt s+k)$ and $k=-2$.
	It vanishes at $\sqrt s=3$.
	Hence, we could not generally determine whether eigenvalues of $\mathscr L_\epsilon$ near $\lambda_0=-\i k(2\sqrt s+k)$ with $k\ne\ell$ have positive real parts, in Theorem \ref{thm:cubic_main} (ii).
\end{rmk}

Finally, we give some numerical computation results for $(\omega,s,\beta_2)=(1,4,2)$.
To obtain these results, we also used the computation tool \texttt{AUTO} \cite{DO12}.
See \cite{YaYa20} for the details of our numerical approach and further numerical results.
Figure \ref{fig:lam} shows how the eigenvalues of $\mathscr L_\epsilon$
 for the bifurcated solitary wave on each branch change.
We only display the eigenvalues with $\Re\lambda,\Im\lambda\ge0$
 since the spectra of $\mathscr L_\epsilon$ are symmetric, as stated in Section~2.
For the bifurcated solitary wave born at $\beta_1=\beta_1^{(\ell)}$,
 all eigenvalues of $\mathscr L_\epsilon$ at $\beta_1=\beta_1^{(\ell)}$
 are given by \eqref{eqn:lambda} and they immediately disappear for $k\ge\ell$
 when $\beta_1$ changes from $\beta_1^{(\ell)}$ (see Remark~\ref{rmk:cubic_main}).
The loci of the eigenvalues of $\mathscr L_\epsilon$ 
 leaving the imaginary axis from \eqref{eqn:lambda}
 when $\beta_1$ changes from $\beta_1^{(\ell)}$
 are plotted for $k=0$ and $\ell=1$ in Fig.~\ref{fig:lam}(a);
 for $k=0,1$ and $\ell=2$ in Fig.~\ref{fig:lam}(b);
 for $k=0,1,2$ and $\ell=3$ in Fig.~\ref{fig:lam}(c);
 and $k=0,1,2,3$ and $\ell=4$ in Figs.~\ref{fig:lam}(d) and (e).
Each curve was computed from $\beta_1=\beta_1^{(\ell)}$ to $100$.
These results indicate that
 the real parts of the eigenvalues become positive
 and the bifurcated solitary waves are unstable
 when $\beta_1\ne\beta_1^{(\ell)}$, as stated in Theorem~\ref{thm:cubic_main}.

In Figs.~\ref{fig:lam}(d) and (e)
 four eigenvalues for the bifurcated solitary wave on the fifth branch ($\ell=4$)
 are displayed and their values at $\beta_1=\beta_1^\mathrm{SN}\approx 19.41626$,
 at which a saddle-node bifurcation occurs (see Fig.~\ref{fig:3a}(a)),
 are plotted as the circle `$\circ$'.
Especially, the solitary wave seems not to change its stability type
 at the saddle-node bifurcation point
 since all the eigenvalues are far from the imaginary axis.
On the other hand, 
 the computed eigenvalues at $\beta_1=50$ and $\beta_1=100$ were almost the same.
So the eigenvalues are thought to converge to certain values as $\beta_1\to\infty$.
See, e.g., the blue line ($k=1$) in Fig.~\ref{fig:lam}(b),
 the blue and green lines ($k=1$ and $2$) in Fig.~\ref{fig:lam}(c),
 and the green and purple lines ($k=2$ and $3$)  in Fig.~\ref{fig:lam}(d).


\begin{figure}
	\begin{minipage}{0.495\textwidth}
		\includegraphics[scale=0.55]{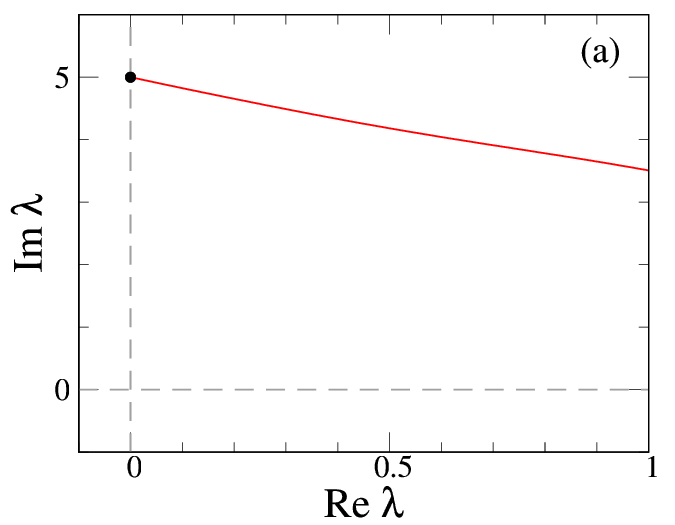}\\[1ex]
		\includegraphics[scale=0.55]{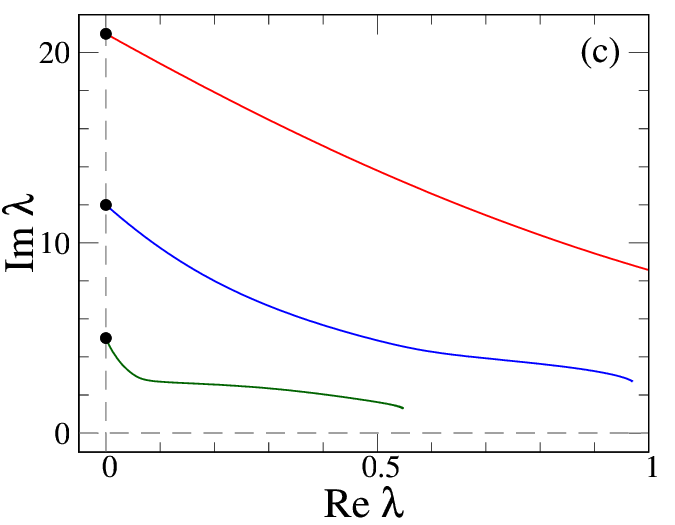}
	\end{minipage}
	\begin{minipage}{0.495\textwidth}
	\begin{center}
		\includegraphics[scale=0.55]{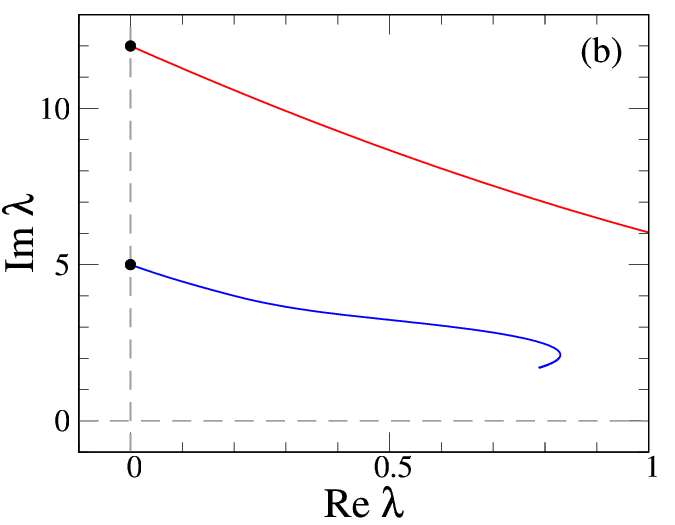}\\[1ex]
		\includegraphics[scale=0.55]{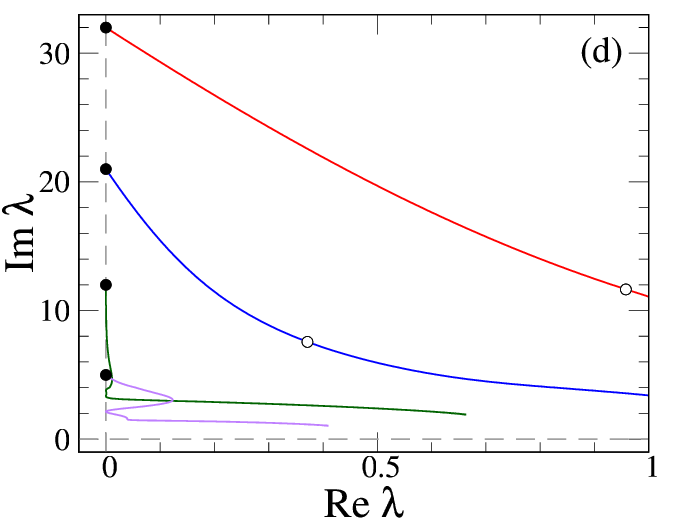}
	\end{center}
	\end{minipage}
	\includegraphics[scale=0.55]{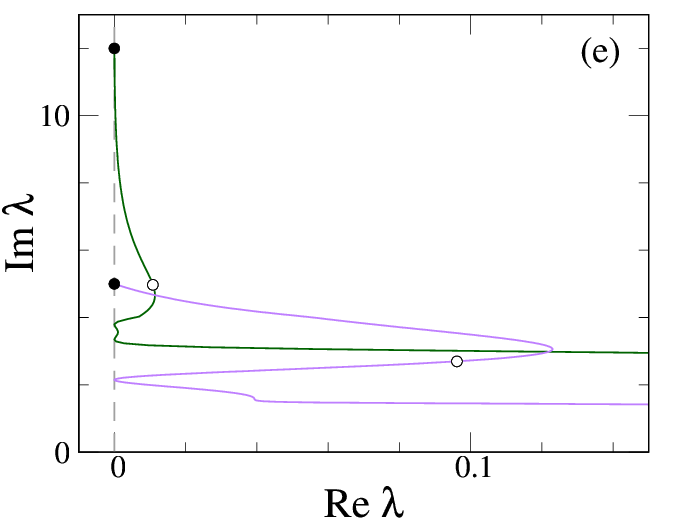}
	\caption{
		Eigenvalues of the linearized operator $\mathscr L_\epsilon$ around the bifurcated solitary waves born at $\beta_1=\beta_1^{(\ell)}$ for $(\omega,s,\beta_2)=(1,4,2)$:
		(a) $\ell=1$; (b) $\ell=2$; (c) $\ell=3$; (d) and (e) $\ell=4$.
Plate~(e) is an enlargement of plate~(d).
The red, blue, green, and purple lines represent the eigenvalues of $\mathscr L_\epsilon$ with $k=0,1,2$, and $3$, respectively,
The bullet `$\bullet$' represents the loci of the eigenvalues
 at $\beta_1=\beta_1^{(\ell)}$.
In plates (d) and (e), the circle `$\circ$' represents the loci of the eigenvalues
 at the saddle-node bifurcation point $\beta_1\approx19.41626$.
		\label{fig:lam}
	}
\end{figure}

\appendix

\renewcommand{\theequation}{\Alph{section}.\arabic{equation}}
\renewcommand{\thefigure}{\Alph{section}.\arabic{figure}}
\setcounter{equation}{0}
\setcounter{figure}{0}


\section{Proofs of Lemmas \ref{lem:Jost_normalize} and \ref{lem:de2E}}


\subsection{Proof of Lemma \ref{lem:Jost_normalize}}\label{sec:proof_Jost_nomalize}

Because of the analiticity of the Jost solutions, without loss of generality we assume that $\lambda$ is sufficiently close to $\lambda_0$.
We now determine $Y_{0j}^\pm$ using \eqref{eqn:Jost_flip} for $j=1,2,5,6$.

For any solution $Y(x)$ of \eqref{eqn:ODE0} and $Z(x)$ of \eqref{eqn:adj_eq}, we have
\begin{align*}
	(Y\cdot Z)'
	=Y'\cdot Z+Y\cdot Z'
	=AY\cdot Z-Y\cdot A^\H Z
	=AY\cdot Z-AY\cdot Z
	=0.
\end{align*}
Hence $\mathscr Z_0(x)^\H\mathscr Y_0(x)$ is independent of $x$ so that
\begin{align}
	\mathscr Y_0(x)^{-1}=(\mathscr Z_0(0)^\H\mathscr Y_0(0))^{-1}\mathscr Z_0(x)^\H,
	\label{eqn:Y0inv_tmp}
\end{align}
where
\begin{align}
	\mathscr Z_0(x)
	&=\bigl(Z_{01}^-\;\;Z_{02}^-\;\;Z_{03}^-\;\;\bar Z_{04}\;\;Z_{05}^+\;\;Z_{06}^+\;\;Z_{07}^+\;\;Z_{08}^+\bigr)(x,\lambda_0)\\
	&=\left(\begin{array}{c|c|c|c}
		\begin{matrix}-p_1'&-q_2'\\-q_1'&-q_2'\end{matrix}&O_2&\begin{matrix}-p_5'&-q_6'\\-q_5'&-q_6'\end{matrix}&O_2\\\hline
		O_2&\begin{matrix}-p_3'&0\\0&-\bar p_4'\end{matrix}&O_2&\begin{matrix}-p_7'&0\\0&-p_8'\end{matrix}\\\hline
		\begin{matrix}p_1&q_2\\q_1&q_2\end{matrix}&O_2&\begin{matrix}p_5&q_6\\q_5&q_6\end{matrix}&O_2\\\hline
		O_2&\begin{matrix}p_3&0\\0&\bar p_4\end{matrix}&O_2&\begin{matrix}p_7&0\\0&p_8\end{matrix}
	\end{array}\right)^*\!\!\!(x,\lambda_0).
	\label{eqn:Z0_def}
\end{align}
Here we have used \eqref{eqn:Jost_0} and \eqref{eqn:Y0_def}.
Since $\Re\nu_1>\Re\nu_2=0$, we easily see that
\begin{itemize}
	\item $Y_{01}^-$ is unique;
	\item $Y_{02}^-$ and $Y_{06}^-$ are unique up to adding elements of $\lspan_\Cset\{Y_{01}^-\}$;
	\item $Y_{05}^-$ is unique up to adding elements of $\lspan_\Cset\{Y_{01}^-,Y_{02}^-,Y_{06}^-\}$.
\end{itemize}
By \eqref{eqn:Jost_0}, \eqref{eqn:Jost_flip}, and $\textbf{(A4-1)}_{\lambda_0}$, $C_1=(Y_{01}^-\cdot Z_{05}^+)(\lambda_0)=2(p_1p_1'+q_1q_1')(0,\lambda_0)\ne0$.
Thus, we uniquely determine $Y_{02}^-$ and $Y_{06}^-$ under the conditions $Y_{02}^-\cdot Z_{05}^+=0$ and $Y_{06}^-\cdot Z_{05}^+=0$.
Using \eqref{eqn:caseA_asym} and \eqref{eqn:Jost_flip}, we have
\begin{align}
	\mathscr Z_0(0)^\H\mathscr Y_0(0)
	=\left(\begin{array}{c|c|c|c}
		O_2&O_2&\begin{matrix}-C_1&0\\0&-C_2\end{matrix}&O_2\\\hline
		O_2&\ast_{11}&O_2&\ast_{12}\\\hline
		\begin{matrix}C_1&0\\0&C_2\end{matrix}&O_2&O_2&O_2\\\hline
		O_2&\ast_{21}&O_2&\ast_{22}
	\end{array}\right),
	\label{eqn:Z0Y0}
\end{align}
where $\ast_{ij}$, $i,j=1,2$, are some $2\times2$ matrices.
Here we have used \eqref{eqn:Jost_0} and \eqref{eqn:Y0_def} again.
Since any fundamental matrix solutions are nonsingular, we obtain
\begin{align*}
	0\ne\det\mathscr Z_0(0)^\H\mathscr Y_0(0)
	=C_1^2C_2^2\det\!\begin{pmatrix}\ast_{11}&\ast_{12}\\\ast_{21}&\ast_{22}\end{pmatrix}
\end{align*}
so that $C_2\ne0$.
Noting that $C_1,C_2\ne0$ and $Y_{06}^-\cdot Z_{02}^-=-2\nu_2\ne0$, we uniquely determine $Y_{05}^-$ under the conditions $Y_{05}^-\cdot Z_{05}^+=0$, $Y_{05}^-\cdot Z_{06}^+=0$, and $Y_{05}^-\cdot Z_{02}^-=0$.

Moreover, we compute the inverse of \eqref{eqn:Z0Y0} as
\begin{align}
	(\mathscr Z_0(0)^\H\mathscr Y_0(0))^{-1}
	=\left(\begin{array}{c|c|c|c}
		O_2&O_2&\begin{matrix}1/C_1&0\\0&1/C_2\end{matrix}&O_2\\\hline
		O_2&\ast_{11}&O_2&\ast_{12}\\\hline
		\begin{matrix}-1/C_1&0\\0&-1/C_2\end{matrix}&O_2&O_2&O_2\\\hline
		O_2&\ast_{21}&O_2&\ast_{22}
	\end{array}\right),
	\label{eqn:Z0Y0_inv}
\end{align}
where $\ast_{ij}$ is different from that of \eqref{eqn:Z0Y0} for $i,j=1,2$.
Substituting \eqref{eqn:Z0_def} and \eqref{eqn:Z0Y0_inv} into \eqref{eqn:Y0inv_tmp}, we obtain the $(i,j)$-components of \eqref{eqn:Y0_inv}, $i,j=1,2,5,6$.
The other components of \eqref{eqn:Y0_inv} are obtained directly from \eqref{eqn:Y0_def} with \eqref{eqn:Jost_0}.
\qed


\subsection{Proof of Lemma \ref{lem:de2E}}\label{sec:proof_de2E}

Differentiating \eqref{eqn:def_evans} twice with respect to $\epsilon$ and using \eqref{eqn:p48} and \eqref{eqn:deY_zero}, we have $\partial_\epsilon^2E(\lambda_0,0)=f(x)+2g(x)$ where
\begin{align}
	f(x)&=\det\bigl(Y_1^-\;\;Y_2^-\;\;Y_3^-\;\;\partial_\epsilon^2(Y_4^--\tau Y_8^+)\;\;Y_5^+\;\;Y_6^+\;\;Y_7^+\;\;Y_8^+\bigr)(x,\lambda_0,0),\label{eqn:f_def}\\
	g(x)&=\det\bigl(\partial_\epsilon Y_1^-\;\;Y_2^-\;\;Y_3^-\;\;\partial_\epsilon(Y_4^--\tau Y_8^+)\;\;Y_5^+\;\;Y_6^+\;\;Y_7^+\;\;Y_8^+\bigr)(x,\lambda_0,0)\nonumber\\
	&\quad+\det\bigl(Y_1^-\;\;\partial_\epsilon Y_2^-\;\;Y_3^-\;\;\partial_\epsilon(Y_4^--\tau Y_8^+)\;\;Y_5^+\;\;Y_6^+\;\;Y_7^+\;\;Y_8^+\bigr)(x,\lambda_0,0)\nonumber\\
	&\quad+\det\bigl(Y_1^-\;\;Y_2^-\;\;Y_3^-\;\;\partial_\epsilon(Y_4^--\tau Y_8^+)\;\;\partial_\epsilon Y_5^+\;\;Y_6^+\;\;Y_7^+\;\;Y_8^+\bigr)(x,\lambda_0,0)\nonumber\\
	&\quad+\det\bigl(Y_1^-\;\;Y_2^-\;\;Y_3^-\;\;\partial_\epsilon(Y_4^--\tau Y_8^+)\;\;Y_5^+\;\;\partial_\epsilon Y_6^+\;\;Y_7^+\;\;Y_8^+\bigr)(x,\lambda_0,0).\label{eqn:g_def}
\end{align}
Note that $f(x)+2g(x)$ does not depend on $x$.

We first compute $g(x)$.
Substituting \eqref{eqn:rho1} into \eqref{eqn:g_def} and using \eqref{eqn:Jost_0}, we have
\begin{align}
	g(x)
	&=-\det\bigl(Y_1^-\;\;Y_2^-\;\;Y_3^-\;\;\sum_{j=1}^2(\rho_j\partial_\epsilon Y_j^-+\rho_{j+4}\partial_\epsilon Y_{j+4}^+)\;\;Y_5^+\;\;Y_6^+\;\;Y_7^+\;\;Y_8^+\bigr)(x,\lambda_0,0)\nonumber\\
	&=-E_\mathrm{A}(\lambda_0)E_\mathrm{B}(\lambda_0)\det\!\begin{pmatrix}
		\sum_{j=1}^2(\rho_j\partial_\epsilon Y_j^-+\rho_{j+4}\partial_\epsilon Y_{j+4}^+)_4&p_8\\
		\sum_{j=1}^2(\rho_j\partial_\epsilon Y_j^-+\rho_{j+4}\partial_\epsilon Y_{j+4}^+)_8&p_8'\\
	\end{pmatrix}\!(x,\lambda_0,0)
	\label{eqn:g_tmp}
\end{align}
with
\begin{align*}
	\rho_j=\tau\left(\int_{-\infty}^\infty\mathscr Y_0(y)^{-1}A_1(y)Y_{08}^+(y,\lambda_0)\,\d y
\right)_j,\quad j=1,2,5,6,
\end{align*}
where $(\cdot)_j$ represents the $j$-th component of the related vector in $\Cset^8$.
On the other hand, substitution of \eqref{eqn:A1_def} and \eqref{eqn:Y0_inv} into \eqref{eqn:deY_int} yields
\begin{align*}
	(\partial_\epsilon Y_j^-)_4(x,\lambda_0,0)
	=&\left(\int_{-\infty}^xa(y)p_8(y,\lambda_0)(p_j+q_j)(y,\lambda_0)\,\d y\right)\bar p_4(x)\\
	&-\left(\int_{-\infty}^xa(y)\bar p_4(y)(p_j+q_j)(y,\lambda_0)\,\d y\right)p_8(x,\lambda_0),\quad j=1,2,\\
	(\partial_\epsilon Y_j^-)_8(x,\lambda_0,0)
	=&\left(\int_{-\infty}^xa(y)p_8(y,\lambda_0)(p_j+q_j)(y,\lambda_0)\,\d y\right)\bar p_4'(x)\\
	&-\left(\int_{-\infty}^xa(y)\bar p_4(y)(p_j+q_j)(y,\lambda_0)\,\d y\right)p_8'(x,\lambda_0),\quad j=1,2,\\
	(\partial_\epsilon Y_j^+)_4(x,\lambda_0,0)
	=&-\left(\int_x^\infty a(y)p_8(y,\lambda_0)(p_j+q_j)(y,\lambda_0)\,\d y\right)\bar p_4(x)\\
	&+\left(\int_x^\infty a(y)\bar p_4(y)(p_j+q_j)(y,\lambda_0)\,\d y\right)p_8(x,\lambda_0),\quad j=5,6,\\
	(\partial_\epsilon Y_j^+)_8(x,\lambda_0,0)
	=&-\left(\int_x^\infty a(y)p_8(y,\lambda_0)(p_j+q_j)(y,\lambda_0)\,\d y\right)\bar p_4'(x)\\
	&+\left(\int_x^\infty a(y)\bar p_4(y)(p_j+q_j)(y,\lambda_0)\,\d y\right)p_8'(x,\lambda_0),\quad j=5,6.
\end{align*}
Substituting the above equations into \eqref{eqn:g_tmp}, we obtain
\begin{align*}
	g(x)
	=E_\mathrm{A}(\lambda_0)E_\mathrm{B}(\lambda_0)\sum_{j=1}^2\biggl(
	&-\rho_j\int_{-\infty}^x a(y)p_8(y)\bigl(p_j(y)+q_j(y)\bigr)\d y\\
	&+\rho_{j+4}\int_x^\infty a(y)p_8(y)\bigl(p_{j+4}(y)+q_{j+4}(y)\bigr)\d y\biggr)
\end{align*}
at $\lambda=\lambda_0$.

We next compute $f(x)$.
Since $\partial_\epsilon^2 Y_j^\pm(x,\lambda_0,0)$ satisfies
\begin{align*}
	(\partial_\epsilon^2 Y_j^\pm)'=A_0(x,\lambda_0)\partial_\epsilon^2Y_j^\pm+2A_1(x)\partial_\epsilon Y_j^\pm+2A_2(x)Y_{0j}^\pm,
\end{align*}
we have
\begin{align*}
	\partial_\epsilon^2(Y_4^--\tau Y_8^+)(x,\lambda_0,0)
	&=2\mathscr Y_0(x)\biggl(
		\int_{-\infty}^x\mathscr Y_0(y)^{-1}A_1(y)\partial_\epsilon Y_4^-(y,\lambda_0,0)\,\d y\\
		&\qquad\qquad+\int_x^\infty\mathscr Y_0(y)^{-1}A_1(y)\partial_\epsilon Y_8^+(y,\lambda_0,0)\,\d y\\
		&\qquad\qquad+\int_{-\infty}^\infty\mathscr Y_0(y)^{-1}A_2(y)Y_{04}^-(y,\lambda_0)\,\d y
	\biggr).
\end{align*}
Substituting the above expression into \eqref{eqn:f_def}, we obtain
\begin{align*}
	f(x)=2\tau E_\mathrm{A}(\lambda_0)E_\mathrm{B}(\lambda_0)(f_1(x)+f_2(x)+f_3)
\end{align*}
after some lengthy manipulation, where
\begin{align*}
	&f_1(x)=\int_{-\infty}^x\int_{-\infty}^ya(y)p_8(y)a(y')p_8(y')\nonumber\\
	&\qquad\qquad\qquad\sum_{j=1}^2\frac{1}{C_j}\Bigl(\bigl(p_{j+4}(y)+q_{j+4}(y)\bigr)\bigl(p_j(y')+q_j(y')\bigr)\nonumber\\
	&\qquad\qquad\qquad\qquad\quad-\bigl(p_j(y)+q_j(y)\bigr)\bigl(p_{j+4}(y')+q_{j+4}(y')\bigr)\Bigr)\d y'\d y,\nonumber\\
	&f_2(x)=\int_x^\infty\int_y^\infty a(y)p_8(y)a(y')p_8(y')\nonumber\\
	&\qquad\qquad\qquad\sum_{j=1}^2\frac{1}{C_j}\Bigl(\bigl(p_j(y)+q_j(y)\bigr)\bigl(p_{j+4}(y')+q_{j+4}(y')\bigr)\nonumber\\
	&\qquad\qquad\qquad\qquad\quad-\bigl(p_{j+4}(y)+q_{j+4}(y)\bigr)\bigl(p_j(y')+q_j(y')\bigr)\Bigr)\d y'\d y,\nonumber
\end{align*}
at $\lambda=\lambda_0$.
Here $C_1$ and $C_2$ are given in Lemma \ref{lem:Jost_normalize} and $f_3$ is given by \eqref{eqn:f3_def}.
Finally, we take the limit $x\to\infty$ in $\partial_\epsilon^2E(\lambda_0,0)=f(x)+2g(x)$ and evaluate
\begin{align*}
	&f_1(x)-\sum_{j=1}^2\rho_j\int_{-\infty}^x a(y)p_8(y)\bigl(p_j(y)+q_j(y)\bigr)\d y\to\frac{2J_1}{C_1}+\frac{2J_2}{C_2},\\
	&f_2(x)+\sum_{j=1}^2\rho_{j+4}\int_x^\infty a(y)p_8(y)\bigl(p_{j+4}(y)+q_{j+4}(y)\bigr)\d y\to0
\end{align*}
as $x\to\infty$, which yields \eqref{eqn:de2E_mid}.\qed





\section{Further Computations of $\bar b_2$}

\subsection*{Case of $\ell=0$}

We take $\ell=0$ in \eqref{eqn:V1_cubic} and \eqref{eqn:b2_cubic} to obtain
\begin{equation}
	V_1^{(0)}(x)=\sech^{\sqrt{s}}x
	\label{eqn:l0}
\end{equation}
and
\begin{align*}
	\bar b_2=\frac{\sqrt\pi\,\Gamma(2\sqrt s)}{\Gamma\bigl(2\sqrt s+\frac{1}{2}\bigr)}(s^\frac{3}{2}-\beta_2),
\end{align*}
respectively.
It follows from Theorem \ref{thm:cnls_bif} that a supercritical (resp.\ subcritical) bifurcation of homoclinic orbits occurs at $\beta_1=\beta_1^{(0)}$
 if $\beta_2<$  (resp.\ $>$) $s^\frac{3}{2}$.
Since $V_1^{(0)}(x)$ has 
 no zero, so does the $V$-component of the bifurcating homoclinic solution \eqref{eqn:bif_UV}
 near $\beta_1=\beta_1^{(0)}$.

\subsection*{Case of $\ell=1$}

We take $\ell=1$ in \eqref{eqn:V1_cubic} and \eqref{eqn:b2_cubic} to obtain
\begin{equation*}
	V_1^{(1)}(x)=\sech^{\sqrt{s}}x\,\tanh x
\end{equation*}
and
\begin{align*}
	\bar b_2=\frac{\sqrt\pi\,\Gamma(2\sqrt s)}{\Gamma\bigl(2\sqrt s+\frac{1}{2}\bigr)}\frac{\sqrt s(7s-4)-3\beta_2}{(4\sqrt s+1)(4\sqrt s+3)},
\end{align*}
respectively.
By Theorem \ref{thm:cnls_bif},
 a supercritical (resp.\ subcritical) bifurcation of homoclinic orbits occurs at $\beta_1=\beta_1^{(1)}$ if $\beta_2<$ (resp.\ $>$) $\sqrt s(7s-4)/3$.
Since $V_1^{(1)}(x)$ has 
 only one zero,
 so does the $V$-component of the bifurcating homoclinic solution \eqref{eqn:bif_UV}
 near $\beta_1=\beta_1^{(1)}$.


\subsection*{General case of $\ell$}

From Proposition~\ref{prop:a2b2_cubic} we see that $\bar b_2$ has the form
\begin{align*}
	\bar b_2=\frac{\sqrt\pi\,\Gamma(2\sqrt s)}{\Gamma\bigl(2\sqrt s+\frac{1}{2}\bigr)}\frac{Q_1^{(\ell)}(\sqrt s)-Q_2^{(\ell)}(\sqrt s)\beta_2}{Q_3^{(\ell)}(\sqrt s)},
\end{align*}
where $Q_j^{(\ell)}(x)$, $j=1,2$, are some polynomials of $x$ with integer coefficients and
\begin{align*}
	Q_3^{(\ell)}(x)=\prod_{j=1}^{\lfloor\ell/2\rfloor}(x+j)^3\prod_{k=1}^{2\ell}(4x+2k-1).
\end{align*}
By Theorem \ref{thm:cnls_bif}, a supercritical (resp.\ subcritical) bifurcation of homoclinic orbits occurs at $\beta_1=\beta_1^{(\ell)}$ if $\beta_2<$ (resp.\ $>$) $Q_1^{(\ell)}(\sqrt s)/Q_2^{(\ell)}(\sqrt s)$.
Expressions of $Q_1^{(\ell)}(x)$ and $Q_2^{(\ell)}(x)$ for $\ell=2,3,4$ are given as follows:
\begin{align*}
	&Q_1^{(2)}(x)=x(145x^5+571x^4+487x^3-663x^2-1260x-540),\\
	&Q_2^{(2)}(x)=41x^3+167x^2+227x+105,\\
	&Q_1^{(3)}(x)=27x(229x^5+1343x^4+1655x^3-3271x^2-8336x-4560),\\
	&Q_2^{(3)}(x)=27(49x^3+303x^2+603x+385),\\
	&Q_1^{(4)}(x)=27x(16627x^8+245867x^7+1380611x^6+3183693x^5-660138x^4\\
		&\qquad\qquad-19430100x^3-41883400x^2-38216000x-13040000),\\
	&Q_2^{(4)}(x)=81(961x^6+14721x^5+92393x^4+303879x^3+551546x^2\\
		&\qquad\qquad+522180x+200200).
\end{align*}
Graphs of $\beta_2=Q_1^{(\ell)}(\sqrt s)/Q_2^{(\ell)}(\sqrt s)$ for $\ell\le 4$ are plotted in Fig.~\ref{fig:bifbdry}.
Especially, we observe that the bifurcations detected by Theorem \ref{thm:bifex} are supercritical for $\ell\le4$
 if $s>0$ is sufficiently large for $\beta_2$ fixed.

\begin{figure}
	\begin{center}
		\includegraphics[scale=0.48]{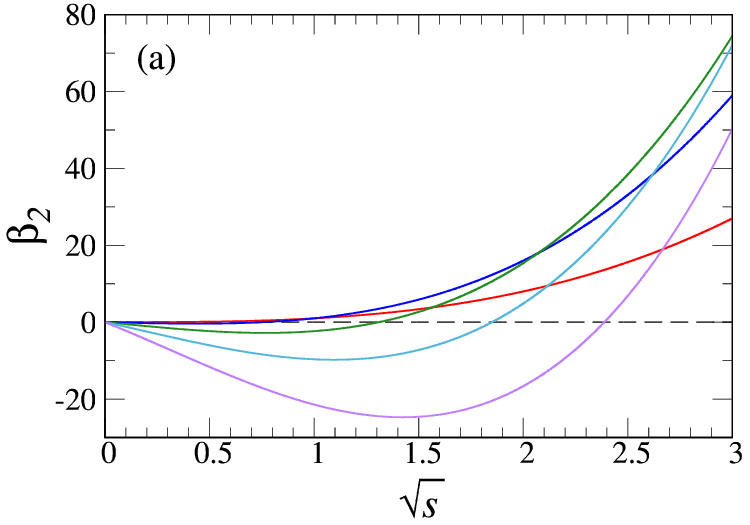}
		\includegraphics[scale=0.48]{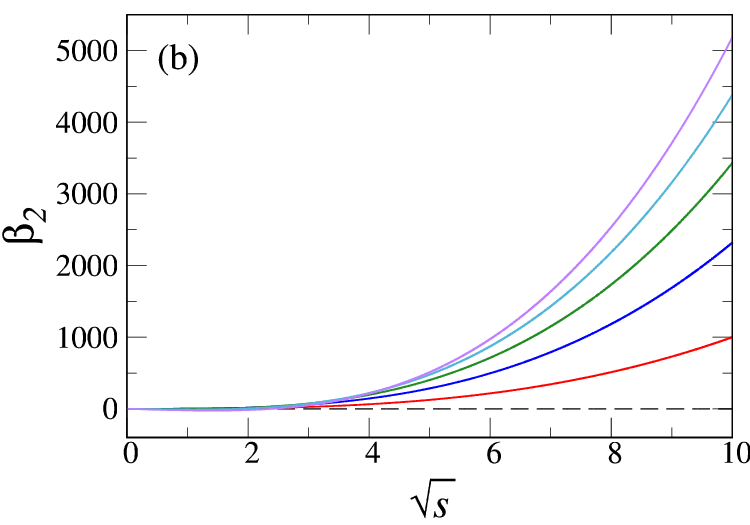}
	\end{center}
	\caption{
		Graphs of $\beta_2=Q_1^{(\ell)}(\sqrt s)/Q_2^{(\ell)}(\sqrt s)$ for $\ell=0$ (red), $1$ (blue), $2$ (green), $3$ (light blue), and $4$ (purple):
		(a) $0<\sqrt s<3$; (b) $0<\sqrt s<10$.
		For each $\ell\le 4$, a supercritical (resp.\ subcritical) bifurcation occurs
		at $\beta_1=\beta_1^{(\ell)}$ if $(\sqrt s,\beta_2)$ is located below (resp.\ above) the curve.
		\label{fig:bifbdry}
	 }
\end{figure}


\section{Proofs of Lemmas \ref{lem:E_simple} and \ref{lem:concrete_comp}}
\renewcommand{\theequation}{C.\arabic{equation}}
\setcounter{equation}{0}


\subsection{Proof of Lemma \ref{lem:E_simple}}\label{sec:proof_E_simple}

We first prove part (ii).
Let $k\in\{1,2,\ldots,\ell\}$.
The condition $\lambda_0(k)=-\lambda_0(k')$ is written as
\begin{align}
	-\i k(2\sqrt s+k)=\i k'(2\sqrt s+k')
	\label{eqn:simple_E2_E3}
\end{align}
for $k'\in\{-\lfloor\sqrt s\rfloor,-\lfloor\sqrt s\rfloor+1,\ldots,\ell\}$.
Let
\begin{align*}
	G(k,k')\defeq-\frac{k^2+(k')^2}{2(k+k')}.
\end{align*}
As easily shown, there exists $k'\in\{-\lfloor\sqrt s\rfloor,-\lfloor\sqrt s\rfloor+1,\ldots,\ell\}$ satisfying \eqref{eqn:simple_E2_E3} if and only if $-\lfloor G(k,k')\rfloor\le k'\le\ell$, $k+k'<0$, and $\sqrt s=G(k,k')$.
Using the fact that $-\lfloor G(k,k')\rfloor\le k'\iff k'\ge-\lfloor(\sqrt 2+1)k\rfloor$,
we see that $\lfloor(\sqrt 2+1)k\rfloor\le k'<-k$.
This implies part (ii).
Part (i) is also similarly proved.

We turn to parts (iii) and (iv).
For part (iii), \eqref{eqn:simple_E2_E3} is replaced with $-\i k(2\sqrt s+k)=\i$ for $k\in\{-\lfloor\sqrt s\rfloor,\ldots,\ell\}$ and $\sqrt s>0$.
This condition holds if and only if $(k^2+1)/(2k)\le k<0$ and $\sqrt s=-(k^2+1)/(2k)$, which is equivalent to $s=1$ and $k=-1$.
For part (iv), we notice that $\lambda_0(k)=-\i k(2\sqrt s+k)=-\i$ for $k\in\{-\lfloor\sqrt s\rfloor,\ldots,\ell\}$ if and only if $(k^2-1)/(2k)\le k\le\ell$, $-(k^2-1)/(2k)>0$, and $\sqrt s=-(k^2-1)/(2k)$.
One of the conditions does not occur at least.\qed


\subsection{Proof of Lemma \ref{lem:concrete_comp}}\label{sec:proof_comp}

We begin with part (i) and write $\lambda_0=-\i kh$ and $\nu_2(\lambda_0)=\i\nu$, where $h=2\sqrt s+k$ and $\nu=\sqrt{-kh-1}\ge0$.

We first assume that $k<0$ is even.
Using \eqref{eqn:p4_cubic_even} 
 and \eqref{eqn:l0}, we compute \eqref{eqn:Iell_def} as
\begin{align}
	&\mathcal I^{(0)}(\nu_4(\lambda_0),\nu_2(\lambda_0))\nonumber\\
	&=\int_{-\infty}^\infty\e^{\i\nu x}\bigl((\i\nu-\tanh x)^2-\sech^2x\bigr)\sech^{h+1}x\,\pFq{2}{1}{\frac{k}{2},\frac{h+1}{2}}{\frac{h+k}{2}+1}{\sech^2x}\d x\nonumber\\
	&=\sum_{j=0}^{-k/2}\frac{\bigl(\frac{k}{2}\bigr)_j\bigl(\frac{h+1}{2}\bigr)_j}{j!\,\bigl(\frac{h+k}{2}+1\bigr)_j}\int_{-\infty}^\infty\e^{-\i\nu x}\bigl((\i\nu+\tanh x)^2-\sech^2x\bigr)\sech^{h+2j+1}x\,\d x.\label{eqn:J_cubic_mid}
\end{align}
Using integration by parts along with \eqref{eqn:Kr_Fourier}, we have
\begin{align}
	&\mathcal K_{r+2}(\alpha)
	=2^{r+1}\frac{\Gamma(\frac{r+2+\i\alpha}{2})\Gamma(\frac{r+2-\i\alpha}{2})}{\Gamma(r+2)}\notag\\
	&=2^{r-1}\frac{(r^2+\alpha^2)\Gamma(\frac{r+\i\alpha}{2})\Gamma(\frac{r-\i\alpha}{2})}{r(r+1)\Gamma(r)}
	=\frac{r^2+\alpha^2}{r(r+1)}\mathcal K_r(\alpha)\label{eqn:Kr_rec}
\end{align}
and
\begin{align*}
	\int_{-\infty}^\infty\e^{-\i\alpha x}\sech^rx\tanh x\,\d x
	=-\frac{\i\alpha}{r}\mathcal K_r(\alpha).
\end{align*}
So the integral in \eqref{eqn:J_cubic_mid} is evaluated as
\begin{align*}
	&(1-\nu^2)K_{h+2j+1}(\nu)-2K_{h+2j+3}(\nu)+\frac{2\nu^2}{h+2j+1}\mathcal K_{h+2j+1}(\nu)\\
	&=\biggl((1-\nu^2)-2\frac{(h+2j+1)^2+\nu^2}{(h+2j+1)(h+2j+2)}+\frac{2\nu^2}{h+2j+1}\biggr)\mathcal K_{h+2j+1}(\nu)\\
	&=kh\biggl(1-\frac{2}{h+2j+2}\biggr)\mathcal K_{h+2j+1}(\nu).
\end{align*}
Using \eqref{eqn:Kr_rec} in the above relation repeatedly, we obtain
\begin{align}
	\mathcal I^{(0)}(\nu_4(\lambda_0),\nu_2(\lambda_0))
	&=kh\sum_{j=0}^{-k/2}\frac{\bigl(\frac{k}{2}\bigr)_j\bigl(\frac{h+1}{2}\bigr)_j}{j!\,\bigl(\frac{h+k}{2}+1\bigr)_j}\biggl(1-\frac{2}{h+2j+2}\biggr)\nonumber\\
	&\qquad\quad\times\biggl(\prod_{m=0}^{j-1}\frac{(h+2m+1)^2+\nu^2}{(h+2m+1)(h+2m+2)}\biggr)\mathcal K_{h+1}(\nu)\nonumber\\
	&=kh\biggl(\sum_{j=0}^{-k/2}a_j-2\sum_{j=0}^{-k/2}b_j\biggr)\mathcal K_{h+1}(\nu),\label{eqn:I0_tmp}
\end{align}
where
\begin{align*}
	a_j=\frac{\bigl(\frac{k}{2}\bigr)_j\bigl(\frac{h+1}{2}\bigr)_j}{j!\,\bigl(\frac{h+k}{2}+1\bigr)_j}\prod_{m=0}^{j-1}\frac{(h+2m+1)^2+\nu^2}{(h+2m+1)(h+2m+2)},\quad
	b_j=\frac{a_j}{h+2j+2}.
\end{align*}
Since
\begin{align*}
	&\frac{a_{j+1}}{a_j}
	=\frac{\bigl(\frac{k}{2}+j\bigr)\bigl(\frac{h+1+\i\nu}{2}+j\bigr)\bigl(\frac{h+1-\i\nu}{2}+j\bigr)}{(j+1)\bigl(\frac{h+k}{2}+1+j\bigr)\bigl(\frac{h}{2}+1+j\bigr)},\\
	&\frac{b_{j+1}}{b_j}
	=\frac{\bigl(\frac{k}{2}+j\bigr)\bigl(\frac{h+1+\i\nu}{2}+j\bigr)\bigl(\frac{h+1-\i\nu}{2}+j\bigr)}{(j+1)\bigl(\frac{h+k}{2}+1+j\bigr)\bigl(\frac{h}{2}+2+j\bigr)},
\end{align*}
$a_0=1$, and $b_0=1/(h+2)$, we have
\begin{align*}
	&\sum_{j=0}^{-k/2}a_j
	=\pFq{3}{2}{\frac{k}{2},\frac{h+1+\i\nu}{2},\frac{h+1-\i\nu}{2}}{\frac{h+k}{2}+1,\frac{h}{2}+1}{1},\ 
	\sum_{j=0}^{-k/2}b_j
	=\frac{1}{h+2}\pFq{3}{2}{\frac{k}{2},\frac{h+1+\i\nu}{2},\frac{h+1-\i\nu}{2}}{\frac{h+k}{2}+1,\frac{h}{2}+2}{1}
\end{align*}
by \eqref{eqn:def_hypgeo}.
From a contiguous relation of ${}_3F_2$,
\begin{align*}
	d\,\pFq{3}{2}{a,b,c}{d,e}{z}-a\,\pFq{3}{2}{a+1,b,c}{d+1,e}{z}+(a-d)\pFq{3}{2}{a,b,c}{d+1,e}{z}=0
\end{align*}
(see (15) of \cite{Ra45}), we have
\begin{align*}
	\pFq{3}{2}{-n,b,c}{d+1,e}{z}
	=\frac{1}{d+n}\biggl(d\,\pFq{3}{2}{-n,b,c}{d,e}{z}+n\,\pFq{3}{2}{-n+1,b,c}{d+1,e}{z}\biggr).
\end{align*}
Applying \eqref{eqn:Saal} to the right hand side, we obtain
\begin{align}
	\pFq{3}{2}{-n,b,c}{d+1,e}{1}=\frac{(d-b)_n(d-c)_n}{(d+1)_n(d-b-c)_n}\biggl(1+\frac{n(d-b-c)}{(d-b)(d-c)}\biggr)\label{eqn:3F2_excess2}
\end{align}
for $n\in\Zset_{\ge0}$ and $-n+b+c+1=d+e$, so that $\sum_{j=0}^{-k/2}b_j=0$.
Using \eqref{eqn:Saal} and \eqref{eqn:I0_tmp}, we obtain
\begin{align*}
	\mathcal I^{(0)}(\nu_4(\lambda_0),\nu_2(\lambda_0))
	&=kh\frac{\bigl(\frac{1+\i\nu}{2}\bigr)_{-k/2}\bigl(\frac{1-\i\nu}{2}\bigr)_{-k/2}}{\bigl(\frac{h}{2}+1\bigr)_{-k/2}\bigl(-\frac{h}{2}\bigr)_{-k/2}}\mathcal K_{h+1}(\nu)\\
	&=\frac{(-1)^{k/2}kh}{\bigl(\frac{h+k}{2}+1\bigr)_{-k}}\biggl(\prod_{j=0}^{-k/2-1}\biggl(j+\frac{1+\i\nu}{2}\biggr)\biggl(j+\frac{1-\i\nu}{2}\biggr)\biggr)\mathcal K_{h+1}(\nu),
\end{align*}
which yields part (i) when $k$ is even.

We next assume that $k<0$ is odd.
Using arguments similar to the above, we obtain
\begin{align}
	\mathcal I^{(0)}(\nu_4(\lambda_0),\nu_2(\lambda_0))
	&=\i\nu\sum_{j=0}^{-(k+1)/2}\frac{\bigl(\frac{k+1}{2}\bigr)_j\bigl(\frac{h}{2}+1\bigl)_j}{j!\,\bigl(\frac{h+k}{2}+1\bigr)_j}\nonumber\\
	&\qquad\qquad\biggl(\frac{-kh+2}{h+2j+1}-\frac{-2kh+6}{(h+2j+1)(h+2j+3)}\biggr)\mathcal K_{h+2j+1}(\nu)\nonumber\\
	&=\i\nu\biggl(\frac{-kh+2}{h+1}\pFq{3}{2}{\frac{k+1}{2},\frac{h+1+\i\nu}{2},\frac{h+1-\i\nu}{2}}{\frac{h+k}{2}+1,\frac{h+3}{2}}{1}\biggr.\nonumber\\
	&\qquad\quad\biggl.-\frac{-2kh+6}{(h+1)(h+3)}\pFq{3}{2}{\frac{k+1}{2},\frac{h+1+\i\nu}{2},\frac{h+1-\i\nu}{2}}{\frac{h+k}{2}+1,\frac{h+5}{2}}{1}\biggr)\mathcal K_{h+1}(\nu).\label{eqn:I0_tmp2}
\end{align}
Using \eqref{eqn:Saal} and \eqref{eqn:3F2_excess2} again, we compute
\begin{align*}
	&\pFq{3}{2}{\frac{k+1}{2},\frac{h+1+\i\nu}{2},\frac{h+1-\i\nu}{2}}{\frac{h+k}{2}+1,\frac{h+3}{2}}{1}
	=\frac{\bigl(1+\frac{\i\nu}{2}\bigr)_{-(k+1)/2}\bigl(1-\frac{\i\nu}{2}\bigr)_{-(k+1)/2}}{\bigl(\frac{h+3}{2}\bigr)_{-(k+1)/2}\bigl(\frac{-h+1}{2}\bigr)_{-(k+1)/2}},\\
	&\pFq{3}{2}{\frac{k+1}{2},\frac{h+1+\i\nu}{2},\frac{h+1-\i\nu}{2}}{\frac{h+k}{2}+1,\frac{h+5}{2}}{1}
	=\frac{\bigl(1+\frac{\i\nu}{2}\bigr)_{-(k+1)/2}\bigl(1-\frac{\i\nu}{2}\bigr)_{-(k+1)/2}}{\bigl(\frac{h+5}{2}\bigr)_{-(k+1)/2}\bigl(\frac{-h+1}{2}\bigr)_{-(k+1)/2}}\biggl(1-\frac{\bigl(\frac{k+1}{2}\bigr)\bigl(\frac{-h+1}{2}\bigr)}{\bigl(1+\frac{\i\nu}{2}\bigr)\bigl(1-\frac{\i\nu}{2}\bigr)}\biggr).
\end{align*}
Substituting these expressions into \eqref{eqn:I0_tmp2}, we obtain part (i) when $k$ is odd.
Part (ii) is proved in the same way as in part (i).

We finally prove part (iii).
Let $\lambda_0=\i$ and $\nu=\sqrt{1-s}>0$.
So $\nu_2(\i)=0$ and $\nu_4(\i)=\i\nu$.
\eqref{eqn:Iell_def} is written as
\begin{align*}
	\mathcal I^{(0)}(\nu_4(\i),\nu_2(\i))
	=\int_{-\infty}^\infty&\sech^{\sqrt s+1+\i\nu}x\,(1-2\sech^2x)\\
	&\pFq{2}{1}{-\sqrt s+\i\nu,\sqrt s+1+\i\nu}{1+\i\nu}{\frac{1+\tanh x}{2}}\d x.
\end{align*}
Changing the variable as $t=(1+\tanh x)/2$ and using \eqref{eqn:int_pFq}, we have
\begin{align}
	\mathcal I^{(0)}(\nu_4(\i),\nu_2(\i))
	=&\frac{2^{\sqrt s+\i\nu}\Gamma(\frac{\sqrt s+1+\i\nu}{2})^2}{\Gamma(\sqrt s+1+\i\nu)}\biggl(\pFq{2}{1}{-\sqrt s+\i\nu,\frac{\sqrt s+1+\i\nu}{2}}{1+\i\nu}{1}\biggr.\nonumber\\
	&\biggl.-2\frac{\sqrt s+1+\i\nu}{\sqrt s+2+\i\nu}\pFq{3}{2}{-\sqrt s+\i\nu,\sqrt s+1+\i\nu,\frac{\sqrt s+3+\i\nu}{2}}{1+\i\nu,\sqrt s+3+\i\nu}{1}\biggr).
	\label{eqn:comp_tmp1}
\end{align}
By Gauss' summation theorem and Watson's summation theorem (see, e.g., (1.7.6) and (2.3.3.13) of \cite{Sl66}), the first and second terms are computed as
\begin{align}
	\pFq{2}{1}{-\sqrt s+\i\nu,\frac{\sqrt s+1+\i\nu}{2}}{1+\i\nu}{1}
	=\frac{\Gamma(1+\i\nu)\Gamma\bigl(\frac{\sqrt s+1-\i\nu}{2}\bigr)}{\Gamma(\sqrt s+1)\Gamma\bigl(\frac{-\sqrt s+1+\i\nu}{2}\bigr)}
	\label{eqn:comp_tmp2}
\end{align}
and
\begin{align}
	\pFq{3}{2}{-\sqrt s+\i\nu,\sqrt s+1+\i\nu,\frac{\sqrt s+3+\i\nu}{2}}{1+\i\nu,\sqrt s+3+\i\nu}{1}
	=\frac{(\sqrt s+2+\i\nu)\Gamma(1+\i\nu)\Gamma\bigl(\frac{\sqrt s+3-\i\nu}{2}\bigr)}{\Gamma(\sqrt s+2)\Gamma\bigl(\frac{-\sqrt s+1+\i\nu}{2}\bigr)},
	\label{eqn:comp_tmp3}
\end{align}
respectively.
Substituting \eqref{eqn:comp_tmp2} and \eqref{eqn:comp_tmp3} into \eqref{eqn:comp_tmp1}, we obtain the desired result.\qed


\begin{thebibliography}{99}
\bibitem{AbSt64}
	M.\ Abramowitz and I.\ Stegun,
	\textit{Handbook of Mathematical Functions with Formulas, Graphs and Mathematical Tables},
	Dover Publications, Inc., New York, 1965.
\bibitem{Ag95}
	G.\ Agrawal,
	\textit{Nonlinear Fiber Optics},
	Academic Press, New York, 6th edition, 2019.
\bibitem{AGJ90}
	J.\ Alexander, R.\ Gardner, and C.\ Jones,
	A topological invariant arising in the stability analysis of travelling waves,
	\textit{J.\ reine angew.\ Math.}, \textbf{410} (1990) 167--212.
\bibitem{BeLi83}
	H.\ Berestycki and P.-L.\ Lions,
	Nonlinear scalar field equations, I: Existence of a ground state,
	\textit{Arch.\ Rational Mech.\ Anal.}, \textbf{82} (1983) 313--345.
\bibitem{Bh15}
	S.\ Bhattarai,
	Stability of solitary-wave solutions of coupled NLS equations with power-type nonlinearites,
	\textit{Adv.\ Nonlinear Anal.}, \textbf{4} (2015) 73--90.
\bibitem{BlYa12}
	D.\ Bl\'azquez-Sanz and K.\ Yagasaki,
	Analytic and algebraic conditions for bifurcations of homoclinic orbits I: Saddle equilibria,
	\textit{J.\ Differential Equations}, \textbf{253} (2012) 2916--2950.
\bibitem{Ca96}
	T.\ Cazenave,
	\textit{An introduction to nonlinear Schr{\"o}dinger equations},
	Universidade Federal do Rio de Janeiro, Centro de Ci\^encias Matem\'aticas e da Natureza, Instituto de Matem\'atica, Rio de Janeiro, 3rd edition, 1996.
\bibitem{CiZu00}
	R.\ Cipolatti and W.\ Zumpichiatti,
	Orbitally stable standing waves for a system of coupled nonlinear Schr{\"o}dinger equations,
	\textit{Nonlinear Anal.}, \textbf{42} (2000) 445--461.
\bibitem{CoLe55}
	E.\ Coddington and N.\ Levinson,
	\textit{Theory of Ordinary Differential Equations},
	McGraw-Hill, New York, 1955.
\bibitem{CPV05}
	S.\ Cuccagna, D.\ Pelinovsky, and V.\ Vougalter,
	Spectra of positive and negative energies in the linearized NLS problem,
	\textit{Comm.\ Pure Appl.\ Math.}, \textbf{58} (2009) 1--29.
\bibitem{DO12}
	E.\ Doedel and B.\ Oldeman,
	AUTO-07P: Continuation and Bifurcation Software for Ordinary Differential Equations,
	2012, available online from \texttt{http://indy.cs.concordia.ca/auto}.
\bibitem{Ev72a}
	J.\ Evans,
	Nerve axon equations, I: Linear approximations,
	\textit{Indiana U.\ Math.\ J.}, \textbf{21} (1972) 877--955.
\bibitem{Ev72b}
	J.\ Evans,
	Nerve axon equations, II: Stability at rest,
	\textit{Indiana U.\ Math.\ J.}, \textbf{22} (1972) 75--90.
\bibitem{Ev72c}
	J.\ Evans,
	Nerve axon equations, III: Stability of the nerve impulse,
	\textit{Indiana U.\ Math.\ J.}, \textbf{22} (1972) 577--594.
\bibitem{Ev75}
	J.\ Evans,
	Nerve axon equations, IV: The stable and unstable impulse,
	\textit{Indiana U.\ Math.\ J.}, \textbf{24} (1975) 1169--1190.
\bibitem{GSS90}
	M.\ Grillakis and J.\ Shatah and W.\ Strauss,
	Stability theory of solitary waves in the presence of symmetry. II,
	\textit{J.\ Funct.\ Anal.}, \textbf{94} (1990) 308--348.
\bibitem{Gu92}
	J.\ Gruendler,
	Homoclinic solutions for autonomous dynamical systems in arbitrary dimension,
	\textit{SIAM J.\ Math.\ Anal.}, \textbf{23} (1992) 702--721.
\bibitem{KKS04}
	T.\ Kapitula and P.\ Kevrekidis and B.\ Sandstede,
	Counting eigenvalues via the Krein signature in infinite-dimensional Hamiltonian systems,
	\textit{Phys.\ D}, \textbf{195} (2004) 263--282.
\bibitem{KKS05}
	T.\ Kapitula and P.\ Kevrekidis and B.\ Sandstede,
	Addendum: Counting eigenvalues via the Krein signature in infinite-dimensional Hamiltonian systems,
	\textit{Phys.\ D}, \textbf{201} (2005) 199--201.
\bibitem{KaPr13}
	T.\ Kapitula and K.\ Promislow,
	\textit{Spectral and Dynamical Stability of Nonlinear Waves},
	Springer, New York, 2013.
\bibitem{Ka90}
	D.\ Kaup,
	Perturbation theory for solitons in optical fibers,
	\textit{Phys.\ Rev.\ A}, \textbf{42} (1990) 5689--5694.
\bibitem{Ku04}
	Y.\ Kuznetsov,
	\textit{Elements of Applied Bifurcation Theory},
	Springer, New York, 3rd edition, 2004.
\bibitem{LaLi77}
	L.\ Landau and E.\ Lifshitz,
	\textit{Quantum Mechhanics: Non-relativistic Theory},
	Pergamon Press, Oxford, 1977.
\bibitem{LiPr00}
	Y.\ Li and K.\ Promislow,
	The mechanism of the polarizational mode instability in birefringent fiber optics,
	\textit{SIAM J.\ Math.\ Anal.}, \textbf{31} (2000) 1351--1373.
\bibitem{Ng11}
	N.\ Nguyen,
	On the orbital stability of solitary waves for the 2-coupled nonlinear Schr{\"o}dinger system,
	\textit{Commun.\ Math.\ Sci.}, \textbf{9} (2011) 997--1012.
\bibitem{NgWa11}
	N.\ Nguyen and Z.-Q.\ Wang,
	Orbital stability of solitary waves for a nonlinear Schr{\"o}dinger system,
	\textit{Adv.\ Differential Equations}, \textbf{16} (2011) 977--1000.
\bibitem{NgWa16}
	N.\ Nguyen and Z.-Q.\ Wang,
	Existence and stability of a two-parameter family of solitary waves for a 2-coupled nonlinear Schr{\"o}dinger system,
	\textit{Discrete Contin.\ Dyn.\ Syst.}, \textbf{36} (2016) 1005--1021.
\bibitem{Oh96}
	M.\ Ohta,
	Stability of solitary waves for coupled nonlinear Schr{\"o}dinger equations,
	\textit{Nonlinear Anal.}, \textbf{26} (1996) 933--939.
\bibitem{PeWe92}
	R.\ Pego and M.\ Weinstein,
	Eigenvalues, and instabilities of solitary waves,
	\textit{Phil.\ Trans.\ R.\ Soc.\ Lond.\ A}, \textbf{340} (1992) 47--94.
\bibitem{Pe05}
	D.\ Pelinovsky,
	Inertia law for spectral stability of solitary waves in coupled nonlinear Schr{\"o}dinger equations,
	\textit{Proc.\ R.\ Soc.\ A}, \textbf{461} (2005) 783--812.
\bibitem{PeYa05}
	D.\ Pelinovsky and J.\ Yang,
	Instabilities of multihump vector solitons in coupled nonlinear Schr{\"o}dinger equations,
	\textit{Stud.\ Appl.\ Math.}, \textbf{115} (2005) 109--137.
\bibitem{PuSi03}
	M.\ van der Put and M.\ Singer,
	\textit{Galois Theory of Linear Differential Equations},
	Springer, New York, 2003.
\bibitem{Ra45}
	E.\ Rainville,
	The contiguous function relations for ${}_pF_q$ with applications to Bateman's $J_n^{u,v}$ and Rice's $H_n(\zeta,p,v)$,
	\textit{Bull.\ Amer.\ Math.\ Soc.}, \textbf{51} (1945) 714--723.
\bibitem{Ro76}
	G.\ Roskes,
	Some nonlinear multiphase interactions,
	\textit{Stud.\ Appl.\ Math.}, \textbf{55} (1976) 231--238.
\bibitem{Sl66}
	L.\ Slater,
	\textit{Generalized Hypergeomtetric Functions},
	Cambridge University Press, Cambridge, 1966.
\bibitem{Ya20}
	K.\ Yagasaki,
	Analytic and algebraic conditions for bifurcations of homoclinic orbits II: Reversible systems,
	submitted for publication.
\bibitem{YaSt20}
	K.\ Yagasaki and T.\ Stachowiak,
	Bifurcations of radially symmetric solutions in a coupled elliptic system with critical growth in $\mathbb{R}^d$ for $d=3,4$,
	\textit{J.\ Math.\ Anal.\ Appl.}, \textbf{484} (2020) 123726.
\bibitem{YaYa20}
	K.\ Yagasaki and S.\ Yamazoe,
	Numerical computations for bifurcations and spectral stability of solitary waves in coupled nonlinear Schr{\"o}dinger equations,
	submitted for publication.
\bibitem{Ya10}
	J.\ Yang,
	\textit{Nonlinear Waves in Integrable and Nonintegrable Systems},
	SIAM, Philadelphia, PA, 2010.
\end{thebibliography}


\end{document}